\documentclass[a4paper,reqno]{amsart}

\usepackage{enumitem}
\usepackage{dsfont}
\usepackage{chemarrow}
\usepackage{comment}
\usepackage{subfiles}
\usepackage{upgreek}
\usepackage{amsfonts}
\usepackage{amsmath,amsthm,amssymb,colonequals}
\usepackage{indentfirst}
\usepackage{amscd}
\usepackage{setspace}
\usepackage{mathrsfs}
\usepackage{float}
\usepackage{graphicx, subcaption}
\usepackage{wrapfig}
\usepackage{enumitem}
\usepackage{euscript}
\usepackage{stackrel}

\setlist[enumerate]{format=\normalfont}

\usepackage{array,multirow,booktabs,longtable}
\usepackage{centernot}
\usepackage{mathtools}
\usepackage{stmaryrd}

\makeatletter
\newcommand{\xMapsto}[2][]{\ext@arrow 0599{\Mapstofill@}{#1}{#2}}
\def\Mapstofill@{\arrowfill@{\Mapstochar\Relbar}\Relbar\Rightarrow}
\makeatother
\usepackage{tikz,mathrsfs}
\usetikzlibrary{arrows,decorations.pathmorphing,decorations.pathreplacing,positioning,shapes.geometric,shapes.misc,decorations.markings,decorations.fractals,calc,patterns}

\tikzset{>=stealth',
        cvertex/.style={circle,draw=black,inner sep=1pt,outer sep=3pt},
        vertex/.style={circle,fill=black,inner sep=1pt,outer sep=3pt},
        star/.style={circle,fill=yellow,inner sep=0.75pt,outer sep=0.75pt},
        tvertex/.style={inner sep=1pt,font=\scriptsize},
        gap/.style={inner sep=0.5pt,fill=white}}

\newtheorem{theorem}{Theorem}[section]
\newtheorem{prop}[theorem]{Proposition}
\newtheorem{lemma}[theorem]{Lemma}
\newtheorem{definition}[theorem]{Definition}
\newtheorem{cor}[theorem]{Corollary}

\theoremstyle{definition}

\newtheorem{example}[theorem]{Example}

\newtheorem{remark}[theorem]{Remark}

\newtheorem{notation}[theorem]{Notation}

\numberwithin{equation}{section}

\usepackage[colorlinks]{hyperref}

\newcommand{\Jac}{\scrJ\mathrm{ac}}

\newcommand\N{\mathbb{N}}
\newcommand\Z{\mathbb{Z}}
\newcommand\Q{\mathbb{Q}}

\newcommand\C{\mathbb{C}}
\newcommand{\GV}{\textnormal{GV}}

\DeclareMathOperator{\Supp}{\mathrm{Supp}}
\renewcommand{\mod}{\mathop{\mathrm{mod}}}
\newcommand{\Spec}{\mathop{\mathrm{Spec}}}
\DeclareMathOperator{\Ext}{\mathrm{Ext}}

\renewcommand{\Im}{\mathop{\mathrm{Im}}}
\DeclareMathOperator{\Hom}{\mathrm{Hom}}
\DeclareMathOperator{\End}{\mathrm{End}}
\DeclareMathOperator{\Sing}{\mathrm{Sing}}
\DeclareMathOperator{\add}{\mathrm{add}}

\DeclareMathOperator{\refl}{\mathrm{ref}}

\renewcommand{\dim}{\mathop{\mathrm{dim}}}
\DeclareMathOperator{\depth}{\mathrm{depth}}

\renewcommand{\min}{\mathop{\mathrm{min}}}

\newcommand{\Curve}{\mathrm{C}}

\newcommand{\scrA}{\EuScript{A}}

\newcommand{\scrJ}{\EuScript{J}}

\newcommand{\scrN}{\EuScript{N}}
\newcommand{\scrO}{\EuScript{O}}

\newcommand{\scrR}{\EuScript{R}}
\newcommand{\scrS}{\EuScript{S}}
\newcommand{\scrT}{\EuScript{T}}
\newcommand{\scrU}{\EuScript{U}}

\newcommand{\scrX}{\EuScript{X}}
\newcommand{\scrY}{\EuScript{Y}}

\newcommand{\lcl}{(\kern -2.5pt(}
\newcommand{\rcl}{)\kern -2.5pt)}
\newcommand{\lal}{[\kern -1pt[}
\newcommand{\ral}{ ] \kern -1pt ]}
\newcommand{\lbl}{\langle \kern -2.5pt \langle}
\newcommand{\rbl}{\rangle \kern -2.5pt \rangle}
\renewcommand*\partial{\textnormal{\reflectbox{6}}}

\newcommand{\x}{\mathsf{x}}

\newcommand{\q}{\mathbf{q}}
\newcommand{\CM}{\mathrm{CM}}

\newcommand{\llcurve}{\{\kern -3pt\{}
\newcommand{\rrcurve}{\}\kern -3pt\}}
\newcommand{\m}{\mathfrak{m}}
\newcommand{\F}{\mathcal{F}}
\renewcommand{\c}[1]{\mathcal{#1}}

\newcommand{\M}{\mathsf{M}}
\newcommand{\n}{\mathsf{N}}
\newcommand{\MA}{\mathsf{MA}}
\newcommand{\p}{\mathbf{p}}
\newcommand{\uk}{\upkappa}
\newcommand{\con}{\mathrm{con}}

\begin{document}

\title{Gopakumar--Vafa invariants associated to $cA_n$ singularities}
\author{Hao Zhang}
\address{Hao Zhang, The Mathematics and Statistics Building, University of Glasgow, University Place, Glasgow, G12 8QQ, UK.}
\email{h.zhang.4@research.gla.ac.uk}
\begin{abstract}
This paper describes Gopakumar--Vafa (GV) invariants associated to $cA_n$ singularities. We (1) generalise GV invariants to crepant partial resolutions of $cA_n$ singularities, (2) show that generalised GV invariants also satisfy Toda's formula and are determined by their associated contraction algebra, (3) give filtration structures on the parameter space of contraction algebras associated to $cA_n$ crepant resolutions with respect to generalised GV invariants, and (4) numerically constrain the possible tuples of GV invariants that can arise. We further give all the tuples that arise from GV invariants of $cA_2$ crepant resolutions.
\end{abstract}
\maketitle

\maketitle
\setcounter{tocdepth}{1}
\tableofcontents

\parindent 0pt
\parskip 5pt
\section{Introduction}

Gopakumar-Vafa (GV) invariants are designed to count the number of pseudo-holomorphic curves and represent the number of BPS states on a Calabi-Yau 3-fold; it has been conjectured that this is equivalent to other curve counting Gromov-Witten invariants and Pandharipande--Thomas invariants \cite{MT}. A standard approach to computing GV invariants is via moduli spaces of one-dimensional stable sheaves subject to suitable numerical constraints \cite{K}. In practice, these moduli spaces are typically complicated, and explicit calculations are often difficult.

In the more restrictive setting of crepant (partial) resolutions of $cA_n$ singularities, two additional tools become available. First, Toda's formula \cite{T2} (see also \cite{HT,BW2}) relates GV invariants to the dimensions of the associated contraction algebra. Second, Iyama and Wemyss \cite{IW1} provide a concrete algebraic description of crepant partial resolutions of $cA_n$ singularities together with their contraction algebras. The smooth case is now particularly well understood: Zhang \cite{Z} gives an intrinsic algebraic definition of Type~$A$ potentials and proves that (monomialised) Type~$A$ potentials correspond precisely to crepant resolutions of $cA_n$ singularities.
 
The aim of this paper is to develop the curve-counting consequences of these results within algebraic geometry. Along the way, we generalise GV invariants in two directions: first to crepant \emph{partial} resolutions, and second to \emph{non-isolated} $cA_n$ singularities.

It is also worth emphasising that related curve-counting invariants have appeared in the physics literature. In particular, computations in \cite{CSV,C,DSV} evaluate M2-brane BPS state counts (corresponding to five-dimensional hypermultiplets) for various classes of cDV singularities. These include crepant resolutions with a single exceptional curve and crepant partial resolutions of quasi-homogeneous isolated cDV singularities. Although physically motivated, these calculations lead to precise mathematical predictions that, in the $cA_n$ setting, coincide with our generalised GV invariants. The framework developed here therefore provides a natural mathematical setting for these predictions: it extends Toda's formula to crepant partial resolutions and to non-isolated $cA_n$ singularities, while simultaneously giving an algebraic description in terms of contraction algebras. In this way, the results here not only recover predictions from physics in special cases but also place them within a broader, more systematic mathematical framework.

\subsection{Singular Invariants}\label{mainresult}
Throughout, let $\uppi\colon \scrX \to \Spec \scrR$ be a crepant \emph{partial} resolution, where $\scrR$ is a (not necessarily isolated) $cA_n$ singularity; see \ref{def:crepant}. When $\scrX$ is smooth (equivalently, when $\uppi$ is a crepant resolution), our constructions recover the classical invariants and results.

We first introduce our invariants $N_{\upbeta}(\uppi)$, defined without assuming smoothness of $\scrX$ or isolatedness of $\scrR$. Write $\Curve_1,\Curve_2,\dots,\Curve_m$ for the exceptional curves of $\uppi$. For any curve class $\upbeta\in \bigoplus_{i=1}^m \Z\langle \Curve_i\rangle$, set
\[
N_{\upbeta}(\uppi)\colonequals
\begin{cases}
\dim_\C \dfrac{\C\lal x,y\ral}{I_{\upbeta}} & \textnormal{if } \upbeta=\Curve_i+\Curve_{i+1}+\cdots+\Curve_j,\\[0.9em]
0 & \textnormal{otherwise},
\end{cases}
\]
where $I_{\upbeta}\subseteq (x,y)$ is an ideal determined by $\upbeta$ and $\uppi$ (see \ref{def:Nij}).

These invariants parallel the classical GV invariants: when $\uppi$ is a crepant resolution, the curve classes
\[
\{\Curve_i+\Curve_{i+1}+\cdots+\Curve_j \mid 1\leq i\leq j\leq m\}
\]
are precisely those with non-zero GV invariants \cite{NW1,VG}. Moreover, we prove in \ref{intro: 37} that in the smooth case the invariants $N_{\upbeta}$ encode the same information as the integer-valued GV invariants $\GV_{\upbeta}$ (see \ref{GV}): this justifies the terminology \emph{generalised GV invariants}.

The invariant $N_{\upbeta}$ also admits an interpretation as a local intersection number of plane curves in the complete local plane $\Spec \C\lal x,y\ral$, via the ideal $I_{\upbeta}$ (see \ref{rmk:intersect} and \ref{example:gv}). From this perspective, it is striking that an invariant designed to count curves virtually on a threefold can be expressed in terms of plane curve intersections.

Associated to $\uppi \colon \scrX \to \Spec\scrR$ is a noncommutative algebra $\Lambda_{\mathrm{con}}(\uppi)$, called the contraction algebra \cite{DW2}, which encodes the local geometry of the contraction. The following result shows that Toda’s formula \ref{34} extends to this more general, possibly singular, setting.

\begin{prop}[\ref{thm:Toda},  \ref{rmk:ijtobeta}]\label{intro:Toda}
Let $\uppi$ be a crepant partial resolution of a $cA_n$ singularity with $m$ exceptional curves.
For any $1 \leq s  \leq t \leq m$, the following equality holds.
\begin{equation*}
    \operatorname{dim}_{\C}e_s\Lambda_{\mathrm{con}}(\uppi)e_t= 
   \sum_{\upbeta=(\upbeta_1,\dots ,\upbeta_m)} \upbeta_s \cdot \upbeta_t \cdot N_{\upbeta}(\uppi)=\operatorname{dim}_{\C}e_t\Lambda_{\mathrm{con}}(\uppi)e_s.
\end{equation*}
In particular, $ \operatorname{dim}_{\C}\Lambda_{\mathrm{con}}(\uppi)=\sum_{\upbeta} |\upbeta|^2 N_{\upbeta}(\uppi)$ where $|\upbeta| = \upbeta_1 + \dots + \upbeta_m$.
\end{prop}

Hua--Toda \cite{HT,T2} show that, when $\scrX$ is smooth and $\scrR$ is isolated, the GV invariants are determined by the isomorphism class of the contraction algebra. The next result extends this to crepant partial resolutions of (not necessarily isolated) $cA_n$ singularities.

For ease of notation, given a curve class $\upbeta=(\upbeta_1,\dots,\upbeta_m)$, define its \emph{reflected class} by
$\Bar{\upbeta}\colonequals (\upbeta_m,\dots,\upbeta_1)$.
This symmetry arises from the involution of the doubled $A_n$ quiver reversing the orientation of the chain; since the contraction algebra is isomorphic to a quiver algebra on the doubled $A_n$ quiver, the reflected class corresponds to this involution.

\begin{theorem}[\ref{371}, \ref{rmk:ijtobeta}]\label{intro:deter}
Let $\uppi_k \colon \scrX_k \rightarrow \Spec \scrR_k$ be two crepant partial resolutions of $cA_{n_k}$ singularities $\scrR_k$ with $m_k$ exceptional curves for $k=1,2$. If $\Lambda_{\mathrm{con}}(\uppi_1) \cong \Lambda_{\mathrm{con}}(\uppi_2)$, then $m_1=m_2$ and one of the following cases holds:
\begin{enumerate}
     \item $N_{\upbeta}(\uppi_1)=N_{\upbeta}(\uppi_2)$ for every curve class $\upbeta$,
    \item $N_{\upbeta}(\uppi_1)=N_{\Bar{\upbeta}}(\uppi_2)$ for every curve class $\upbeta$.
\end{enumerate}
\end{theorem}

The papers \cite{NW1,VG} give a combinatorial description of the matrix governing the transformation of non-zero GV invariants under a flop (see \S\ref{sec: reduction} for the $cA_n$ case). We show in \ref{example:permutation} that the generalised GV invariants satisfy the same transformation rule.

\subsection{Restriction to the Smooth Case}
We now restrict to crepant resolutions of (not necessarily isolated) $cA_n$ singularities. Although the generalised GV invariants need not coincide with the classical GV invariants in every case, we prove that they carry equivalent information.
\begin{theorem}[\ref{37}, \ref{rmk:ijtobeta}]\label{intro: 37}
Let $\uppi$ be a crepant resolution of a $cA_n$ singularity. The following holds for any curve class $\upbeta$.
\begin{enumerate}
\item  $N_{\upbeta}(\uppi) = \infty \iff \GV_{\upbeta}(\uppi)=-1$.
\item  $N_{\upbeta}(\uppi) < \infty \iff \GV_{\upbeta}(\uppi)=N_{\upbeta}(\uppi)$.
\end{enumerate}
\end{theorem}

Together with \ref{intro:deter}, the following shows that the contraction algebra determines the associated GV invariants. This extends the results in \cite{HT, T2} to non-isolated $cA_n$ singularities.
\begin{cor}[\ref{371}, \ref{rmk:ijtobeta}]\label{intro:determin}
Let $\uppi_k \colon \scrX_k \rightarrow \Spec \scrR_k$ be two crepant resolutions of $cA_n$ singularities $\scrR_k$ for $k=1,2$. If $\Lambda_{\mathrm{con}}(\uppi_1) \cong \Lambda_{\mathrm{con}}(\uppi_2)$, then one of the following holds:
\begin{enumerate}
     \item $\GV_{\upbeta}(\uppi_1)=\GV_{\upbeta}(\uppi_2)$ for every curve class $\upbeta$,
    \item $\GV_{\upbeta}(\uppi_1)=\GV_{\Bar{\upbeta}}(\uppi_2)$ for every  curve class $\upbeta$.
\end{enumerate}
\end{cor}

\subsection{Filtration}
Continuing the assumption that $\scrX$ is smooth (equivalently, $\uppi$ is a crepant resolution), the contraction algebra $\Lambda_{\mathrm{con}}(\uppi)$ is isomorphic to the complete Jacobi algebra of a quiver with potential \cite{VM}. The possible potentials were explicitly described in \cite{Z}. This extra data motivates us to study the filtration structure of the parameter space of such potentials with respect to generalised GV invariants.




We briefly recall the relevant definitions from \cite{Z}.
Fix $n\ge 1$ and consider the quiver $Q_n$, the double of the usual $A_n$ quiver with a single loop at each vertex. Label the arrows of $Q_n$ from left to right as follows:
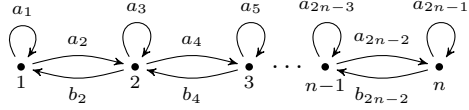
\[
\begin{array}{c}
\begin{tikzpicture}[bend angle=15, looseness=1.2]
\node (a) at (-1.5,0) [vertex] {};
\node (b) at (0,0) [vertex] {};
\node (c) at (1.5,0) [vertex] {};
\node (c2) at (2,0) {$\hdots$};
\node (d) at (2.5,0) [vertex] {};
\node (e) at (4,0) [vertex] {};
\node (a1) at (-1.5,-0.2) {$\scriptstyle 1$};
\node (a2) at (0,-0.2) {$\scriptstyle 2$};
\node (a3) at (1.5,-0.2) {$\scriptstyle 3$};
\node (a4) at (2.5,-0.25) {$\scriptstyle n-1$};
\node (a5) at (4,-0.25) {$\scriptstyle n$};
\draw[->,bend left] (a) to node[above] {$\scriptstyle a_{2}$} (b);
\draw[<-,bend right] (a) to node[below] {$\scriptstyle b_{2}$} (b);
\draw[->,bend left] (b) to node[above] {$\scriptstyle a_{4}$} (c);
\draw[<-,bend right] (b) to node[below] {$\scriptstyle b_{4}$} (c);
\draw[->,bend left] (d) to node[above] {$\scriptstyle a_{2n-2}$} (e);
\draw[<-,bend right] (d) to node[below] {$\scriptstyle b_{2n-2}$} (e);
\draw[<-]  (a) edge [in=120,out=55,loop,looseness=10] node[above] {$\scriptstyle a_{1}$} (a);
\draw[<-]  (b) edge [in=120,out=55,loop,looseness=11] node[above] {$\scriptstyle a_{3}$} (b);
\draw[<-]  (c) edge [in=120,out=55,loop,looseness=11] node[above] {$\scriptstyle a_{5}$} (c);
\draw[<-]  (d) edge [in=120,out=55,loop,looseness=11] node[above] {$\scriptstyle a_{2n-3}$} (d);
\draw[<-]  (e) edge [in=120,out=55,loop,looseness=11] node[above] {$\scriptstyle a_{2n-1}$} (e);
\end{tikzpicture}
\end{array}
\]

From this, define paths $\x_i$ and $\x_i'$ as follows. First, for each $1\leq i\leq n$, let $b_{2i-1}$ denote the trivial path $e_i$ at vertex $i$. Then, for $1\leq i\leq 2n-1$, set $\x_i \colonequals a_ib_i$ and $\x_i' \colonequals b_ia_i$.
For example, in the case $n=3$, 
 \[
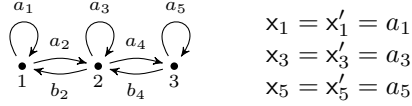
 \begin{array}{cl}
\begin{array}{c}
\begin{tikzpicture}[bend angle=15, looseness=1.2]
\node (a) at (-1,0) [vertex] {};
\node (b) at (0,0) [vertex] {};
\node (c) at (1,0) [vertex] {};

\node (a1) at (-1,-0.2) {$\scriptstyle 1$};
\node (a2) at (0,-0.2) {$\scriptstyle 2$};
\node (a3) at (1,-0.2) {$\scriptstyle 3$};

\draw[->,bend left] (a) to node[above] {$\scriptstyle a_{2}$} (b);
\draw[<-,bend right] (a) to node[below] {$\scriptstyle b_{2}$} (b);
\draw[->,bend left] (b) to node[above] {$\scriptstyle a_{4}$} (c);
\draw[<-,bend right] (b) to node[below] {$\scriptstyle b_{4}$} (c);
\draw[<-]  (a) edge [in=120,out=55,loop,looseness=11] node[above] {$\scriptstyle a_{1}$} (a);
\draw[<-]  (b) edge [in=120,out=55,loop,looseness=11] node[above] {$\scriptstyle a_{3}$} (b);
\draw[<-]  (c) edge [in=120,out=55,loop,looseness=11] node[above] {$\scriptstyle a_{5}$} (c);
\end{tikzpicture}
\end{array}
&
\begin{array}{l}
\x_1=\x_1'=a_1\\
\x_3=\x_3'=a_3\\
\x_5=\x_5'=a_5
\end{array}
\end{array}
\]
whereas $\x_2=a_2b_2$, $\x_2'=b_2a_2$, and $\x_4=a_4b_4$, $\x_4'=b_4a_4$. 

Given the above $\x_i$ and $\x_i'$, a \emph{monomialised Type $A$ potential} on $Q_{n}$ is a potential of the form
\begin{equation}\label{def:TypeA}
    \sum_{i=1}^{2n-2}\x_i^{\prime}\x_{i+1} + \sum_{i=1}^{2n-1}\sum_{j=2}^{\infty} k_{ij} \x_i^{j},
\end{equation}
for some coefficients $k_{ij} \in \C$. The main result of \cite{Z} shows that the complete Jacobi algebra of any monomialised Type $A$ potential on $Q_n$ can be realised as the contraction algebra of a crepant resolution of a $cA_n$ singularity (see \ref{11}). Moreover, there is a correspondence between crepant resolutions of $cA_n$ singularities and our intrinsic noncommutative monomialised Type $A$ potentials (see \ref{12}).

Since the contraction algebra determines the associated GV invariants by \ref{intro:determin}, this correspondence suggests studying GV invariants of crepant resolutions of $cA_n$ singularities via their associated monomialised Type~$A$ potentials on $Q_n$.

For fixed $n \geq 1$, consider the family of all monomialised Type~$A$ potentials on $Q_n$,
\begin{equation*}
     f(\upkappa)=\sum_{i=1}^{2n-2}\x_i^{\prime}\x_{i+1} + \sum_{i=1}^{2n-1}\sum_{j=2}^{\infty} \upkappa_{ij} \x_i^{j},
\end{equation*}
parametrised by
\begin{equation*}
   \M \colonequals  \bigl\{(k_{12},k_{13}, \dots, k_{22},k_{23},\dots, k_{2n-1,2},k_{2n-1,3}, \dots)\mid \text{all }k_{ij} \in\mathbb{C}\bigl\}.
\end{equation*}
Using the correspondence above, for each $k\in \M$ we define generalised GV invariants $N_{\upbeta}(f(k))$ via the associated crepant resolution (see \ref{def:gvf}).
The next result gives a filtration structure of the parameter space $\M$ in terms of generalised GV invariants.

\begin{theorem}[\ref{stra}]\label{13}
Fix $s,t$ with $1 \leq s \leq t \leq n$, then $\M$ admits a filtration structure $\M=M_1 \supsetneq M_2 \supsetneq M_3 \supsetneq \cdots $ such that:
\begin{enumerate}
\item For each $i \geq 1$, one has $N_{\upbeta}(f(k))=i$ for all $k \in M_{i} \setminus M_{i+1}$.
\item Each $M_{i}$ is the common zero locus of power series in the coordinates $\upkappa$.
\item If $s=t$, then for each $i \geq 2$, $M_i=\{k \in \M \mid k_{2s-1,j} =0 \textnormal{ for }2 \leq j \leq i \}$.
\end{enumerate}
\end{theorem}

We stress that the filtration in \ref{13} depends strongly on the curve class $\upbeta$; as $\upbeta$ varies, so does the filtration. This behaviour mirrors a familiar phenomenon in moduli theory: objects in general position exhibit the most regular behaviour, while more degenerate behaviour occurs along loci of higher codimension (see \ref{rmk:moduli}).

\subsection{Obstructions}
For any curve class $\upbeta$ and $N \in \N_{\infty} \colonequals \N \cup \infty$, Theorem~\ref{13} shows that there exists a crepant resolution $\uppi$ of a $cA_n$ singularity such that $N_{\upbeta}(\uppi)=N$.
However, this statement no longer holds when one considers
generalised GV invariants for multiple curve classes simultaneously.
We therefore study obstructions and constructions for the tuples of generalised GV invariants that can arise from crepant resolutions of $cA_n$ singularities.

\begin{notation}[\ref{notation:gv_tuple2}, \ref{notation:gv_tuple}]\label{notation:gv_tuple_intro}
Fix a curve class $\upbeta = \Curve_s+\Curve_{s+1}+\dots +\Curve_t$, and a tuple $(q_s,q_{s+1}, \dots, q_t)\in \N_{\infty}^{t-s+1}$. Set $\mathbf{q}_{\min}\colonequals \min\{q_i\}$, and consider the subset of crepant resolutions of $cA_n$ singularities with respect to $(q_s, \dots, q_t)$ defined as
\begin{equation*}
\mathsf{CA}_{\mathbf{q}} \colonequals\{ cA_n\textnormal{ crepant resolution }\uppi \mid (N_{\Curve_s}(\uppi),N_{\Curve_{s+1}}(\uppi),\dots, N_{\Curve_t}(\uppi))=(q_s, q_{s+1}, \dots,q_t)  \}.
\end{equation*}        
\end{notation}

The following is the main obstruction result, which is new even in the case when $\scrX$ is smooth and $\scrR$ is isolated (in which case $N_{\upbeta}=\GV_{\upbeta}$ by \ref{intro: 37}).

\begin{theorem}[\ref{thm:obs}]\label{thm:obs_intro}
Fix integers $s$, $t$ with $1 \leq s \leq t \leq n$, and a tuple $(q_s, q_{s+1},\dots, q_t)\in \N_{\infty}^{t-s+1}$. With notation as in \textnormal{\ref{notation:gv_tuple_intro}} and $\upbeta \colonequals \Curve_s+ \Curve_{s+1}+ \dots +\Curve_t$, the following holds.
\begin{enumerate}
\item For every $\uppi \in \mathsf{CA}_{\mathbf{q}}$ necessarily $N_{\upbeta}(\uppi) \geq \mathbf{q}_{\min}$. Moreover, there exists $\uppi \in \mathsf{CA}_{\mathbf{q}}$ such that $N_{\upbeta}(\uppi)=\mathbf{q}_{\min}$.
\item Assume that $\mathbf{q}_{\min}$ is finite. Then the equality $N_{\upbeta}(\uppi) = \mathbf{q}_{\min}$ holds for all $\uppi \in \mathsf{CA}_{\mathbf{q}}$ if and only if $\#\{i  \mid q_i=\mathbf{q}_{\min}\}=1 $.
\end{enumerate}   
\end{theorem}

In \ref{cor:obs} we show that the actions on curve classes from \cite[5.4]{NW1} and \cite[5.10]{VG}, together with \ref{thm:obs_intro}, yield further obstructions and constructions. As a sample consequence (see the end of \S\ref{section:obs} for more), we obtain the following complete description in the $cA_2$ case.

\begin{cor}[\ref{prop:gv_cA2}]
The generalised GV invariants of crepant resolutions of $cA_2$ singularities have the following two forms:
\[
\begin{tikzpicture}[bend angle=30, looseness=1]
\node (a) at (0,0) {$N_{\Curve_1}$};
\node (b) at (1.2,0) {$N_{\Curve_2}$};
\node (c) at (0.6,-0.6) {$N_{\Curve_1+\Curve_2}$};
\node (d) at (2,-0.4) {$=$};
\node (e) at (2.7,0) {$p$};
\node (f) at (3.5,0) {$q$};
\node (g) at (3.1,-0.6) {$\min(p,q)$};
\node (h) at (4.3,-0.4) {or};
\node (e) at (5,0) {$p$};
\node (f) at (5.8,0) {$p$};
\node (g) at (5.4,-0.6) {$r$};
\end{tikzpicture}
\]
where $p$, $q$, $r \in \N_{\infty}$ with $p \neq q$ and $ r \geq p$. Moreover, all possible such $p,q,r$ arise.
\end{cor}

\subsection*{Conventions}\label{con}
Throughout this paper, we work over the complex numbers $\mathbb{C}$, which is necessary for various statements in \S \ref{geometry}. We also adopt the following notation.
\begin{enumerate}
\item In \S\ref{sec: GVPartial}, $m$ denotes the number of exceptional curves of a crepant partial resolution of a $cA_n$ singularity.
\item In \S\ref{sec:GV}, \S\ref{sec: Filtrations} and \S\ref{section:obs}, $n$ denotes both the number of vertices of the quiver $Q_{n}$ and the index in $cA_n$. Moreover, $\p$ denotes a tuple $(p_1,p_2,\dots,p_{2n-1})$ with $p_i \in \N_{\infty}$ and $p \geq 2$ (see \ref{TypeAp}).
\item Vector space dimension is written $\dim_{\C}V$.
\end{enumerate}

\subsection*{Acknowledgements}
This work forms part of the author’s PhD at the University of Glasgow, funded by the China Scholarship Council.
The author would like to thank his supervisor, Michael Wemyss, for valuable guidance, and his external examiner, Alastair Craw, for helpful comments. He also thanks Yanki Lekili for helpful explanations concerning the interpretation of Gopakumar--Vafa invariants as local intersection numbers of plane curve germs.


\section{Preliminaries and Recap}

\subsection{Algebraic Preliminaries}
To set notation, let $Q=(Q_0,Q_1,t,h)$ be a \emph{quiver}, where $Q_0$ is a finite set of vertices, $Q_1$ is a finite set of arrows, and $t,h\colon Q_1\to Q_0$ are the tail and head maps. A path $a$ is \emph{cyclic} if $h(a) = t(a)$.

Let $k$ be a field. The \emph{complete path algebra} $k\lbl Q\rbl$ is the completion of the usual path algebra $kQ$ with respect to the arrow ideal. That is, the elements of $k\lbl Q\rbl$ are possibly infinite $k$-linear combinations of paths in $Q$.

\begin{definition}\label{QP}
Let $Q$ be a quiver.
\begin{enumerate}

\item  A \emph{quiver with potential} (\emph{QP} for short) is a pair $(Q, W)$, where $W$ is a $k$-linear combination of cyclic paths in $Q$.
\item For each $a \in Q_1$ and cyclic path $a_1 \dots a_d$ in $Q$, define the \emph{cyclic derivative} as
\begin{equation*}
\partial_{a}\left(a_{1} \ldots a_{d}\right)=\sum_{i=1}^{d} \delta_{a, a_{i}} a_{i+1} \ldots a_{d} a_{1} \ldots a_{i-1}
\end{equation*}
(where $\delta_{a, a_{i}}$ is the Kronecker delta), and then extend $\partial_{a}$ by linearity.
\item The \emph{Jacobi ideal} $J(W)$ is the closure of the two-sided ideal in $k\lbl Q\rbl$ generated by $\partial_{a}W$ for all $a \in Q_1$. 
\item The \emph{Jacobi algebra} $\Jac(Q,W)$ is the quotient $k\lbl Q\rbl/J(W)$. When the quiver is clear from context, we simply write $\Jac(W)$.
\end{enumerate}
\end{definition}

\subsection{Geometric Preliminaries}\label{geometry}
We briefly recall the geometric framework needed in this paper. We first introduce compound Du Val (cDV) singularities and their crepant (partial) resolutions. We then review modification algebras and contraction algebras associated to such resolutions, before recalling Gopakumar--Vafa invariants in the case of crepant resolutions.


Throughout the remainder of the paper, we reserve the notation $\scrR$ for complete local $\C$-algebras of the following form.
\begin{definition}
A complete local $\mathbb{C}$-algebra $\scrR$ is called a \emph{compound Du Val (cDV) singularity} if
\begin{equation*}
\scrR \cong \frac{\mathbb{C} \lal u, v, x, t \ral}{f+t g}
\end{equation*}
where $f \in \mathbb{C} \lal u, v, x \ral$ defines a Du Val, or equivalently Kleinian, surface singularity and $g \in \mathbb{C} \lal u, v, x, t \ral$ is arbitrary.
\end{definition}

\begin{definition}\label{def:crepant}
A projective birational morphism $\uppi \colon \scrX \rightarrow \Spec \scrR$ is called \emph{crepant partial resolution} if $\omega_{\scrX} \cong \uppi^{*}\omega_{\scrR}$. A \emph{minimal model} is a crepant partial resolution such that $\scrX$ has only $\Q$-factorial terminal singularities. When $\scrX$ is furthermore smooth, we call $\uppi$ a \emph{crepant resolution}.
If $\scrR$ is isolated, crepant partial resolutions and crepant resolutions are equivalently called \emph{flopping contractions} and \emph{smooth flopping contractions}, respectively $($see e.g. \cite[\S 1]{R1}$)$.
\end{definition}

Gopakumar--Vafa invariants are traditionally defined for smooth Calabi--Yau threefolds, which in the present setting corresponds to the case of crepant resolutions (see \S\ref{intro: GV}). Minimal models are often the primary objects of study in the birational geometry of threefolds. In this paper, we work more generally with crepant partial resolutions, allowing for singularities, with the aim of formulating and studying generalised Gopakumar--Vafa invariants in this broader context (see \S\ref{sec: generalised GV1}).

\subsubsection{Contraction Algebras}\label{intro: contration}
This subsection first introduces modification algebras and contraction algebras of crepant partial resolutions of cDV singularities and then recalls some associated theorems.

Given $\scrR$ \textnormal{cDV} as before, a finitely generated $\scrR$-module $M\in\mod \scrR$ is called \emph{maximal Cohen--Macaulay} (CM) provided
\[
\depth_\scrR M\colonequals\inf \{ i\geq 0\mid \Ext^i_\scrR(\scrR/ \m,M)\neq 0 \}=\dim \scrR.
\]
We write $\CM \, \scrR$ for the category of CM $\scrR$-modules.  Further, for $(-)^*\colonequals\Hom_\scrR(-,\scrR)$, $M\in\mod \scrR$ is called reflexive if the natural morphism $M\to M^{**}$ is an isomorphism, and we write $\refl \scrR$ for the category of reflexive $\scrR$-modules.

\begin{definition}\label{MMdefin}
We say $N\in\refl \scrR$ is a \emph{modifying (M) module} if $\End_\scrR(N)\in\CM \, \scrR$, and we say that $N\in\refl \scrR$ is a \emph{maximal modifying (MM) module} if it is modifying and it is maximal with respect to this property;  equivalently,
\[
\add N=\{ X\in\refl \scrR\mid \End_{\scrR}(N \oplus X)\in\CM \, \scrR  \}.
\]
If $N$ is an \textnormal{M} module $($resp. \textnormal{MM} module$)$, we call $\End_\scrR(N)$ a \emph{modification algebra} $($resp. \emph{maximal modification algebra}$)$.  
\end{definition}

The notion of a smooth noncommutative minimal model, known as a noncommutative crepant resolution, was introduced by Van den Bergh \cite{V1}.

\begin{definition} \label{thm32}
A \emph{noncommutative crepant resolution (NCCR)} of $\scrR$ is a algebra of the form $\Lambda\colonequals\End_\scrR(N)$ where $N \in\refl \scrR$, such that $\Lambda \in \CM \, \scrR$ and has finite global dimension.
\end{definition}

It turns out that if there exists an NCCR $\End_\scrR(N)$, then $N$ is automatically MM, and further all MM modules give NCCRs.  In other words, if one noncommutative minimal model is smooth, they all are \cite[5.11]{IW3}.

\begin{theorem}\cite[\S 4]{W1}\label{36}
Let $\scrR$ be \textnormal{cDV}, then there exist bijections
\begin{align*}
    \qquad \qquad (\mathrm{M}\,\scrR)\cap(\CM \,\scrR)  & \longleftrightarrow\left\{\text {crepant partial resolutions } \uppi: \scrX \rightarrow \Spec  \scrR \right\}, \\
    (\mathrm{MM}\,\scrR)\cap(\CM \,\scrR)  &\longleftrightarrow\left\{\text {minimal models } \uppi: \scrX \rightarrow \Spec  \scrR \right\}.
\end{align*}
If further $R$ admits a crepant resolution, then
\begin{equation*}
    (\mathrm{MM}\,\scrR)\cap(\CM \, \scrR)  \longleftrightarrow\left\{\text {crepant resolutions } \uppi: \scrX \rightarrow \Spec  \scrR \right\}.
\end{equation*}
\end{theorem}

The passage from left to right of the first line takes a given $N \in (\mathrm{M}\,\scrR)\cap(\CM \, \scrR)$ and associates a certain moduli space of representations of $\End_{\scrR}(N)$. 
In particular, passing from a crepant partial resolution to the corresponding modification algebra retains the geometric information relevant for our purposes.

We next explain the passage from right to left in detail. 
Let $\uppi \colon \scrX \rightarrow \Spec \scrR$ be a crepant partial resolution with exceptional curves $\Curve_1, \Curve_2 \dots, \Curve_m$. For any $1 \leq i \leq m$, there is a unique bundle $\mathcal{N}_i$ on $\scrX$ \cite[3.5.4]{V1},
and 
\begin{equation*}
    \mathcal{N} \colonequals \scrO_{\scrX} \oplus \bigoplus_{i=1}^m \mathcal{N}_i
\end{equation*}
is a tilting bundle on $\scrX$ \cite[3.5.5]{V1}. Pushing forward via $\uppi$ gives $\uppi_{*}(\scrO_{\scrX})= \scrR$ and $\uppi_{*}(\mathcal{N}_i)= N_i$ for some $\scrR$-module $N_i$. Set $N= \scrR \oplus \bigoplus_{i=1}^m N_i$. Then $N, \, \End_{\scrR}\left(N\right) \in \CM \, \scrR$ \cite[\S 4]{V1}, thus $N \in  (\mathrm{M}\,\scrR)\cap(\CM \, \scrR)$.
Therefore,
\begin{equation*}
    \Lambda(\uppi) \colonequals \End_{\scrX}(\mathcal{N})  \cong \End_{\scrR}\left(N\right).
\end{equation*}
where the isomorphism follows from the crepancy of $\uppi$; see \cite[3.2.10]{V1}.

The contraction algebra associated to $\uppi$ can be defined as a quotient of the modification algebra $\End_{\scrR}\left(N\right)$.
\begin{definition}
With notation above, define the \emph{contraction algebra} associated to a crepant partial resolution $\uppi$ to be the stable endomorphism algebra
\begin{equation*}
   \Lambda_{\mathrm{con}}(\uppi)  \colonequals  \underline{\End}_{\scrR}(N) = \End_{\scrR}(N)/\langle \scrR \rangle,
\end{equation*}
where $\langle \scrR \rangle$ denotes the two-sided ideal consisting of all morphisms which factor through $\add \scrR$.
\end{definition}

The distinction between flopping contractions and divisor-to-curve contractions can be detected by the finite dimensionality of the associated contraction algebra, as follows.
\begin{theorem}\label{352}
Suppose that $\uppi \colon \scrX \rightarrow \Spec \scrR$ is a crepant partial resolution, and write $Z$ for the locus in $\Spec \scrR$ for which $\uppi$ is not an isomorphism. Then $\operatorname{Supp}_{\scrR}\Lambda_{\mathrm{con}}(\uppi)= Z$, and 
\begin{equation*}
    \uppi  \text{ is a flopping contraction}  \iff  \operatorname{dim}_{\mathbb{C}}\Lambda_{\mathrm{con}}(\uppi) < \infty.
\end{equation*}
If moreover $\scrX$ is smooth, then these conditions are equivalent to $\scrR$ being an isolated singularity.
\end{theorem} 

\begin{proof}
The equivalences preceding the final assertion are established in \cite[4.8]{DW1}, so we only justify the case where $\scrX$ is smooth. 
If $\scrX$ is smooth, then the locus over which $\uppi$ fails to be an isomorphism coincides with the singular locus of $\scrR$, that is, $Z= \Sing \scrR$.  Hence $\Supp_{\scrR}\Lambda_{\mathrm{con}}(\uppi)= \Sing \scrR$. 

Since $\scrR$ is complete local and $\Lambda_{\mathrm{con}}(\uppi)$ is finitely generated over $\scrR$, it follows from \cite[2.13, 2.15, 2.17]{E3} that
\[
\operatorname{dim}_{\mathbb{C}}\Lambda_{\mathrm{con}}(\uppi) < \infty \iff \Supp_{\scrR}\Lambda_{\mathrm{con}}(\uppi)= V(\m),
\]
where $\m$ denotes the maximal ideal of $\scrR$ and $V(\m)\subseteq \Spec \scrR$ is the corresponding Zariski closed point.
Therefore, $\dim_{\mathbb{C}}\Lambda_{\mathrm{con}}(\uppi) < \infty  \iff \Sing\scrR=V(\m)$, which is equivalent to $\scrR$ having an isolated singularity.
\end{proof}

\subsubsection{Gopakumar--Vafa invariants}\label{intro: GV}
Now let $\uppi: \scrX \to\Spec\scrR$ be a crepant resolution.
The reduced fibre above the origin $\uppi^{-1}(0)^{\textnormal{red}} = \bigcup_{i=1}^m \Curve_i$ is a union of rational curves. Let $A_1(\uppi) \colonequals \bigoplus_{i=1}^m \Z \left\langle \Curve_i \right\rangle$ be the abelian group freely generated by these curves. 

For a curve class $\upbeta=(\upbeta_1,\dots,\upbeta_m)\in A_1(\uppi)$, the associated genus-zero Gopakumar--Vafa invariant $\GV_{\upbeta}(\scrX)$ is defined using the virtual fundamental class (equivalently, via Behrend's constructible function), and therefore yields a deformation-invariant numerical invariant.
\begin{definition}\label{GV}
The invariant $\GV_{\upbeta}(\scrX)$ admits the following equivalent interpretations:
\begin{enumerate}
\item Set
\begin{equation*}
    \GV_{\upbeta}(\scrX)=\int_{\operatorname{Sh}_{\upbeta}(\scrX)} v=\sum_{n \in \mathbb{Z}} n \chi\left(v^{-1}(n)\right) \quad \text{ or } \quad \GV_{\upbeta}(\scrX)= \int_{[\operatorname{Sh}_{\upbeta}(\scrX)]^{vir}} 1
\end{equation*}
where $v$ denotes Behrend's constructible function \cite{B} on the moduli space $\operatorname{Sh}_{\upbeta}(\scrX)$ of one-dimensional stable sheaves $F$ on $\scrX$ with support $\upbeta$ and Euler characteristic $\chi(F)=1$. This definition is equivalent to integration against the virtual fundamental class $[\operatorname{Sh}_{\upbeta}(\scrX)]^{vir}$ induced by the associated symmetric perfect obstruction theory \cite{K, MT}.

\item $\GV_{\upbeta}(\scrX)=\Omega_{\scrX}^{\text{num}}(1, \upbeta)$ where $\Omega_{\scrX}(1, \upbeta)$ is a noncommutative BPS invariant \cite{VG}.

\item If $\scrR$ is isolated, $\GV_{\upbeta}(\scrX)$ is equal to the number of $(-1, -1)$-curves with curve class $\upbeta$ on a one-parameter deformation of $\uppi \colon \scrX \rightarrow \Spec \scrR$ \cite{BKL}. 
\end{enumerate}
\end{definition}

If furthermore $\scrR$ is isolated, GV invariants can be read off from the contraction algebra via Toda's formula.

\begin{theorem} \textnormal{(Toda's formula, \cite[\S 4.4]{T2})}  \label{34}
Let $\uppi \colon \scrX \rightarrow \Spec \scrR$ be a crepant resolution of an isolated \textnormal{cDV} singularity $\scrR$ with exceptional curves $\bigcup_{i=1}^m \Curve_i$. For any $1 \leq s  \leq t \leq m$, the following equality holds.
\begin{equation*}
    \operatorname{dim}_{\C}e_s\Lambda_{\mathrm{con}}(\uppi)e_t= 
   \sum_{\upbeta=(\upbeta_1,\dots ,\upbeta_m)} \upbeta_s \cdot \upbeta_t \cdot \GV_{\upbeta}(\uppi)=\operatorname{dim}_{\C}e_t\Lambda_{\mathrm{con}}(\uppi)e_s.
\end{equation*}
In particular, $ \operatorname{dim}_{\C}\Lambda_{\mathrm{con}}(\uppi)=\sum_{\upbeta} |\upbeta|^2 \GV_{\upbeta}(\uppi)$ where $|\upbeta| = \upbeta_1 + \dots + \upbeta_m$.
\end{theorem}

\subsection{Recap}
We conclude this section by recalling several foundational results from~\cite{Z} that will be used throughout the paper.

The first main result in \cite{Z} is that the complete Jacobi algebra of any monomialised Type $A$ potential on $Q_n$ (as defined in \eqref{def:TypeA}) can be realised as the contraction algebra of a crepant resolution of some $cA_n$ singularity. 

\begin{theorem}\cite[5.11]{Z}\label{11}
For any monomialised Type $A$ potential $f$ on $Q_n$, there exists a crepant resolution $\uppi \colon \scrX \rightarrow \Spec\scrR$ where $\scrR$ is $cA_n$, such that $\Lambda_{\mathrm{con}}(\uppi) \cong \Jac(f)$.
\end{theorem}

We furthermore obtain the converse to \ref{11}, as follows.
\begin{theorem}\cite[5.13]{Z}\label{514}
For any crepant resolution $\uppi \colon \scrX \rightarrow \Spec \scrR$ where $\scrR$ is $cA_n$, there exists a monomialised Type $A$ potential $f$ on $Q_{n}$ such that $\Jac(f) \cong  \Lambda_{\mathrm{con}}(\uppi)$.   
\end{theorem}

Combining \ref{11} and \ref{514} gives a correspondence between crepant resolutions of $cA_n$ singularities and our intrinsic noncommutative monomialised Type $A$ potentials, as follows.

\begin{cor}\cite[5.21]{Z}\label{12}
For any $n$, the set of isomorphism classes of contraction algebras associated to crepant resolutions of $cA_n$ singularities is equal to the set of isomorphism classes of Jacobi algebras of monomialised Type $A$ potentials on $Q_n$.
\end{cor}

\section{Generalised GV Invariants of Crepant Partial Resolutions}\label{sec: GVPartial}

In this section, we introduce generalised Gopakumar--Vafa invariants for crepant partial resolutions of $cA_n$ singularities. In \S\ref{sec: generalised GV1}, we define these invariants in purely algebraic terms. In \S\ref{sec: generalised GV2}, we show that they satisfy an analogue of Toda's formula and are determined by the associated contraction algebra. Finally, in \S\ref{sec: generalised GV3}, we restrict to crepant resolutions and show that the generalised GV invariants recover the classical GV invariants.

\subsection{Generalised GV invariants}\label{sec: generalised GV1}
Recall that every $cA_{t-1}$ singularity $\scrR$ has the form 
\begin{equation*}
    \scrR \cong \frac{\mathbb{C} \lal u, v, x, y \ral}{uv-f_0f_1 \dots f_{n}},
\end{equation*}
where $t$ is the order of the power series $f_0f_1 \dots f_{n}$, and each $f_i$ is a prime element of $\mathbb{C}\lal x,y \ral$. 
For any subset $I\subseteq\{ 0,1,\hdots, n\}$ set $I^c=\{ 0,1,\hdots, n\}\setminus I$ and denote
\[
f_I:=\prod_{i\in I}f_i\ \mbox{ and }\ M_I:=(u,f_I)
\]
where $M_I$ is an ideal of $\scrR$ generated by $u$ and $f_I$. For a collection of subsets $\emptyset\subsetneq I_1\subsetneq I_2\subsetneq\hdots\subsetneq
I_m\subsetneq\{0,1,\hdots,n\}$, we say that $\F=(I_1,\hdots,I_m)$ is a {\em flag in the set $\{ 0,1,\hdots, n\}$}.  We say that the flag $\c{F}$ is {\em maximal} if $n=m$.  Given a flag $\c{F}=(I_1,\hdots,I_m)$, we define
\[
M^\c{F}:=\scrR\oplus\left(\bigoplus_{j=1}^{m} M_{I_j}\right) .
\]

To ease notation, set $I_0:=\emptyset$ and $I_{m+1}:=\{0,1,\hdots,n\}$, and then $g_j:=f_{I_{j+1}\setminus I_{j}}$ for all $0\leq j\leq m$. Thus $f_{I_j}=\prod_{i=0}^{j-1}g_i$ and $M_{I_j}=(u,\prod_{i=0}^{j-1}g_i)$.
Then using \cite[\S 5]{IW1} $\F$ is given pictorially by
\[
\begin{array}{ccc}
\begin{array}{c}
\F
\end{array} &
\begin{array}{c}
\begin{tikzpicture}[xscale=0.6,yscale=0.6]
\draw[black] (-0.1,-0.04,0) to [bend left=25] (2.1,-0.04,0);
\draw[black] (1.9,-0.04,0) to [bend left=25] (4.1,-0.04,0);
\node at (5.5,0,0) {$\hdots$};
\draw[black] (6.9,-0.04,0) to [bend left=25] (9.1,-0.04,0);
\node at (1,0.6,0) {$\scriptstyle \Curve_{1}$};
\node at (3,0.6,0) {$\scriptstyle \Curve_{2}$};
\node at (8,0.6,0) {$\scriptstyle \Curve_{m}$};
\filldraw [red] (0,0,0) circle (1pt);
\filldraw [red] (2,0,0) circle (1pt);
\filldraw [red] (4,0,0) circle (1pt);
\filldraw [red] (7,0,0) circle (1pt);
\filldraw [red] (9,0,0) circle (1pt);
\node at (0,-0.4,0) {$\scriptstyle g_0$};
\node at (2,-0.4,0) {$\scriptstyle g_1$};
\node at (4,-0.4,0) {$\scriptstyle g_2$};
\node at (7,-0.4,0) {$\scriptstyle g_{m-1}$};
\node at (9,-0.4,0) {$\scriptstyle g_{m}$};
\end{tikzpicture} 
\end{array}
\end{array}
\]

By \cite[5.1]{IW1}, the set $(\mathrm{M}\,\scrR)\cap(\CM \, \scrR)$ is equal to modules $M^{\F}$, where $\F$ is a flag in $\{0,1,\hdots,n\}$.
By \ref{36}, for each flag $\F$ there exists a crepant partial resolution $\uppi^{\F} \colon \scrX^{\F} \rightarrow \Spec\scrR$ such that $\Lambda(\uppi^{\F}) \cong  \End_{\scrR}(M^{\F})$ and $\Lambda_{\con}(\uppi^{\F}) \cong  \underline\End_{\scrR}(M^{\F})$.

We now introduce our numerical invariants for crepant partial resolutions.
\begin{definition}\label{def:Nij}
With notation as above, define the \emph{generalised GV invariant} $N_{\upbeta}(\uppi^{\F})$ for a curve class $\upbeta  \in \bigoplus_{i=1}^m \Z \left\langle \Curve_i \right\rangle$ by
\[
N_{\upbeta}(\uppi^{\F})\colonequals
\begin{cases}
\dim_\mathbb{C}\tfrac{\mathbb{C}\lal x,y \ral}{(g_{i-1},g_{j})} & \mbox{if } \upbeta = \Curve_i +  \Curve_{i+1}+\hdots+ \Curve_j\\
0 &\mbox{otherwise.}
\end{cases}
\]
\end{definition}
The above generalised GV invariant \ref{def:Nij} is parallel to GV invariants. Indeed, if $\uppi^{\F}$ is a crepant resolution, then $\{\Curve_i+\Curve_{i+1}+ \dots +\Curve_j \mid 1\leq i \leq j \leq m\}$ are the only curve classes with non-zero \textnormal{GV} invariants \cite{NW1, VG}.

Thus throughout this paper we will often write $N_{ij}(\uppi)$ (resp. $\GV_{ij}(\uppi)$) for $N_{\upbeta}(\uppi)$ (resp. $\GV_{\upbeta}(\uppi)$) when $\upbeta=\Curve_i +  \Curve_{i+1}+\hdots+ \Curve_j$.

\begin{example}\label{example:Nij}
Consider $f_0f_1f_2f_3f_4f_5$ with a flag $\c{F}=(\{ 0,1\}\subsetneq \{ 0,1,2 \})$. Then $g_0=f_0f_1$, $g_1=f_2$, $g_2=f_3f_4f_5$, and $\c{F}$ corresponds to
\[
\begin{tikzpicture}[xscale=0.6,yscale=0.6]
\draw[black] (-0.1,-0.04,0) to [bend left=25] (2.1,-0.04,0);
\draw[black] (1.9,-0.04,0) to [bend left=25] (4.1,-0.04,0);
\filldraw [red] (0,0,0) circle (1pt);
\filldraw [red] (2,0,0) circle (1pt);
\filldraw [red] (4,0,0) circle (1pt);
\node at (0,-0.4,0) {$\scriptstyle f_0f_1$};
\node at (2,-0.4,0) {$\scriptstyle f_2$};
\node at (4,-0.4,0) {$\scriptstyle f_3f_4f_5$};

\node at (1,0.6,0) {$\scriptstyle \Curve_{1}$};
\node at (3,0.6,0) {$\scriptstyle \Curve_{2}$};
\end{tikzpicture}
\]
Then $M^\c{F}$ is $\scrR\oplus (u,f_0f_1)\oplus (u,f_0f_1f_2)$, and the generalised GV invariants are
\begin{equation*}
    N_{11}(\uppi^{\F})= \operatorname{dim}_{\mathbb{C}} \frac{\mathbb{C}\lal x,y \ral}{(f_0f_1,f_2)}, \ N_{22}(\uppi^{\F})= \operatorname{dim}_{\mathbb{C}} \frac{\mathbb{C}\lal x,y \ral}{(f_2,f_3f_4f_5)}, \ N_{12}(\uppi^{\F})= \operatorname{dim}_{\mathbb{C}} \frac{\mathbb{C}\lal x,y \ral}{(f_0f_1,f_3f_4f_5)}.
\end{equation*}
\end{example}

\begin{remark}\label{rmk:intersect}
The generalised GV invariants $N_{ij}(\uppi^{\F})$ from Definition~\ref{def:Nij}
admit an interpretation as local intersection multiplicities of plane curve germs.
Recall that
\[
g_{i-1}= \prod_{k \in I_{i}\setminus I_{i-1}} f_k
\quad \text{and} \quad
g_{j}= \prod_{k \in I_{j+1}\setminus I_{j}} f_k,
\]
where each $f_k$ is irreducible in $\C\lal x,y\ral$.
Hence
\[
Z(g_{i-1})=\bigcup_{k \in I_{i}\setminus I_{i-1}} Z(f_k),
\qquad
Z(g_{j})=\bigcup_{k \in I_{j+1}\setminus I_{j}} Z(f_k),
\]
and each $Z(f_k)$ is an irreducible plane curve germ passing through the origin.

Assume that $Z(g_{i-1})$ and $Z(g_j)$ have no common irreducible component.
Equivalently, $\Spec \C\lal x,y \ral/(g_{i-1},g_j)$ is zero-dimensional, or
$N_{ij}(\uppi^{\F})<\infty$ by \ref{lemma:Toda2} below.
In this case, the local intersection multiplicity at the origin is given by
\[
I_0\big(Z(g_{i-1}),Z(g_j)\big)
=\operatorname{dim}_{\C}\frac{\mathbb{C}\lal x,y \ral}{(g_{i-1},g_j)}.
\]
Thus, under the above hypothesis, we have
\[
N_{ij}(\uppi^{\F})
= I_0\big(Z(g_{i-1}),Z(g_j)\big).
\]

Moreover, by \ref{37} below, the generalised GV invariants recover the
classical GV invariants whenever $\uppi^{\F}$ is a crepant resolution
(equivalently, $\scrX^{\F}$ is smooth).
Consequently, in the $cA_n$ case, a curve-counting invariant of a threefold can be expressed as a local
intersection multiplicity of plane curve germs.
\end{remark}

\begin{example}\label{example:gv}
Consider the $cA_2$ singularity $\scrR$ defined by $uv=xy(x+y^n)$ where $n \geq 1$ and the $\scrR$-module $M= \scrR \oplus (u,x) \oplus (u,xy)$. Let $\uppi \colon \scrX \rightarrow \Spec \scrR$ be the associated crepant resolution, which is given pictorially by
\[
\begin{array}{ccc}
\begin{array}{c}
\scrX
\end{array} &
\begin{array}{c}
\begin{tikzpicture}[xscale=0.6,yscale=0.6]
\draw[black] (-0.1,-0.04,0) to [bend left=25] (2.1,-0.04,0);
\draw[black] (1.9,-0.04,0) to [bend left=25] (4.1,-0.04,0);

\node at (1,0.6,0) {$\scriptstyle \Curve_{1}$};
\node at (3,0.6,0) {$\scriptstyle \Curve_{2}$};

\node at (0,-0.4,0) {$\scriptstyle x$};
\node at (2,-0.4,0) {$\scriptstyle y$};
\node at (4,-0.4,0) {$\scriptstyle x+y^n$};

\end{tikzpicture} 
\end{array}
\end{array}
\]
The generalised GV invariants of $\uppi$ are
\begin{enumerate}
\item $N_{11}(\uppi)=\operatorname{dim}_\mathbb{C} \mathbb{C}\lal x,y \ral/(x,y) =1$, which equals the local intersection number of the curves $x=0$ and $y=0$ at the origin.
\item $N_{22}(\uppi)=\operatorname{dim}_\mathbb{C} \mathbb{C}\lal x,y \ral/(y,x+y^n) =1$, which equals the local intersection number of the curves $y=0$ and $x+y^n=0$.
\item $N_{12}(\uppi)=\operatorname{dim}_\mathbb{C} \mathbb{C}\lal x,y \ral/(x,x+y^n) =n$, which equals the local intersection number of the curves $x=0$ and $x+y^n=0$.
\end{enumerate}

For ease of visualisation, Figure~\ref{fig:example-n2} depicts these curves inside $\mathbb{C}^2$ in the case $n=2$. 

Although the generalised GV invariants are defined using the complete local ring $\mathbb{C}\lal x,y \ral$, the intersection numbers depend only on the germs at the origin, so the global picture faithfully reflects the local geometry.

\begin{figure}[h]
  \centering
\begin{tikzpicture}[scale=0.6]
  \draw[->] (-2.5,0) -- (2.5,0) node[right] {$y=0$};
  \draw[->] (0,-2.2) -- (0,2.2) node[above] {$x=0$};
  \draw[thick] (-2.5,0) -- (2.5,0);
  \draw[thick] (0,-2.2) -- (0,2.2);
  \draw[thick,domain=-1.5:1.5,samples=100]
    plot ({- \x*\x}, {\x});
    \node at (-3,2) {$x + y^2 = 0$};
\end{tikzpicture}
\caption{The case $n=2$.}
\label{fig:example-n2}
\end{figure}
\end{example}

The following result gives an explicit quiver description of the modification algebras associated to $cA_n$ singularities.

\begin{cor}\cite[5.33]{IW1}\label{35}
Given a flag $\c{F}=(I_1,\hdots,I_m)$, with notation as above the quiver of $\End_{\scrR}(M^\c{F})$ is as follows:
\[
\begin{tikzpicture}
\node at (0,-1.3) {$\scriptstyle m\geq 2$};
\node at (0,0)
{\begin{tikzpicture}[xscale=1.8,yscale=1.4,bend angle=13, looseness=1]
\node (1) at (1,0) {$\scriptstyle {M_{I_1}}$}; 
\node (2) at (2,0) {$\scriptstyle {M_{I_2}}$};
\node (4) at (3,0) {$\scriptstyle \cdots$};
\node (5) at (4,0) {$\scriptstyle {M_{I_m}}$};
\node (R) at (2.5,-1) {$\scriptstyle \scrR$};
\draw [bend right,<-,pos=0.5] (1) to node[inner sep=0.5pt,fill=white,below=-3pt] {$\scriptstyle inc$} (2);
\draw [bend right,<-,pos=0.5] (2) to node[inner sep=0.5pt,fill=white,below=-3pt] {$\scriptstyle g_1$}(1);
\draw [bend right,<-,pos=0.5] (2) to node[inner sep=0.5pt,fill=white,below=-3pt] {$\scriptstyle inc$} (4);
\draw [bend right,<-,pos=0.5] (4) to node[inner sep=0.5pt,fill=white,below=-3pt] {$\scriptstyle g_2$}(2);
\draw [bend right,<-,pos=0.5] (4) to node[inner sep=0.5pt,fill=white,below=-3pt] {$\scriptstyle inc$} (5);
\draw [bend right,<-,pos=0.5] (5) to node[inner sep=0.5pt,fill=white,below=-3pt] {$\scriptstyle g_{m-1}$} (4);
\draw [bend right=7,<-,pos=0.5] ($(R)+(135:4pt)$) to node[inner sep=0.5pt,fill=white,below=-5pt] {$\scriptstyle inc$} ($(1)+(-45:6pt)$);
\draw [bend right=7,<-,pos=0.5] ($(1)+(-75:6pt)$) to node[inner sep=0.5pt,fill=white,below=-3pt] {$\scriptstyle g_0$}  ($(R)+(165:4pt)$);
\draw [bend right=7,<-,pos=0.5] ($(R)+(15:4pt)$) to node[inner sep=0.5pt,fill=white,below=-1pt] {$\scriptstyle \frac{g_{m}}{u}$} ($(5)+(-100:5pt)$);
\draw [bend right=7,<-,pos=0.5] ($(5)+(-125:5pt)$) to node[inner sep=0.5pt,fill=white,below=-3pt] {$\scriptstyle u$} ($(R)+(45:4pt)$);
\end{tikzpicture}};
\node at (6,0) {\begin{tikzpicture} 
\node (C1) at (0,0)  {$\scriptstyle \scrR$};
\node (C1a) at (-0.1,0.05)  {};
\node (C1b) at (-0.1,-0.05)  {};
\node (C2) at (1.75,0)  {$\scriptstyle M_{I_1}$};
\node (C2a) at (1.85,0.05) {};
\node (C2b) at (1.85,-0.05) {};
\draw [->,bend left=45,looseness=1,pos=0.5] (C1) to node[inner sep=0.5pt,fill=white]  {$\scriptstyle g_0$} (C2);
\draw [->,bend left=20,looseness=1,pos=0.5] (C1) to node[inner sep=0.5pt,fill=white]  {$\scriptstyle u$} (C2);
\draw [->,bend left=45,looseness=1,pos=0.5] (C2) to node[inner sep=0.5pt,fill=white]  {$\scriptstyle \frac{g_1}{u}$} (C1);
\draw [->,bend left=20,looseness=1,pos=0.5] (C2) to node[inner sep=0.5pt,fill=white,below=-5pt] {$\scriptstyle inc$} (C1);
\end{tikzpicture}};
\node at (6,-1.3) {$\scriptstyle  m=1$};
\end{tikzpicture}
\] 
together with the possible addition of some loops, given by the following rules:
\begin{itemize}
\item Consider vertex $\scrR$.  If $(g_0,g_{m})=(x,y)$ in the ring $\C\lal x,y\ral$, add no loops at vertex $\scrR$. Hence suppose $(g_0,g_{m})\subsetneq (x,y)$. If there exists $t\in (x,y)$ such that $(g_0,g_{m},t)=(x,y)$, add a loop labelled $t$ at vertex $\scrR$.  If there exists no such $t$, add two loops labelled $x$ and $y$ at vertex $\scrR$.
\item Consider vertex $M_{I_i}$.  If $(g_{i-1},g_{i})=(x,y)$ in the ring $\C\lal x,y\ral$, add no loops at vertex $M_{I_i}$.  Hence suppose $(g_{i-1},g_{i})\subsetneq (x,y)$. If there exists $t\in (x,y)$ such that $(g_{i-1},g_{i},t)=(x,y)$, add a loop labelled $t$ at vertex $M_{I_i}$.  If there exists no such $t$, add two loops labelled $x$ and $y$ at vertex $M_{I_i}$.
\end{itemize}
\end{cor}

\subsection{Contraction algebra determines generalised GV invariants}\label{sec: generalised GV2}

Throughout this subsection we retain the notation $\scrR$, $\F$, $M_{I_j}$, $g_j$ and $\uppi^{\F}$ from \S\ref{sec: generalised GV1}.
The main technical input is the following proposition, which identifies certain stable Hom spaces in $\underline{\CM}\,\scrR$ with the quotient ring appearing in the definition of the corresponding generalised GV invariant.

\begin{prop}\label{lemma:con_iso}
There are isomorphisms of $\scrR$-modules
\[
\underline{\Hom}_{\scrR}\big((u,g_0),(u,g_0\cdots g_{m-1})\big)
\ \cong\ 
\frac{\C\lal u,v,x,y\ral}{(u,v,g_0,g_m)}
\ \cong\ 
\underline{\Hom}_{\scrR}\big((u,g_0\cdots g_{m-1}),(u,g_0)\big).
\]
In particular, the dimension of each as a $\mathbb{C}$-vector space equals $\dim_{\C} \mathbb{C} \lal x, y \ral/(g_0,g_m)$.
\end{prop}
\begin{proof}
\noindent
(1) We prove 
\[
\underline{\Hom}_{\scrR}\big((u,g_0),(u,g_0\cdots g_{m-1})\big)
\cong \frac{\mathbb{C}\lal u, v, x, y \ral}{(u,v,g_0,g_m)}.
\]
We first claim that $\underline{\Hom}_{\scrR}\big((u,g_0),(u,g_0\hdots g_{m-1})\big) \cong \Ext_{\scrR}^1\big((u,g_0),(u,g_m)\big)$.

From \cite[\S 5]{IW1} there is an exact sequence
\begin{equation}\label{302}
0 \rightarrow (u,g_m) \xrightarrow[]{\begin{psmallmatrix} \frac{g_0\hdots g_{m-1}}{u}&-inc \end{psmallmatrix}} \scrR^2 \xrightarrow[]{\begin{psmallmatrix}u \\ g_0\hdots g_{m-1} \end{psmallmatrix}} (u, g_0\hdots g_{m-1}) \rightarrow 0.
\end{equation}
Thus $\Omega (u,g_0\hdots g_{m-1})=(u,g_m)$ where $\Omega$ denotes the syzygy. Then we have
\begin{align*}
\underline{\Hom}_{\scrR}\big((u,g_0),(u,\prod_{i=0}^{m-1}g_i)\big) &\cong  \underline{\Hom}_{\scrR}\big((u,g_0),\Omega (u,\prod_{i=0}^{m-1}g_i)[1]\big) \tag{$\Omega[1]=\mathrm{Id}$ in $\underline{\mathrm{CM}}\,\scrR$}\\
&\cong \underline{\Hom}_{\scrR}\big((u,g_0), (u,g_m)[1]\big)\tag{by above}\\ 
 &\cong \Ext_{\scrR}^1\big((u,g_0),(u,g_m)\big)\tag{by e.g. \cite{IW3}}.
\end{align*}

We next claim that $\Ext_{\scrR}^1\big((u,g_0),(u,g_m)\big) \cong (u, G)/(u, g_0G, Gg_m, Gv)$ as $\scrR$-modules, where $G \colonequals g_1g_2 \dots g_{m-1}$ and the right-hand side is the quotient of one ideal by another.

Applying $\mathbb{F}=\Hom\big((u,g_0),-\big)$ to the short exact sequence \eqref{302} gives
\[
0 \to \mathbb{F} (u,g_m) \rightarrow \mathbb{F}\scrR^2  \xrightarrow[]{\begin{psmallmatrix}u \\ \prod_{i=0}^{m-1} g_{i} \end{psmallmatrix}}
   \mathbb{F}(u,\prod_{i=0}^{m-1}g_i) \to  \Ext_{\scrR}^1\big((u,g_0),(u,g_m)\big) \rightarrow \Ext_{\scrR}^1\big((u,g_0),\scrR^2\big).
\]
Since $(u,g_0)\in\mathrm{CM}\, \scrR$ by \cite[5.3]{IW1}, $\Ext_{\scrR}^1\big((u,g_0),\scrR^2\big)=0$. Further, by \cite[5.4]{IW1}, there are isomorphisms
\begin{align*}
    (u,\prod_{i=1}^m g_i) &\cong \mathbb{F}\scrR\quad\mbox{via } \ r \mapsto(\cdot \frac{r}{u}), \\
     (u, \prod_{i=1}^{m-1} g_i) &\cong \mathbb{F}(u,\prod_{i=0}^{m-1}g_i)\quad\mbox{via }  \ r \mapsto(\cdot r).
\end{align*}
Combining these together gives an exact sequence
\begin{equation*}
    (u,\prod_{i=1}^m g_i)^{\oplus 2} \xrightarrow[]{d=\begin{psmallmatrix}inc \\ \frac{\prod_{i=0}^{m-1} g_{i}}{u} \end{psmallmatrix}}  (u, \prod_{i=1}^{m-1} g_i)  \rightarrow  \Ext_{\scrR}^1\big((u,g_0),(u,g_m)\big) \rightarrow 0.
\end{equation*}
Thus $\Ext_{\scrR}^1\big((u,g_0),(u,g_m)\big) \cong  (u, \prod_{i=1}^{m-1} g_i) /\Im d$. It is elementary to check that $\Im d \cong (u,g_0G,g_mG,vG)$, proving the second claim.

Finally, we claim that $(u,G)/(u,g_0G,g_mG,vG) \cong \mathbb{C} \lal u, v, x, y \ral / (u,v,g_0,g_m)$ as $\scrR$-modules.

We first define a $\mathbb{C} \lal u, v, x, y \ral$-homomorphism $\varphi$ as follows,
\begin{equation*}
    \varphi \colon \mathbb{C} \lal u, v, x, y \ral \xrightarrow[]{\cdot G}(u,G)/(u,g_0G,g_mG,vG).
\end{equation*}
Clearly, $\varphi$ is well defined and $(u,v,g_0,g_m) \subseteq \ker \varphi$. We claim that $\ker \varphi\subseteq (u,v,g_0,g_m)$. 

Let $r \in \mathbb{C} \lal u, v, x, y \ral$ be such that $\varphi(r)=0$.  Then $rG= r_1u+r_2g_0G+r_3g_mG+r_4vG$ for some $r_i \in \mathbb{C} \lal u, v, x, y \ral$. Thus $r_1u=(r-r_2g_0-r_3g_m-r_4v)G$. Since $u$ and $G$ have no common factors, we have $r_1 = r_5G$ for some $r_5 \in \mathbb{C} \lal u, v, x, y \ral$. Thus $rG= (r_5u+r_2g_0+r_3g_m+r_4v)G$. Since $\mathbb{C} \lal u, v, x, y \ral$ is domain, then $r=r_5u+r_2g_0+r_3g_m+r_4v \in (u,v,g_0,g_m)$, and so $\ker \varphi\subseteq (u,v,g_0,g_m)$, proving the claim.  
Thus $\ker \varphi= (u,v,g_0,g_m)$. 

Since $\varphi$ is evidently surjective, it induces a $\mathbb{C} \lal u, v, x, y \ral$-isomorphism
\begin{equation*}
    \overline{\varphi}  \colon \frac{\mathbb{C} \lal u, v, x, y \ral}{(u,v,g_0,g_m)} \xrightarrow[]{\sim} \frac{(u,G)}{(u,g_0G,Gg_m,Gv)}.
\end{equation*}
It is easy to check that this is also an $\scrR$-module isomorphism.

\noindent(2) We prove
\[
\underline{\Hom}_{\scrR}\big((u,g_0\hdots g_{m-1}),(u,g_0)\big)  \cong  \frac{ \mathbb{C} \lal u, v, x, y \ral}{(u,v,g_0,g_m)}.
\]
We first claim that $\underline{\Hom}_{\scrR}\big((u,g_0\hdots g_{m-1}),(u,g_0)\big) \cong \Ext_{\scrR}^1\big((u,\prod_{i=0}^{m-1}g_i),(u,\prod_{i=1}^{m}g_i)\big)$.

Similar to (1), from \cite[\S 5]{IW1} there is an exact sequence
\begin{equation}\label{303}
0 \rightarrow (u,g_1 \hdots g_m) \xrightarrow[]{\begin{psmallmatrix} \frac{g_0}{u}&-inc \end{psmallmatrix}} \scrR^2 \xrightarrow[]{\begin{psmallmatrix}u \\ g_0 \end{psmallmatrix}} (u, g_0) \rightarrow 0.
\end{equation}
Thus $\Omega (u,g_0)=(u,g_1 \hdots g_m)$ and
\begin{align*}
\underline{\Hom}_{\scrR}\big((u,\prod_{i=0}^{m-1}g_i),(u,g_0)\big) &\cong  \underline{\Hom}_{\scrR}\big((u,\prod_{i=0}^{m-1}g_i),\Omega (u,g_0)[1]\big) \tag{$\Omega[1]=\mathrm{Id}$ in $\underline{\mathrm{CM}}\,\scrR$}\\
&\cong \underline{\Hom}_{\scrR}\big((u,\prod_{i=0}^{m-1}g_i), (u,\prod_{i=1}^{m}g_i)[1]\big)\tag{by above}\\ 
 &\cong \Ext_{\scrR}^1\big((u,\prod_{i=0}^{m-1}g_i),(u,\prod_{i=1}^{m}g_i)\big)\tag{by e.g. \cite{IW3}}.
\end{align*}

We next claim that $\Ext_{\scrR}^1\big((u,\prod_{i=0}^{m-1}g_i),(u,\prod_{i=1}^{m}g_i)\big) \cong (u, g_0g_m)/(u^2, ug_0, ug_m, g_0g_m)$ as $\scrR$-modules, where the right-hand side is the quotient of two fractional ideals.

Similar to (1), applying $\mathbb{G}=\Hom_\scrR\big((u,\prod_{i=0}^{m-1}g_i),-\big)$ to the exact sequence \eqref{303} gives
\[
0 \to \mathbb{G} (u,\prod_{i=1}^{m}g_i) \rightarrow \mathbb{G}\scrR^2  \xrightarrow[]{\begin{psmallmatrix}u \\ g_0 \end{psmallmatrix}}
   \mathbb{G}(u,g_0) \to  \Ext_{\scrR}^1\big((u,\prod_{i=0}^{m-1}g_i),(u,\prod_{i=1}^{m}g_i)\big) \rightarrow 0.
\]
By \cite[5.4]{IW1}, there are isomorphisms
\begin{align*}
    (u,g_m) &\cong \mathbb{G}\scrR\quad\mbox{via } \ r \mapsto(\cdot \frac{r}{u}), \\
     (u, g_0g_m) &\cong \mathbb{G}(u,g_0)\quad\mbox{via }  \ r \mapsto(\cdot \frac{r}{u}).
\end{align*}
Combining these together gives an exact sequence
\begin{equation*}
    (u,g_m)^{\oplus 2} \xrightarrow[]{d=\begin{psmallmatrix} u \\ g_0 \end{psmallmatrix}}  (u, g_0g_m) \rightarrow  \Ext_{\scrR}^1\big((u,\prod_{i=0}^{m-1}g_i),(u,\prod_{i=1}^{m}g_i)\big) \rightarrow 0.
\end{equation*}
Thus $\Ext_{\scrR}^1\big((u,\prod_{i=0}^{m-1}g_i),(u,\prod_{i=1}^{m}g_i)\big) \cong (u, g_0g_m)/\Im d$. It is elementary to check that $\Im d \cong (u^2,ug_0,ug_m,g_0g_m)$, proving the second claim.

Finally, we claim that $(u, g_0g_m)/(u^2, ug_0, ug_m, g_0g_m) \cong \C\lal u, v, x, y \ral / (u,v,g_0,g_m)$ as $\scrR$-modules. Similar to (1), we first define a $\mathbb{C} \lal u, v, x, y \ral$-homomorphism $\varphi$ as follows,
\begin{equation*}
    \varphi \colon \mathbb{C} \lal u, v, x, y \ral \xrightarrow[]{\cdot u}(u, g_0g_m)/(u^2, ug_0, ug_m, g_0g_m).
\end{equation*}
Clearly, $\varphi$ is well defined and $(u,v,g_0,g_m) \subseteq \ker \varphi$. We claim that $\ker \varphi\subseteq (u,v,g_0,g_m)$. 

Let $r \in \mathbb{C} \lal u, v, x, y \ral$ be such that $\varphi(r)=0$.  Then $ru= r_1u^2+r_2g_0u+r_3g_mu+r_4g_0g_m$ for some $r_i \in \mathbb{C} \lal u, v, x, y \ral$.
Thus $(r-r_1u-r_2g_0-r_3g_m)u=r_4g_0g_m$. Since $u$ and $g_0g_m$ have no common factors, we have $r_4 = r_5u$ for some $r_5 \in \mathbb{C} \lal u, v, x, y \ral$. Thus $ru= (r_1u+r_2g_0+r_3g_m+r_5g_0g_m)u$. Since $\mathbb{C} \lal u, v, x, y \ral$ is domain, then $r=r_1u+r_2g_0+r_3g_m+r_5g_0g_m \in (u,v,g_0,g_m)$, and so $\ker \varphi\subseteq (u,v,g_0,g_m)$, proving the claim.  

Since $\varphi$ is evidently surjective, it induces a $\mathbb{C} \lal u, v, x, y \ral$-isomorphism
\begin{equation*}
    \overline{\varphi}  \colon \frac{\mathbb{C} \lal u, v, x, y \ral}{(u,v,g_0,g_m)} \xrightarrow[]{\sim} \frac{(u, g_0g_m)}{(u^2, ug_0, ug_m, g_0g_m)}.
\end{equation*}
It is easy to check that this is also an $\scrR$-module isomorphism.
\end{proof}

The following two lemmas establish how the generalised GV invariants change under the contraction in \ref{cor:factor} and \ref{example:gv2}.
\begin{lemma}\label{lemma:Toda2}
Let $p,q \in \C\lal x,y \ral$. If $p$ and $q$ have a non-unit common divisor, then $\dim_{\C}  \C\lal x,y \ral/(p,q)= \infty$.
\end{lemma}
\begin{proof}
Let $r \in \mathbb{C}\lal x,y \ral$ be a non-unit common divisor of $p$ and $q$, so that $p=rp'$ and $q= rq'$ for some $p',q'\in \C\lal x,y \ral$. Then $(p,q)=(r)(p',q') \subseteq (r)$, hence there is a surjection
\[
\C\lal x,y\ral/(p,q)\twoheadrightarrow \C\lal x,y\ral/(r).
\]
Since $r$ is not a unit,  $\operatorname{dim}_{\C}\C\lal x,y \ral/(r) = \infty$, and so the statement follows.
\end{proof}

\begin{lemma}\label{lemma:Toda}
Let $p_i$ and $q_j \in \C \lal x,y \ral$ for $0 \leq i \leq s $ and $0 \leq j \leq t$. Then
\[
\operatorname{dim}_{\C}\frac{\C\lal x,y\ral}{\big(\prod_{i=0}^s p_i,\ \prod_{j=0}^t q_j\big)}
=
\sum_{i=0}^s\sum_{j=0}^t
\operatorname{dim}_{\C}\frac{\C\lal x,y\ral}{(p_i,q_j)}.
\]
\end{lemma}
\begin{proof}
We split the proof into two cases.

\noindent
(1) There exists $i'$ and $j'$ such that the greatest common divisor $\gcd(p_{i'},q_{j'}) \neq 1$.

Since $ \gcd(p_{i'},q_{j'}) \neq 1$, $\gcd( \prod_{i=0}^s p_i,\prod_{j=0}^t q_j) \neq 1$. By \ref{lemma:Toda2},
\begin{equation*}
    \operatorname{dim}_{\C}\frac{\C\lal x,y \ral}{\big( \prod_{i=0}^s p_i,\prod_{j=0}^t q_j \big)}= \infty = \operatorname{dim}_{\C} \frac{\C \lal x,y \ral}{(p_{i'},q_{j'})}.
\end{equation*}
Since the dimension of a vector space can not be negative, the statement follows.

\noindent
(2) The greatest common divisor $\gcd(p_i,q_j)=1$ for each $i$ and $j$.

It suffices to prove that 
\begin{equation*}
  \operatorname{dim}_{\C}\frac{\C \lal x,y\ral}{(p_0,q_0q_1)}=\operatorname{dim}_{\C} \frac{\C \lal x,y\ral}{(p_0,q_0)} +\operatorname{dim}_{\C}\frac{\C \lal x,y\ral}{(p_0,q_1)},
\end{equation*}
since then the statement follows by induction. 
We first consider the natural quotient $\C \lal x,y\ral$-homomorphism 
\begin{align*}
   \varphi \colon  \frac{\C \lal x,y\ral}{(p_0,q_0q_1)}   \twoheadrightarrow \frac{\C \lal x,y\ral}{(p_0,q_0)}.
\end{align*}
It is clear that $\ker \varphi = (q_0)/(p_0,q_0q_1)$. So we only need to prove that $(q_0)/(p_0,q_0q_1) \cong \C\lal x,y \ral/(p_0,q_1) $. To see this, we define a $\C \lal x,y\ral$-homomorphism as
\begin{align*}
   \upvartheta \colon \frac{\C \lal x,y\ral}{(p_0,q_1)} & \rightarrow  \frac{(q_0)}{(p_0,q_0q_1)} \\
    r & \mapsto q_0r
\end{align*}
It is clear that $\upvartheta$ is well-defined and surjective. So we only need to prove the injectivity. If $q_0r= r_1p_0+r_2q_0q_1$ for some $r_1,r_2 \in \C \lal x,y \ral$, since $\gcd (p_0,q_0)=1$, then $r_1=r_3q_0$ for some $r_3 \in \C \lal x,y \ral$. Since $\C \lal x,y \ral$ is a domain, $r=r_3p_0+r_2q_1 \in (p_0,q_1)$, and so $\upvartheta$ is injective.
\end{proof}
In the viewpoint of \ref{rmk:intersect}, the above \ref{lemma:Toda} reflects the additivity of local intersection multiplicities under taking unions of components.

Consider a contraction of $\uppi^{\F}$ as $\scrX \to \scrY \xrightarrow{\omega} \Spec\scrR$. Then there exists a subset $J\subseteq\{1,\dots,m\}$ indexing the curves contracted by $\omega$, namely
\[
A_1(\omega)=\bigoplus_{j\in J}\Z\langle \Curve_j\rangle.
\]
Write $k:=|J|$ and list $J(1)<\cdots<J(k)$.
By \cite[\S 5]{IW1}, $\scrY$ is given pictorially by

\[
\begin{array}{ccc}
\begin{array}{c}
\scrY
\end{array} &
\begin{array}{c}
\begin{tikzpicture}[xscale=1,yscale=0.6]
\draw[black] (-0.1,-0.04,0) to [bend left=25] (2.1,-0.04,0);
\draw[black] (1.9,-0.04,0) to [bend left=25] (4.1,-0.04,0);
\node at (5.5,0,0) {$\hdots$};
\draw[black] (6.9,-0.04,0) to [bend left=25] (9.1,-0.04,0);
\node at (1,0.6,0) {$\scriptstyle \Curve_{J(1)}$};
\node at (3,0.6,0) {$\scriptstyle \Curve_{J(2)}$};
\node at (8,0.6,0) {$\scriptstyle \Curve_{J(k)}$};
\filldraw [red] (0,0,0) circle (1pt);
\filldraw [red] (2,0,0) circle (1pt);
\filldraw [red] (4,0,0) circle (1pt);
\filldraw [red] (7,0,0) circle (1pt);
\filldraw [red] (9,0,0) circle (1pt);
\node at (0,-0.4,0) {$\scriptstyle g_0\hdots g_{J(1)-1}$};
\node at (2,-0.4,0) {$\scriptstyle g_{J(1)}\hdots g_{J(2)-1}$};
\node at (4,-0.4,0) {$\scriptstyle g_{J(2)}\hdots g_{J(3)-1}$};
\node at (7,-0.4,0) {$\scriptstyle g_{J(k-1)}\hdots g_{J(k)-1}$};
\node at (9,-0.4,0) {$\scriptstyle g_{J(k)}\hdots g_{m}$};
\end{tikzpicture} 
\end{array}
\end{array}
\]

The following shows that the generalised GV invariants of $\uppi^{\F}$ determine those of $\omega$.
\begin{cor}\label{cor:factor}
For $1 \leq s \leq t \leq k$, the generalised GV invariants $N_{st}(\omega)$ associated to the curve class $\Curve_{J(s)}+ \Curve_{J(s+1)}+ \hdots + \Curve_{J(t)}$ in $\scrY$ satisfies 
\[
N_{st}(\omega) = \sum_{i=J(s-1)+1}^{J(s)}\sum_{j=J(t)}^{J(t+1)-1} N_{ij}(\uppi^{\F}).
\]
\end{cor}
\begin{proof}
To unify the proof, set $J(0)=0$ and $J(k+1)=m+1$. Then $g_{J(0)}=g_0$ and $g_{J(k+1)-1}=g_m$.
By the definition of the generalised GV invariants~\ref{def:Nij}, we have
\begin{align*}
N_{st}(\omega) & = \operatorname{dim}_\mathbb{C}\frac{\mathbb{C}\lal x,y \ral}{(g_{J(s-1)} \hdots g_{J(s)-1}, \, g_{J(t)} \hdots g_{J(t+1)-1})}    \\
& =  \sum_{i=J(s-1)}^{J(s)-1}\sum_{j=J(t)}^{J(t+1)-1} \operatorname{dim}_\mathbb{C}\frac{\mathbb{C}\lal x,y \ral}{(g_i,g_j)} \tag{by \ref{lemma:Toda}} \\
& = \sum_{i=J(s-1)+1}^{J(s)}\sum_{j=J(t)}^{J(t+1)-1} N_{ij}(\uppi^{\F}).
\qedhere
\end{align*}
\end{proof}

The following example illustrates \ref{cor:factor} in the case of two exceptional curves.

\begin{example}\label{example:gv2}
Consider the singularity $\scrR$ defined by $uv=g_0g_1g_2$, together with the $\scrR$-module $M= \scrR \oplus (u,g_0) \oplus (u,g_0g_1)$, and the associated crepant partial resolution $\uppi \colon \scrX \rightarrow \Spec \scrR$, which is given pictorially by
\[
\begin{array}{ccc}
\begin{array}{c}
\scrX
\end{array} &
\begin{array}{c}
\begin{tikzpicture}[xscale=0.6,yscale=0.6]
\draw[black] (-0.1,-0.04,0) to [bend left=25] (2.1,-0.04,0);
\draw[black] (1.9,-0.04,0) to [bend left=25] (4.1,-0.04,0);

\node at (1,0.6,0) {$\scriptstyle \Curve_{1}$};
\node at (3,0.6,0) {$\scriptstyle \Curve_{2}$};

\node at (0,-0.4,0) {$\scriptstyle g_0$};
\node at (2,-0.4,0) {$\scriptstyle g_1$};
\node at (4,-0.4,0) {$\scriptstyle g_2$};

\end{tikzpicture} 
\end{array}
\end{array}
\]   
Consider the contractions $\scrX \xrightarrow{} \scrY_i \xrightarrow{\omega_i}\Spec\scrR$ such that $A_1(\omega_i)=\Z\langle\Curve_i\rangle$ for $i=1,2$. By~\cite[\S 5]{IW1}, $\scrY_1$ and $\scrY_2$ are given pictorially by

\[
\begin{array}{cccc}
\begin{array}{c}
\scrY_1
\end{array} &
\begin{array}{c}
\begin{tikzpicture}[xscale=0.6,yscale=0.6]
\draw[black] (-0.1,-0.04) to [bend left=25] (2.1,-0.04);
\node at (1,0.6) {$\scriptstyle \Curve_{1}$};

\node at (0,-0.4) {$\scriptstyle g_0$};
\node at (2,-0.4) {$\scriptstyle g_1g_2$};
\end{tikzpicture} 
\end{array}
\qquad \qquad
\begin{array}{c}
\scrY_2
\end{array} &
\begin{array}{c}
\begin{tikzpicture}[xscale=0.6,yscale=0.6]
\draw[black] (-0.1,-0.04) to [bend left=25] (2.1,-0.04);
\node at (1,0.6) {$\scriptstyle \Curve_{2}$};

\node at (0,-0.4) {$\scriptstyle g_0g_1$};
\node at (2,-0.4) {$\scriptstyle g_2$};
\end{tikzpicture} 
\end{array}
\end{array}
\]
Hence the generalised GV invariants of the single exceptional curve in $\scrY_1$ and $\scrY_2$ are
\begin{align*}
  &  N_{11}(\omega_1) = \operatorname{dim}_\mathbb{C} \frac{\mathbb{C}\lal x,y \ral}{(g_0,g_1g_2)}=\operatorname{dim}_\mathbb{C} \frac{\mathbb{C}\lal x,y \ral}{(g_0,g_1)}+\operatorname{dim}_\mathbb{C} \frac{\mathbb{C}\lal x,y \ral}{(g_0,g_2)} =N_{11}(\uppi)+N_{12}(\uppi),\\
    &  N_{11}(\omega_2) = \operatorname{dim}_\mathbb{C} \frac{\mathbb{C}\lal x,y \ral}{(g_0g_1,g_2)}=\operatorname{dim}_\mathbb{C} \frac{\mathbb{C}\lal x,y \ral}{(g_0,g_2)}+\operatorname{dim}_\mathbb{C} \frac{\mathbb{C}\lal x,y \ral}{(g_1,g_2)} =N_{12}(\uppi)+N_{22}(\uppi).
\end{align*}
\end{example}

The papers \cite{NW1, VG} give a combinatorial description of the matrix which controls the transformation of the non-zero GV invariants under a flop. For crepant resolutions of $cA_n$ singularities, see \S\ref{sec: reduction} below. 
By definition \ref{def:Nij}, it is clear that generalised GV invariants of crepant partial resolutions of $cA_n$ singularities also satisfy this transformation under a flop. We provide only an example to illustrate the general idea.
\begin{example}\label{example:permutation}
We continue the example \ref{example:gv2} and consider the $\uppi_i \colon \scrX_i \rightarrow \Spec\scrR$ obtained by flopping the exceptional curve $\Curve_i$ in $\scrX$ for $i=1,2$. Then by \cite[\S 5]{IW1} $\scrX_1$ and $\scrX_2$ are given pictorially by
\[
\begin{array}{cccc}
\begin{array}{c}
\scrX_1
\end{array} &
\begin{array}{c}
\begin{tikzpicture}[xscale=0.6,yscale=0.6]
\draw[black] (-0.1,-0.04,0) to [bend left=25] (2.1,-0.04,0);
\draw[black] (1.9,-0.04,0) to [bend left=25] (4.1,-0.04,0);

\node at (1,0.6,0) {$\scriptstyle \Curve'_{1}$};
\node at (3,0.6,0) {$\scriptstyle \Curve_{2}$};

\filldraw [red] (0,0,0) circle (1pt);
\filldraw [red] (2,0,0) circle (1pt);
\filldraw [red] (4,0,0) circle (1pt);

\node at (0,-0.4,0) {$\scriptstyle g_1$};
\node at (2,-0.4,0) {$\scriptstyle g_0$};
\node at (4,-0.4,0) {$\scriptstyle g_2$};
\end{tikzpicture} 
\end{array}
\qquad \qquad
\begin{array}{c}
\scrX_2
\end{array} &
\begin{array}{c}
\begin{tikzpicture}[xscale=0.6,yscale=0.6]
\draw[black] (-0.1,-0.04,0) to [bend left=25] (2.1,-0.04,0);
\draw[black] (1.9,-0.04,0) to [bend left=25] (4.1,-0.04,0);

\node at (1,0.6,0) {$\scriptstyle \Curve_{1}$};
\node at (3,0.6,0) {$\scriptstyle \Curve'_{2}$};

\filldraw [red] (0,0,0) circle (1pt);
\filldraw [red] (2,0,0) circle (1pt);
\filldraw [red] (4,0,0) circle (1pt);

\node at (0,-0.4,0) {$\scriptstyle g_0$};
\node at (2,-0.4,0) {$\scriptstyle g_2$};
\node at (4,-0.4,0) {$\scriptstyle g_1$};
\end{tikzpicture} 
\end{array}
\end{array}
\]
Thus by \ref{def:Nij} the generalised GV invariants of $\uppi_i$ are obtained by permutating those of $\uppi$ as follows,
\[
\begin{array}{cc}
  \begin{tikzpicture}[bend angle=30, looseness=1]
\node (a) at (-0.5,0) {$N_{11}(\uppi_1)$};
\node (b) at (1,0) {$N_{22}(\uppi_1)$};
\node (c) at (0.3,-0.6) {$N_{12}(\uppi_1)$};
\node (d) at (1.9,-0.4) {$=$};
\node (e) at (2.7,0) {$N_{11}(\uppi)$};
\node (f) at (4.2,0) {$N_{12}(\uppi)$};
\node (g) at (3.5,-0.6) {$N_{22}(\uppi)$};
\end{tikzpicture}
& \qquad
\begin{tikzpicture}[bend angle=30, looseness=1]
\node (a) at (-0.5,0) {$N_{11}(\uppi_2)$};
\node (b) at (1,0) {$N_{22}(\uppi_2)$};
\node (c) at (0.3,-0.6) {$N_{12}(\uppi_2)$};
\node (d) at (1.9,-0.4) {$=$};
\node (e) at (2.7,0) {$N_{12}(\uppi)$};
\node (f) at (4.2,0) {$N_{22}(\uppi)$};
\node (g) at (3.5,-0.6) {$N_{11}(\uppi)$};
\end{tikzpicture}
\end{array}
\]
\end{example}

\begin{remark}\label{rmk:determine}
Given a crepant partial resolution $\uppi$ of a $cA_n$ singularity, by the permutation rule under flops and \ref{cor:factor}, for any morphism $\omega$ obtained from $\uppi$ by a sequence of flops and contractions, the generalised GV invariant of any curve class appearing in $\omega$ can be expressed as a finite sum of some of the generalised GV invariants of $\uppi$.
\end{remark}

Recall that $\uppi^{\F}$ is a crepant partial resolution with $m$ exceptional curves and $\Lambda(\uppi^{\F}) \cong \End_{\scrR}(M^{\F})$. Moreover, $\Lambda(\uppi^{\F})$ can be presented as the quiver in \ref{35} with trivial arrow $e_i$ at each vertex $i$.

The following result gives an analogue of Toda's formula for the generalised GV invariants. In the case of smooth flopping contractions (equivalently, crepant resolutions of isolated $cA_n$ singularities),
this formula recovers Toda's formula (see \ref{34}), via \ref{37} below.
\begin{theorem}\label{thm:Toda}
For any $1 \leq s  \leq t \leq m$, the following equality holds.
\begin{equation*}
    \operatorname{dim}_{\C}e_s\Lambda_{\mathrm{con}}(\uppi^{\F})e_t= 
    \sum_{i=1}^s \sum_{j=t}^m N_{ij}(\uppi^{\F}) =\operatorname{dim}_{\C}e_t\Lambda_{\mathrm{con}}(\uppi^{\F})e_s.
\end{equation*}
In particular, $ \operatorname{dim}_{\C}\Lambda_{\mathrm{con}}(\uppi^{\F})=\sum_{i=1}^m \sum_{j=i}^m (j-i+1)^2N_{ij}(\uppi^{\F})$.
\end{theorem}
\begin{proof}
To ease notation, set $\uppi\colonequals \uppi^{\F}$. 
We first factor $\uppi$ as $\scrX \xrightarrow{} \scrY \xrightarrow{\omega}\Spec\scrR$ such that $A_1(\omega)=\bigoplus_{k=s}^t \Z\langle\Curve_k\rangle$. By \cite[\S 5]{IW1}, 
$\scrY$ is given pictorially by
\[
\begin{array}{ccc}
\begin{array}{c}
\scrY
\end{array} &
\begin{array}{c}
\begin{tikzpicture}[xscale=0.6,yscale=0.6]
\draw[black] (-0.1,-0.04,0) to [bend left=25] (2.1,-0.04,0);
\draw[black] (1.9,-0.04,0) to [bend left=25] (4.1,-0.04,0);
\node at (5.5,0,0) {$\hdots$};
\draw[black] (6.9,-0.04,0) to [bend left=25] (9.1,-0.04,0);
\node at (1,0.6,0) {$\scriptstyle \Curve_{s}$};
\node at (3,0.6,0) {$\scriptstyle \Curve_{s+1}$};
\node at (8,0.6,0) {$\scriptstyle \Curve_{t}$};
\filldraw [red] (0,0,0) circle (1pt);
\filldraw [black] (2,0,0) circle (1pt);
\filldraw [black] (4,0,0) circle (1pt);
\filldraw [black] (7,0,0) circle (1pt);
\filldraw [red] (9,0,0) circle (1pt);
\node at (0,-0.4,0) {$\scriptstyle g_0\hdots g_{s-1}$};
\node at (2,-0.4,0) {};
\node at (4,-0.4,0) {};
\node at (7,-0.4,0) {};
\node at (9,-0.4,0) {$\scriptstyle g_t\hdots g_{m}$};
\end{tikzpicture} 
\end{array}
\end{array}
\]
and $\Lambda_{\mathrm{con}}(\omega) \cong  e_{st}\Lambda_{\mathrm{con}}(\uppi)e_{st}$ where $e_{st} \colonequals e_s+\cdots +e_t$. Thus
\begin{align*}
    e_s\Lambda_{\mathrm{con}}(\uppi)e_t & \cong e_s e_{st}\Lambda_{\mathrm{con}}(\uppi) e_{st}e_t \tag{since $e_{s}e_{st}=e_s$ and $e_{st}e_{t}=e_t$}\\
    & \cong e_s \Lambda_{\mathrm{con}}(\omega) e_t \tag{since $\Lambda_{\mathrm{con}}(\omega) \cong  e_{st}\Lambda_{\mathrm{con}}(\uppi)e_{st}$}\\
    & \cong \underline{\Hom}_{\scrR}((u,g_0\hdots g_{s-1}),(u,g_0\hdots g_{t-1})) \tag{by \ref{35}} \\
    & \cong \frac{\mathbb{C} \lal u, v, x, y \ral}{(u,v,g_0\hdots g_{s-1},g_t\hdots g_{m})},
\end{align*}
where in the last step uses the first isomorphism in \ref{lemma:con_iso}, but with $g_0$ and $g_m$ replaced by $g_0\hdots g_{s-1}$ and $g_t\hdots g_{m}$.
Similarly, 
\begin{align*}
    e_t\Lambda_{\mathrm{con}}(\uppi)e_s & \cong e_t e_{st}\Lambda_{\mathrm{con}}(\uppi) e_{st}e_s \tag{since $e_{t}e_{st}=e_t$ and $e_{st}e_{s}=e_s$}\\
    & \cong e_t \Lambda_{\mathrm{con}}(\omega) e_s \tag{since $\Lambda_{\mathrm{con}}(\omega) \cong  e_{st}\Lambda_{\mathrm{con}}(\uppi)e_{st}$}\\
    & \cong \underline{\Hom}_{\scrR}((u,g_0\hdots g_{t-1}),(u,g_0\hdots g_{s-1})) \tag{by \ref{35}} \\
    & \cong \frac{\mathbb{C} \lal u, v, x, y \ral}{(u,v,g_0\hdots g_{s-1},g_t\hdots g_{m})},
\end{align*}
where again the last step uses the second isomorphism in \ref{lemma:con_iso}, with $g_0$ and $g_m$ replaced by $g_0\hdots g_{s-1}$ and $g_t\hdots g_{m}$.
Combining these together, it follows that
\begin{equation*}
    \operatorname{dim}_{\C}e_s\Lambda_{\mathrm{con}}(\uppi)e_t=  \operatorname{dim}_{\C} \frac{\mathbb{C} \lal  x, y \ral}{(\prod_{i=0}^{s-1}g_{i},\prod_{j=t}^{m}g_{j})}= \operatorname{dim}_{\C}e_t\Lambda_{\mathrm{con}}(\uppi)e_s.
\end{equation*}
Moreover,
\begin{align*}
    \operatorname{dim}_{\C} \frac{\mathbb{C} \lal  x, y \ral}{(\prod_{i=0}^{s-1}g_{i},\prod_{j=t}^{m}g_{j})} & =\sum_{i=0}^{s-1}\sum_{j=t}^m \operatorname{dim}_{\C} \frac{\C \lal x,y \ral}{(g_i,g_j)} \tag{by \ref{lemma:Toda}}\\
    & = \sum_{i=1}^{s}\sum_{j=t}^m \operatorname{dim}_{\C} \frac{\C \lal x,y \ral}{(g_{i-1},g_j)}\\
    & =  \sum_{i=1}^{s}\sum_{j=t}^m  N_{ij}(\uppi) \tag{by definition \ref{def:Nij}}.
\end{align*}
Writing  $N_{ij}=N_{ij}(\uppi)$ and $\Lambda_{\mathrm{con}}=\Lambda_{\mathrm{con}}(\uppi)$ to ease notation, it follows that 
\begin{equation}\label{s_con_t}
    \operatorname{dim}_{\C}e_s\Lambda_{\mathrm{con}}e_t=   \sum_{i=1}^{s}\sum_{j=t}^m  N_{ij}= \operatorname{dim}_{\C}e_t\Lambda_{\mathrm{con}} e_s.
\end{equation}

Now by \ref{35},
\begin{align*}
     \Lambda_{\mathrm{con}} &= 
    \begin{bmatrix}
        e_1\Lambda_{\mathrm{con}}e_1  & e_1\Lambda_{\mathrm{con}}e_2  &\cdots  & e_1\Lambda_{\mathrm{con}}e_m\\
        e_2\Lambda_{\mathrm{con}}e_1  & e_2\Lambda_{\mathrm{con}}e_2  &\cdots  & e_2\Lambda_{\mathrm{con}}e_m\\
        \vdots     & \vdots &  \ddots&  \vdots   \\
        e_m\Lambda_{\mathrm{con}}e_1  & e_m\Lambda_{\mathrm{con}}e_2  &\cdots  & e_m\Lambda_{\mathrm{con}}e_m
    \end{bmatrix},
\end{align*}
so using \eqref{s_con_t}
\begin{align*}
 \operatorname{dim}_{\C} \Lambda_{\mathrm{con}}   & =\begin{bmatrix}
    \sum_{i=1}^1\sum_{j=1}^m  N_{ij} 
        &  \sum_{i=1}^1\sum_{j=2}^m  N_{ij} &\cdots  &  \sum_{i=1}^1\sum_{j=m}^m  N_{ij}\\
          \sum_{i=1}^1\sum_{j=2}^m  N_{ij} &   \sum_{i=1}^2\sum_{j=2}^m  N_{ij}&\cdots  &  \sum_{i=1}^2\sum_{j=m}^m  N_{ij}\\
        \vdots     & \vdots &  \ddots&  \vdots   \\
        \sum_{i=1}^1\sum_{j=m}^m  N_{ij}  &  \sum_{i=1}^2\sum_{j=m}^m N_{ij}  &\cdots  &  \sum_{i=1}^m\sum_{j=m}^m  N_{ij}
    \end{bmatrix}.
\end{align*}
For $1 \leq i \leq j \leq m$, $N_{ij}$ only appears in each entry of the submatrix from row $i$ to row $j$ and column $i$ to column $j $ of the above matrix, and so $N_{ij}$ appears $(j-i+1)^2$ times in $\operatorname{dim}_{\C} \Lambda_{\mathrm{con}}$. Thus $ \operatorname{dim}_{\C}\Lambda_{\mathrm{con}}=\sum_{i=1}^m \sum_{j=i}^m (j-i+1)^2N_{ij}$.
\end{proof}

The following asserts that any algebra isomorphism between contraction algebras can only send $e_i$ to $e_i$ or $e_{m+1-i}$ for $1 \leq i \leq m$.
\begin{prop}\label{372}
Let $\uppi_k \colon \scrX_k \rightarrow \Spec \scrR_k$ be two crepant partial resolutions of $cA_{n_k}$ singularities $\scrR_k$ with $m_k$ exceptional curves for $k=1,2$. If there exists an algebra isomorphism $\phi \colon \Lambda_{\mathrm{con}}(\uppi_1) \xrightarrow{\sim}  \Lambda_{\mathrm{con}}(\uppi_2)$, then $m_1=m_2$ and $\phi$ must belong to one of the following cases:
\begin{enumerate}
    \item $\phi(e_i)=e_i$ for $1 \leq i \leq m$,
    \item $\phi(e_i)=e_{m+1-i}$ for $1 \leq i \leq m$,
\end{enumerate}
where $m \colonequals m_1=m_2$.
\end{prop}
\begin{proof}
Write $\mod \Lambda_{\mathrm{con}}(\uppi_k)$ for the category of finitely generated right $ \Lambda_{\mathrm{con}}(\uppi_k)$-modules for $k=1,2$.
For $1 \leq i \leq m_1$, write $\scrS_i$ for the simple $\Lambda_{\mathrm{con}}(\uppi_1)$-module corresponding to the vertex $i$ in the quiver of $\Lambda_{\mathrm{con}}(\uppi_1)$ (see \cite[\S 5.2]{HW}). Similarly, for $1 \leq i \leq m_2$, write $\scrS_i'$ for the simple $\Lambda_{\mathrm{con}}(\uppi_2)$-module corresponding to the vertex $i$ in the quiver of $\Lambda_{\mathrm{con}}(\uppi_2)$. 
By \cite[2.11]{W1}, for each $1 \leq i \leq m_1$ the simple module $\scrS_i$ corresponds to the $i$-th exceptional curve $\Curve_i$ in $\scrX_1$. The same holds for $\scrS_i'$ and the exceptional curves in $\scrX_2$. 

The algebra isomorphism $\phi$ induces an equivalence $\varphi \colon \mod \Lambda_{\mathrm{con}}(\uppi_1) \xrightarrow{\sim} \mod \Lambda_{\mathrm{con}}(\uppi_2)$. By Morita theory, $m_1=m_2$ and $\varphi$ maps simple modules to simple modules, and furthermore there is a $\upsigma$ in the symmetric group $\mathfrak{S}_{m}$ such that $\varphi(\scrS_i)= \scrS'_{\upsigma(i)}$ where $m \colonequals m_1=m_2$.

We next use the intersection diagram of exceptional curves for a $cA_n$ singularity—namely, a Dynkin diagram of type~$A$, obtained from the diagram in \ref{35} by removing the vertex~$0$—together with the correspondence between $\scrS_i$, $\scrS_i'$ and exceptional curves, to constrain the permutation $\sigma$.

Since $\uppi_1$ is a crepant partial resolution of a $cA_{n_1}$ singularity, $\scrS_2$ is the unique simple $\Lambda_{\mathrm{con}}(\uppi_1)$-module other than $\scrS_1$ that satisfies $\Ext^1_{\Lambda_{\mathrm{con}}(\uppi_1)}(\scrS_1,\scrS_2) \neq 0$ by \ref{35} and the intersection theory of \cite[2.15]{W1}. Since $\mod \Lambda_{\mathrm{con}}(\uppi_1)$ is equivalent to $\mod \Lambda_{\mathrm{con}}(\uppi_2)$, there exists unique simple $\Lambda_{\mathrm{con}}(\uppi_2)$-module $\scrT$ other than $\scrS'_{\sigma(1)}$ such that $\Ext^1_{\Lambda_{\mathrm{con}}(\uppi_2)}(\scrS'_{\upsigma(1)},\scrT) \neq 0$. Thus the exceptional curve $\upsigma(1)$ in $\uppi_2$ must be a edge curve, by \ref{35} and the intersection theory of \cite[2.15]{W1}. Thus $\upsigma(1)=1$ or $m$. We split the proof into two cases.

\noindent
(1) $\upsigma(1)=1$. 

Since $\Ext^1_{\Lambda_{\mathrm{con}}(\uppi_1)}(\scrS_1,\scrS_2) \neq 0$ and $\mod \Lambda_{\mathrm{con}}(\uppi_1)$ is equivalent to $\mod \Lambda_{\mathrm{con}}(\uppi_2)$, we have $\Ext^1_{\Lambda_{\mathrm{con}}(\uppi_2)}(\scrS'_{\upsigma(1)},\scrS'_{\upsigma(2)}) \neq 0$, and so $\Ext^1_{\Lambda_{\mathrm{con}}(\uppi_2)}(\scrS'_{1},\scrS'_{\upsigma(2)}) \neq 0$.
Thus the curve $\upsigma(2)$ in $\uppi_2$ must be connected to the curve $\upsigma(1)=1$, and so $\upsigma(2)=2$ by \ref{35} and the intersection theory of \cite[2.15]{W1}.
Repeating the same process, we can prove $\upsigma(i)=i$, and so  $\varphi(\scrS_i)= \scrS'_{i}$, and furthermore $\phi(e_i)=e_i$ for each $i$.

\noindent
(2) $\upsigma(1)=m$. 

Since $\Ext^1_{\Lambda_{\mathrm{con}}(\uppi_1)}(\scrS_1,\scrS_2) \neq 0$ and $\mod \Lambda_{\mathrm{con}}(\uppi_1)$ is equivalent to $\mod \Lambda_{\mathrm{con}}(\uppi_2)$, we have $\Ext^1_{\Lambda_{\mathrm{con}}(\uppi_2)}(\scrS'_{\upsigma(1)},\scrS'_{\upsigma(2)}) \neq 0$, and so $\Ext^1_{\Lambda_{\mathrm{con}}(\uppi_2)}(\scrS'_{m},\scrS'_{\upsigma(2)}) \neq 0$.
Thus the curve $\upsigma(2)$ in $\uppi_2$ must be connected to the curve $\upsigma(1)=m$, and so $\upsigma(2)=m-1$ by \ref{35} and the intersection theory of \cite[2.15]{W1}.
Repeating the same process, we can prove $\upsigma(i)=m+1-i$, and so  $\varphi(\scrS_i)= \scrS'_{m+1-i}$, and furthermore $\phi(e_i)=e_{m+1-i}$ for each $i$.
\end{proof}

The following strengthens \ref{34} and \ref{thm:Toda}, in that it intrinsically extracts the generalised GV invariants from the contraction algebra, and is new even in the setting of smooth flopping contractions.
\begin{lemma}\label{cor: Nij}
For any $1 \leq i \leq j \leq m$, the following equality holds.
\begin{equation*}
    N_{ij}(\uppi^{\F}) = \operatorname{dim}_{\C}e_i\left( \frac{\Lambda_{\mathrm{con}}(\uppi^{\F})}{\langle e_1, e_2,\dots, e_{i-1}, e_{j+1},e_{j+2},\dots ,e_m \rangle} \right)e_j.
\end{equation*}
\end{lemma}
\begin{proof}
When $i=1$ and $j=m$, 
\begin{align*}
N_{1m}(\uppi^{\F})= &\operatorname{dim}_{\C}\frac{\mathbb{C} \lal x, y \ral}{(g_0,g_m)} \tag{by the definition \ref{def:Nij} of $N_{ij}(\uppi^{\F})$ }\\      =  & \operatorname{dim}_{\C}\underline{\Hom}_{\scrR}\big((u,g_0),(u,g_0\hdots g_{m-1})\big) \tag{by \ref{lemma:con_iso}} \\
= &  \operatorname{dim}_{\C}\underline{\Hom}_{\scrR}(M_{I_1},M_{I_m}) \tag{since $M_{I_1}=(u,g_0)$ and $M_{I_m}=(u,g_0\hdots g_{m-1})$}\\
= & \operatorname{dim}_{\C}e_1\Lambda_{\mathrm{con}}(\uppi^{\F})e_m. \tag{by \ref{35}}
\end{align*}
Thus the statement holds. When $i \neq 1$ or $j \neq m$, we factor $\uppi^{\F}$ as $\scrX \xrightarrow{\omega} \scrY \xrightarrow{} \Spec\scrR$ such that $A_1(\omega)=\bigcup_{k=i}^j\Z\langle\Curve_k\rangle$. By \cite[\S 5]{IW1}, 
$\scrY$ is given pictorially by
\[
\begin{array}{ccc}
\begin{array}{c}
\scrY
\end{array} &
\begin{array}{c}
\begin{tikzpicture}[xscale=0.6,yscale=0.6]
\draw[black] (-0.1,-0.04,0) to [bend left=25] (2.1,-0.04,0);
\node at (3,0,0) {$\hdots$};
\draw[black] (3.9,-0.04,0) to [bend left=25] (6.1,-0.04,0);
\draw[black] (5.9,-0.04,0) to [bend left=25] (8.1,-0.04,0);
\node at (9,0,0) {$\hdots$};
\draw[black] (9.9,-0.04,0) to [bend left=25] (12.1,-0.04,0);
\node at (1,0.6,0) {$\scriptstyle \Curve_{1}$};
\node at (5,0.6,0) {$\scriptstyle \Curve_{i-1}$};
\node at (7,0.6,0) {$\scriptstyle \Curve_{j+1}$};
\node at (11,0.6,0) {$\scriptstyle \Curve_{m}$};
\filldraw [black] (0,0,0) circle (1pt);
\filldraw [black] (2,0,0) circle (1pt);
\filldraw [black] (4,0,0) circle (1pt);
\filldraw [red] (6,0,0) circle (1pt);
\filldraw [black] (8,0,0) circle (1pt);
\filldraw [black] (10,0,0) circle (1pt);
\filldraw [black] (12,0,0) circle (1pt);
\node at (0,-0.4,0) {};
\node at (2,-0.4,0) {};
\node at (4,-0.4,0) {};
\node at (6,-0.4,0) {$\scriptstyle g_{i-1}g_{i} \dots g_j$};
\node at (9,-0.4,0) {};
\node at (11,-0.4,0) {};
\end{tikzpicture} 
\end{array}
\end{array}
\]
where the red dot labelled $g_{i-1}g_{i}\cdots g_j$ corresponds, complete locally, to the singularity $\scrS \colonequals \C\lal u,v,x,y\ral /(uv - g_{i-1}g_{i}\cdots g_j)$. Here we slightly abuse notation by again using $u,v,x,y$ as local coordinates to define $\scrS$.

Then consider the flat morphism $\Spec\scrS \rightarrow \scrY$, the fibre product $\scrU \colonequals \scrX \times_{\scrY} \Spec\scrS$, and the morphism $\omega|_{\scrU} \colon \scrU \rightarrow \Spec\scrS$. The following picture illustrates the $m=3,\ i=j=2$ case.

\[
\begin{tikzpicture}
    \draw (1.7,0.2) ellipse (60pt and 25pt);
    \draw[black, bend left] (0,0) to (1.2,0);
     \draw[black, bend left] (1,0) to (2.2,0);
     \draw[black, bend left] (2.0,0) to (3.2,0);
     \node (a1) at (0.6,0.4)  {$\scriptstyle\Curve_1$};
     \node (a2) at (1.6,0.4)  {$\scriptstyle\Curve_2$};
     \node (a4) at (2.6,0.4)  {$\scriptstyle\Curve_3$};

\draw[->] (1.7,-1) -- node[left] {$\scriptstyle \omega$} (1.7,-2);
      
     \draw (2,-2.8) ellipse (50pt and 20pt);

 \draw[black, bend left] (1,-3) to (2.2,-3);
     \draw[black, bend left] (2,-3) to (3.2,-3);
     \node (a1) at (1.6,-2.6)  {$\scriptstyle\Curve_1$};
     \node (a2) at (2.6,-2.6)  {$\scriptstyle\Curve_3$};
     \filldraw [red] (2.1,-2.95) circle (1pt);
     \node at (2.1,-3.2,0) {$\scriptstyle g_1g_2$};

      \draw (6,0.3) ellipse (20pt and 12pt); 
     \draw[black, bend left] (5.6,0.1) to (6.4,0.1);
     \node (a2) at (6,0.4)  {$\scriptstyle\Curve_2$};

\draw[->] (6,-0.6) -- node[left] {$\scriptstyle \omega|_{\scrU}$} (6,-2.2);
      
     \draw (6,-2.8) ellipse (12pt and 6pt);
     \draw[fill][red] (6,-2.8) circle (1pt);

    \draw[->] (5.2,-2.8) -- node[above] {} (4,-2.8);
     \draw[->] (4.8,0.2) -- node[left] {} (4,0.2);

    \node (d1) at (-1,0.2)  {$\scrX$};
    \node (d2) at (-0.4,-2.8)  {$\scrY$};
    \node (d3) at (7.2,0.2)  {$\scrU$};
     \node (d4) at (8.7,-2.8)  {$\Spec\scrS \colonequals  \Spec\frac{\C\lal u,v,x,y \ral}{(uv-g_1g_2)} $};

\end{tikzpicture}
\]
By \cite[\S 5]{IW1},
\begin{equation*}
   \Lambda_{\mathrm{con}}(\omega|_{\scrU}) \cong  \Lambda_{\mathrm{con}}(\uppi^{\F})/\langle e_1, e_2,\dots, e_{i-1}, e_{j+1},e_{j+2},\dots ,e_m \rangle.
\end{equation*}
Thus we have
\begin{align*}
N_{ij}(\uppi^{\F})= &\operatorname{dim}_{\C}\frac{\mathbb{C} \lal x, y \ral}{(g_{i-1},g_j)} \tag{by the definition \ref{def:Nij} of $N_{ij}(\uppi^{\F})$ }\\      =  & \operatorname{dim}_{\C}\underline{\Hom}_{\scrS}\big((u,g_{i-1}),(u,g_{i-1}\hdots g_{j-1})\big) \tag{by \ref{lemma:con_iso}} \\
= & \operatorname{dim}_{\C}e_i\Lambda_{\mathrm{con}}(\omega)e_j \tag{by \ref{35}}\\
= & \operatorname{dim}_{\C}e_i \bigl( \Lambda_{\mathrm{con}}(\uppi^{\F})/\langle e_1, e_2,\dots, e_{i-1}, e_{j+1},e_{j+2},\dots ,e_m \rangle \bigr) e_j. \qedhere
\end{align*}
\end{proof}

The following shows that the contraction algebra of a crepant partial resolution of a $cA_n$ singularity determines its associated generalised GV invariants. 
\begin{theorem}\label{371}
Let $\uppi^{\F_k} \colon \scrX^{\F_k} \rightarrow \Spec \scrR_k$ be two crepant partial resolutions of $cA_{n_k}$ singularities $\scrR_k$ with $m_k$ exceptional curves for $k=1,2$. If $\Lambda_{\mathrm{con}}(\uppi^{\F_1}) \cong \Lambda_{\mathrm{con}}(\uppi^{\F_2})$, then $m_1=m_2$ and one of the following cases holds:
\begin{enumerate}
     \item $N_{ij}(\uppi^{\F_1})=N_{ij}(\uppi^{\F_2})$ for $1 \leq i \leq j \leq m$,
    \item $N_{ij}(\uppi^{\F_1})=N_{m+1-j,m+1-i}(\uppi^{\F_2})$ for $1 \leq i \leq j \leq m$,
\end{enumerate}
where $m \colonequals m_1=m_2$.
\end{theorem}

\begin{proof}
To ease notation, set $\uppi_k \colonequals \uppi^{\F_k}$ for $k=1,2$. 
Since $\Lambda_{\mathrm{con}}(\uppi_1) \cong \Lambda_{\mathrm{con}}(\uppi_2)$, $m_1=m_2$ by \ref{372}.
Let $\phi$ be the algebra isomorphism between $\Lambda_{\mathrm{con}}(\uppi_1)$ and $\Lambda_{\mathrm{con}}(\uppi_2)$. By \ref{372}, either $\phi(e_i)=e_i$ or $\phi(e_i)=e_{m+1-i}$ for $1 \leq i \leq m$. Then we split the proof into two cases.

\noindent
(1) $\phi(e_i)=e_i$ for $1 \leq i \leq m$.  In that case, for $1 \leq i \leq j \leq m$, 
\begin{align*}
    N_{ij}(\uppi_1) & \stackrel{\scriptstyle \ref{cor: Nij}}{=}  \operatorname{dim}_{\C}e_i \bigl( \Lambda_{\mathrm{con}}(\uppi_1)/\langle e_1, e_2,\dots, e_{i-1}, e_{j+1},e_{j+2},\dots ,e_m \rangle \bigr) e_j\\
   & = \operatorname{dim}_{\C}e_i \bigl( \Lambda_{\mathrm{con}}(\uppi_2)/\langle e_1, e_2,\dots, e_{i-1}, e_{j+1},e_{j+2},\dots ,e_m \rangle \bigr) e_j\\
   & \stackrel{\scriptstyle \ref{cor: Nij}}{=}N_{ij}(\uppi_2).
\end{align*}

\noindent
(2) $\phi(e_i)=e_{m+1-i}$ for $1 \leq i \leq m$. In that case, for $1 \leq i \leq j \leq m$, 
\begin{align*}
    N_{ij}(\uppi_1)& \stackrel{\scriptstyle \ref{cor: Nij}}{=} \operatorname{dim}_{\C}e_i \bigl( \Lambda_{\mathrm{con}}(\uppi_1)/\langle e_1, e_2,\dots, e_{i-1}, e_{j+1},e_{j+2},\dots ,e_m \rangle \bigr) e_j\\
   & = \operatorname{dim}_{\C}e_{m+1-i}  \bigl( \Lambda_{\mathrm{con}}(\uppi_2)/\langle e_m, e_{m-1},\dots, e_{m-i+2}, e_{m-j},e_{m-j+1},\dots ,e_1 \rangle \bigr) e_{m+1-j}\\
   &  \stackrel{\scriptstyle \ref{thm:Toda}}{=} \operatorname{dim}_{\C}e_{m+1-j}  \bigl( \Lambda_{\mathrm{con}}(\uppi_2)/\langle e_1, e_{2},\dots, e_{m-j}, e_{m-i+2},e_{m-i+3},\dots ,e_m \rangle \bigr) e_{m+1-i} \\
  &  \stackrel{\scriptstyle \ref{cor: Nij}}{=} N_{m+1-j,m+1-i}(\uppi_2). \qedhere
\end{align*}
\end{proof}

\subsection{Classical case: known facts}\label{sec: generalised GV3}
In this subsection, we restrict to those $cA_n$ singularities that admit a crepant resolution, and recall several facts about their NCCRs from \cite{IW1}. These results will be used in \S\ref{subsec:newres} to show that the generalised GV invariants are equivalent to the classical GV invariants.

Recall that in \S\ref{sec: generalised GV1}, every $cA_{t-1}$ singularity $\scrR$ can be written in the form 
\begin{equation*}
    \scrR \cong \frac{\mathbb{C} \lal u, v, x, y \ral}{uv-f_0f_1 \dots f_{n}},
\end{equation*}
where $t$ is the order of the power series $f_0f_1 \dots f_{n}$, and each $f_i$ is a prime element of $\mathbb{C}\lal x,y \ral$.
Moreover, $\scrR$ admits a crepant resolution if and only if each $f_i$ has a linear term by e.g. \cite[5.1]{IW1}. Throughout this subsection, we only consider those $\scrR$ that admit a crepant resolution. In particular, $t=n+1$, and hence $\scrR$ is a $cA_n$ singularity.

Recall from \S\ref{sec: generalised GV1} the notion of the maximal flag $\F$ in the set $\{ 0,1,\hdots, n\}$, and the associated $\CM$ $\scrR$-module $M^{\F}$. Following the notation in \cite[\S 5]{IW1}, we identify maximal flags with elements of the symmetric group $\mathfrak{S}_{n+1}$. 
Thus, for each $\sigma \in \mathfrak{S}_{n+1}$, we regard $\sigma$ as the maximal flag
\[
\{\sigma(0) \}\subset\{\sigma(0),\sigma(1) \}\subset\hdots\subset\{\sigma(0),\hdots,\sigma(n-1)\}.
\]
We define
\[
M^\sigma:=\scrR\oplus\left(\bigoplus_{j=0}^{n-1} M_{\{\sigma(0),\hdots,\sigma(j)\}}\right).
\]

By \cite[5.1]{IW1}, the modules in $(\mathrm{MM}\,\scrR)\cap(\CM\,\scrR)$ are precisely the $M^\sigma$ for $\sigma \in \mathfrak{S}_{n+1}$. Since we assume that $\scrR$ admits a crepant resolution, by \ref{36} there exists a crepant resolution
$\uppi^{\sigma} \colon \scrX^{\sigma} \to \Spec \scrR$
such that $\Lambda(\uppi^{\sigma}) \cong \End_{\scrR}(M^{\sigma})$.

\begin{notation}\label{notation:cAn}
We adopt the following notation.
\begin{enumerate}\label{notation: cAn}
\item Let $k \geq 1$ and consider a $k$-tuple $\mathbf{r}= (r_1,\dots , r_k)$ with $1 \leq r_i \leq n$ for each $i$. Set
\begin{equation*}
    \sigma(\mathbf{r}) \colonequals (r_k,r_k+1) \cdots (r_2,r_2+1)(r_1,r_1+1) \in \mathfrak{S}_{n+1},
\end{equation*}
and define $M^{\mathbf{r}}\colonequals M^{\sigma(\mathbf{r})}$. We write $\uppi^{\mathbf{r}}\colon \scrX^{\mathbf{r}} \rightarrow \Spec\scrR$ for $\uppi^{\sigma(\mathbf{r})}\colon \scrX^{\sigma(\mathbf{r})} \rightarrow \Spec\scrR$.
\item For $1\leq i \leq n$, we write $\uppi^i$, $\scrX^i$ and $M^i$ for $\uppi^{(i)}$, $\scrX^{(i)}$ and $M^{(i)}$ respectively.
\end{enumerate}
\end{notation}

\subsubsection{Reduction Steps for GV invariants}\label{sec: reduction}
This subsection recalls various permutation results from \cite{NW1, VG}, and then shows that GV invariants are suitably local. 

The first reduction step we will use below is to permute the GV invariant of an arbitrary curve class into that of a particular curve class.
From \cite[5.4]{NW1} and \cite[5.10]{VG}, for any $cA_n$ crepant resolution $\uppi$ and any $1 \leq i \leq n$, there is a linear isomorphism 
\begin{equation*}
    F_{i} \colon A_1(\uppi) \rightarrow  A_1(\uppi^{i}),
\end{equation*}
such that $\GV_{\upbeta}(\uppi)= \GV_{|F_{i}(\upbeta)|}(\uppi^{i})$ for any $\upbeta \in A_1(\uppi)$.   
Here we identify $A_1(\uppi)\cong \mathbb{Z}^n\cong A_1(\uppi^{i})$, and so that $F_i$ is an element of $\mathbf{M}_n(\Z)$. Explicitly,
\begin{equation*}
F_i=
\begin{cases}
\mathbf{I}_n -2E_{11}+E_{12},   & \text{ if } i=1\\
\mathbf{I}_n -2E_{nn}+E_{n,n-1},   & \text{ if } i=n\\
\mathbf{I}_n -2E_{ii}+E_{i,i-1}+E_{i,i+1}, & \text{ otherwise }
\end{cases}
\end{equation*}
where $E_{ij} \in \mathbf{M}_n(\Z)$ denotes the standard basis matrix with a $1$ in the $(i,j)$-entry and zeros elsewhere.
Inspired by the above $\GV_{\upbeta}(\uppi)= \GV_{|F_{i}(\upbeta)|}(\uppi^{i})$, we adopt the following notation.

\begin{notation}\label{notation:NW}
For $1 \leq i \leq n$ and $\mathbf{r}$ as in \ref{notation:cAn}, denote $|F_i| \colonequals |-| \circ F_i$ and $|F_{\mathbf{r}}| \colonequals |F_{r_k}| \circ \cdots \circ |F_{r_2}| \circ |F_{r_1}|$. Thus $\GV_{\upbeta}(\uppi)= \GV_{|F_{i}|(\upbeta)}(\uppi^{i})$ and $\GV_{\upbeta}(\uppi)= \GV_{|F_{\mathbf{r}}|(\upbeta)}(\uppi^{\mathbf{r}})$.  
\end{notation}

For $1 \leq i \leq j \leq n$, we write $v_{ij}$ for the vector in $\Z^n$ corresponding to the curve class $\Curve_i+\Curve_{i+1}+ \cdots + \Curve_j$. Thus $v_{ij} = \sum_{k=i}^j \mathsf{e}_k$ where $\mathsf{e}_k$ denotes the $k$-th standard basis vector.

\begin{lemma}\label{lemma:permutate}
With the notation as above, the following holds.
\begin{enumerate}
\item For $2 \leq i \leq j \leq n$, $F_{i-1}v_{ij}=v_{i-1,j}$.
\item For $1 \leq i < j \leq n$, $F_{j}v_{ij}=v_{i,j-1}$.
\item \label{lemma:permutate3} For $1 \leq i \leq j \leq n$, set
\begin{equation*}
    \mathbf{r} = 
    \begin{cases}
        \emptyset \text{ and } F_{\mathbf{r}}=\mathrm{Id}, & \text{ if } i=j=1\\
        (j,j-1,\dots,3, 2), & \text{ if } i=1 \text { and } 2 \leq j \leq n \\
        (i-1,i-2,\dots ,2, 1,j,j-1,\dots,3, 2), & \text{ if } 2 \leq i \leq j \leq n
    \end{cases}
\end{equation*}
then $|F_{\mathbf{r}}|v_{ij}=v_{11}$.
\end{enumerate}
\end{lemma}
\begin{proof}
From the basic facts of linear algebra, we have $E_{ij}\mathsf{e}_t=\begin{cases}
    \mathsf{e}_i, & \text{ if } t=j \\
    \mathbf{0},  & \text{ else }
\end{cases}$.

\noindent
(1) When $3 \leq i \leq j \leq n$, then $F_{i-1}=\mathbf{I}_n -2E_{i-1,i-1}+E_{i-1,i-2}+E_{i-1,i}$, and so
\begin{equation*}
F_{i-1}v_{ij}=  (\mathbf{I}_n -2E_{i-1,i-1}+E_{i-1,i-2}+E_{i-1,i})(\sum_{k=i}^j \mathsf{e}_k)=v_{ij}+\mathsf{e}_{i-1}=v_{i-1,j}.
\end{equation*}
When $2=i \leq j \leq n$, then $F_{i-1}=F_{1}=\mathbf{I}_n -2E_{11}+E_{12}$, and so
\begin{equation*}
   F_{i-1}v_{ij}=F_{1}v_{2j}=(\mathbf{I}_n -2E_{11}+E_{12})(\sum_{k=2}^j \mathsf{e}_k) =v_{2j}+\mathsf{e}_{1}=v_{1j}=v_{i-1,j}.
\end{equation*}

\noindent
(2) When $1 \leq i < j \leq n-1$, then $F_{j}=\mathbf{I}_n -2E_{jj}+E_{j,j-1}+E_{j,j+1}$, and so
\begin{equation*}
    F_jv_{ij}=(\mathbf{I}_n -2E_{jj}+E_{j,j-1}+E_{j,j+1})(\sum_{k=i}^j \mathsf{e}_k)= v_{ij}-2\mathsf{e}_j+\mathsf{e}_j=v_{i,j-1}.   
\end{equation*}
When $1 \leq i < j = n$, then $F_j=F_{n}=\mathbf{I}_n -2E_{nn}+E_{n,n-1}$, and so
\begin{equation*}
   F_jv_{ij}= F_nv_{in}=(\mathbf{I}_n -2E_{nn}+E_{n,n-1})(\sum_{k=i}^n \mathsf{e}_k)= v_{in}-2\mathsf{e}_n+\mathsf{e}_n=v_{i,n-1}=v_{i,j-1}.   
\end{equation*}

\noindent
(3) We only prove the case of $2 \leq i \leq j \leq n$. The other two cases are similar.

By (1), $|F_{1}|\circ |F_{2}| \circ \dots \circ |F_{i-2}|\circ |F_{i-1}| v_{ij}=v_{1j}$. By (2), $|F_{2}|\circ |F_{3}|\circ \dots \circ |F_{j-1}| \circ |F_{j}| v_{1j}=v_{11}$. Thus $|F_{\mathbf{r}}|v_{ij}= |F_{2}|\circ |F_{3}| \circ \dots \circ |F_{j-1}| \circ |F_{j}| \circ |F_{1}|\circ |F_{2}|\circ\dots \circ |F_{i-2}| \circ |F_{i-1}|v_{ij}= v_{11}$.
\end{proof}

The second reduction step will show that the GV invariants are suitably local, following \cite{VG}. More precisely, flopping a curve only affects a neighbourhood of that curve.

Fix integers $s$ and $t$ with $1 \leq s \leq t \leq n$. We factor $\uppi$ as 
\begin{equation*}
  \uppi \colon  \scrX \xrightarrow{\omega} \scrY \xrightarrow{} \Spec\scrR
\end{equation*}
such that $A_1(\omega)=\bigoplus_{k=s}^t \Z\langle  \Curve_k\rangle$.
Let $\Spec S$ be the affine open subset of $\scrY$ containing the singular point, and let $\scrS$ denote the completion of $S$ at the singular point. Then we consider the flat morphism $\Spec\scrS \rightarrow \scrY$, the fibre product $\scrU \colonequals \scrX \times_{\scrY} \Spec\scrS$, and the induced morphism $\omega|_{\scrU} \colon \scrU \rightarrow \Spec\scrS$. 

By abuse of notation, we also write $A_1(\omega|_{\scrU})=\bigoplus_{k=s}^t \Z\langle  \Curve_k\rangle$ for the exceptional curve classes of $\omega|_{\scrU}$. The following picture illustrates the case $n=4$, $s=2$ and $t=3$, where the red dots represent the singular points of $\scrY$ and $\Spec\scrS$.

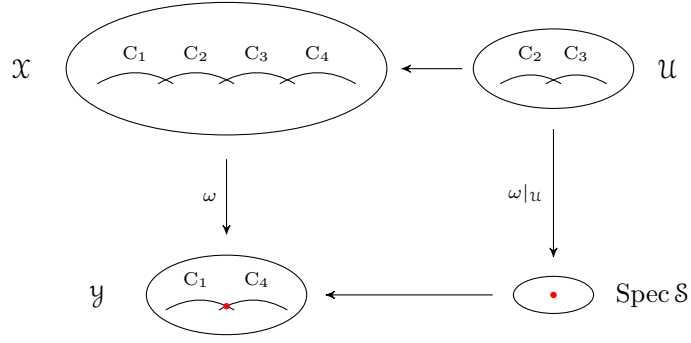
\[
\begin{tikzpicture}
    \draw (1.7,0.2) ellipse (60pt and 25pt);
    \draw[black, bend left] (0,0) to (1,0);
     \draw[black, bend left] (0.8,0) to (1.8,0);
     \draw[black, bend left] (1.6,0) to (2.6,0);
     \draw[black, bend left] (2.4,0) to (3.4,0);
     \node (a1) at (0.5,0.4)  {$\scriptstyle\Curve_1$};
     \node (a2) at (1.3,0.4)  {$\scriptstyle\Curve_2$};
     \node (a3) at (2.1,0.4)  {$\scriptstyle\Curve_3$};
     \node (a4) at (2.9,0.4)  {$\scriptstyle\Curve_4$};

\draw[->] (1.7,-1) -- node[left] {$\scriptstyle \omega$} (1.7,-2);
      
     \draw (1.7,-2.8) ellipse (30pt and 15pt);
     \draw[black, bend left] (0.9,-3) to (1.8,-3);
     \draw[black, bend left] (1.6,-3) to (2.5,-3);
     \node (b1) at (1.3,-2.6)  {$\scriptstyle\Curve_1$};
     \node (b2) at (2.1,-2.6)  {$\scriptstyle\Curve_4$};
    \draw[fill][red] (1.7,-2.95) circle (1pt);

      \draw (6,0.2) ellipse (30pt and 15pt); 
     \draw[black, bend left] (5.3,0) to (6.1,0);
     \draw[black, bend left] (5.9,0) to (6.7,0);
     \node (a2) at (5.7,0.4)  {$\scriptstyle\Curve_2$};
     \node (a3) at (6.3,0.4)  {$\scriptstyle\Curve_3$};

\draw[->] (6,-0.6) -- node[left] {$\scriptstyle \omega|_{\scrU}$} (6,-2.3);
      
     \draw (6,-2.8) ellipse (15pt and 7pt);
     \draw[fill][red] (6,-2.8) circle (1pt);

    \draw[->] (5.2,-2.8) -- node[above] {} (3,-2.8);
     \draw[->] (4.8,0.2) -- node[left] {} (4,0.2);

    \node (d1) at (-1,0.2)  {$\scrX$};
    \node (d2) at (0,-2.8)  {$\scrY$};
    \node (d3) at (7.5,0.2)  {$\scrU$};
     \node (d4) at (7.3,-2.8)  {$\Spec\scrS$};

\end{tikzpicture}
\]

We now prove that flopping a curve only affects a neighbourhood of that curve. Recall from notation~\ref{notation:cAn} that $\scrX^i$ denotes the variety obtained by flopping the exceptional curve $\Curve_i$ in $\scrX$. For $s \leq i \leq t$, we consider the diagram below: flopping $\Curve_i$ in $\scrX$ yields a morphism $\omega^i \colon \scrX^i \to \scrY$, with induced birational map $\upvarphi \colon \scrX^i \dashrightarrow \scrX$. Pulling back $\omega^i$ along $\Spec\scrS \to \scrY$, we obtain the morphism $(\omega|_{\scrU})^i$, and define the birational map
$\upvarphi^i \colonequals (\omega|_{\scrU})^i \circ (\omega|_{\scrU})^{-1}$.

\[
\begin{tikzpicture}[>=stealth,scale=1.2]
\node[black!70!white] (Xi) at (2,0) {$\scrX^i$};
\node[black!70!white] (Ui) at (4,0) {$\scrU^{+}$};

\node[black!70!white] (Y) at (2,-3) {$\scrY$};
\node[black!70!white] (S) at (4,-3) {$\Spec \scrS$};

\node[black!70!white] (X) at (1,-1) {$\scrX$};
\node[black!70!white] (U) at (3,-1) {$\scrU$};

\draw[->] (Ui)--node[above]{}(Xi);
\draw[->] (U)--node[above]{}(X);
\draw[densely dotted,->] (Xi)--node[above,pos=0.6]{$\upvarphi$}(X);
\draw[densely dotted,->] (Ui)--node[above,pos=0.7]{$\upvarphi^i$}(U);
\draw[->] (X)--node[left]{$\omega$}(Y);
\draw[black!70!white,->] (Xi)--node[gap,pos=0.3]{\phantom .}node[right]{$\omega^i$}(Y);
\draw[->] (U)--node[left]{$\omega|_{\scrU}$}(S);
\draw[->] (Ui)--node[right]{$(\omega|_{\scrU})^i$}(S);
\draw[->] (S)--node[above]{}(Y);
\end{tikzpicture}
\]
\begin{lemma}\label{floplocal}
With the diagram above, $(\omega|_{\scrU})^i$ is the flop of $\omega|_{\scrU}$ obtained by flopping the exceptional curve $\Curve_i$ in $\scrU$; that is $\scrU^i \cong \scrU^{+} \cong \scrX^{i} \times_{\scrY} \Spec\scrS$.
\end{lemma}
\begin{proof}
Since $\scrR$ is complete local, there exists a Cartier divisor $D_i$ on $\scrX$ such that $D_i \cdot \Curve_j=\delta_{i j}$ for all $s \leq j \leq t$. Let $\widetilde{D}_i$ denote the proper transform of $D_i$ on $\scrX_i$. Then $\widetilde{D}_i \cdot \Curve_j=-\delta_{i j}$ for all $s \leq j \leq t$. Let $D_i|_{\scrU}$ denote the pull back of $D_i$ to $\scrU$, and similarly  $\widetilde{D}_i |_{\scrU^+}$. Then for all $s \leq j \leq t$ we have
\begin{equation*}
    D_i|_{\scrU}\cdot \Curve_j=\delta_{i j}, \quad \widetilde{D}_i |_{\scrU^+}\cdot \Curve_j=-\delta_{i j}, \text{ and } (\upvarphi^{i})^*D_i|_{\scrU} \cong \widetilde{D}_i |_{\scrU^+}.
\end{equation*}
Hence, by e.g.\ \cite[2.7]{W1}, $(\omega|_{\scrU})^i$ is the flop of $\omega|_{\scrU}$ obtained by flopping the exceptional curve $\Curve_i$ in $\scrU$.
\end{proof}

\begin{lemma}\textnormal{(GV invariants are local)}\label{lemma: GV_local}
With notation as above, we have $\GV_{ij}(\scrX)=\GV_{ij}(\scrU)$ for any $s \leq i \leq j \leq t$.
\end{lemma}
\begin{proof}
\noindent
(1) We first prove that $\GV_{kk}(\scrX)=\GV_{kk}(\scrU)$ for any $s \leq k \leq t$.

Fix $k$ satisfying $s \leq k \leq t$. Consider the following derived equivalences from \cite[3.5.8]{V1},
\begin{align*} 
 & \mathrm{D}^b(\operatorname{coh} \scrX) \xrightarrow{\sim} \mathrm{D}^b(\bmod \Lambda(\omega)), \qquad 
 \mathrm{D}^b(\operatorname{coh} \scrU)  \xrightarrow{\sim} \mathrm{D}^b(\bmod \Lambda(\omega|_{\scrU}))\\
 &\qquad \quad \scrO_{\Curve_k}(-1) \leftrightarrow S_k \qquad \qquad \qquad \qquad \quad \scrO_{\Curve_k}(-1) \leftrightarrow S'_k 
\end{align*}
where $S_k$ denotes the simple-$\Lambda(\omega)$ module that corresponds to $\scrO_{\Curve_k}(-1)$. The $S_k'$ is similar. 

By \cite[5.3]{VG}, $S_k$ (resp.$\ S_k'$) is the only nilpotent point in the moduli space of semisimple $\Lambda(\omega)$ (resp.$\ \Lambda(\omega|_{\scrU})$)-modules of its dimension vector. So, to compare $\GV_{kk}(\scrX)$ and $\GV_{kk}(\scrU)$, it suffices to compare the values of the Behrend function at these two points. 

By \cite{J}, these values only depend on the formal neighbourhood, which can be presented as the Maurer-Cartan locus of their enhancement differential graded algebras $\End_{\Lambda(\omega)}^{\mathrm{dg}}(S_k)$ and $\End_{\Lambda(\omega|_{\scrU})}^{\mathrm{dg}}(S'_k)$ respectively. By \cite{DW1}, these two dg algebras are dg equivalent, via
\begin{equation*}
    \End_{\Lambda(\omega)}^{\mathrm{dg}}(S_k) \cong \End_{\scrX}^{\mathrm{dg}}(\scrO_{\Curve_k}(-1)) \cong \End_{\scrU}^{\mathrm{dg}}( \scrO_{\Curve_k}(-1)) \cong \End_{\Lambda(\omega|_{\scrU})}^{\mathrm{dg}}(S_k'). 
\end{equation*}
Hence the Behrend function values coincide, and therefore we have $\GV_{kk}(\scrX)=\GV_{kk}(\scrU)$ by \cite[5.3]{VG}.

\noindent
(2) We now prove that $\GV_{ij}(\scrX)=\GV_{ij}(\scrU)$ for any $s \leq i \leq j \leq t$.

When $i = j$, the statement holds by (1). Thus we assume $s \leq i < j \leq t$. Set $\mathbf{r} = (j,j-1, \dots, i+1)$. Then by \ref{lemma:permutate}(2), $\GV_{ij}(\scrX)= \GV_{ii}(\scrX^{\mathbf{r}})$ and $\GV_{ij}(\scrU)= \GV_{ii}(\scrU^{\mathbf{r}})$.  
By \ref{floplocal}, we have
$\scrU^{\mathbf{r}} \cong \scrX^{\mathbf{r}} \times_{\scrY} \Spec\scrS$,
and hence $\GV_{ii}(\scrX^{\mathbf{r}})=\GV_{ii}(\scrU^{\mathbf{r}})$ by (1). Therefore $\GV_{ij}(\scrX)=\GV_{ij}(\scrU)$.
\end{proof}

\subsection{Classical case: new results}\label{subsec:newres}
This subsection first shows in \ref{37} that generalised GV invariants are equivalent to the classical GV invariants in the crepant resolution case. Combined with \ref{371}, this yields \ref{thm: determine}, asserting that the contraction algebra of a crepant resolution of a $cA_n$ singularity determines its associated GV invariants. For the isolated $cA_n$ singularities, this is already known from Toda's formula \ref{34} and \cite{HT}; our result extends this to the non-isolated case.

\begin{theorem}\label{37}
Let $\uppi \colon \scrX \to \Spec\scrR$ be a crepant resolution, where $\scrR$ is a $cA_n$ singularity. Then for any $1 \leq i \leq j \leq n$, the following holds.
\begin{enumerate}
\item  $N_{ij}(\uppi) = \infty \iff \GV_{ij}(\uppi)=-1$.
\item  $N_{ij}(\uppi) < \infty \iff \GV_{ij}(\uppi)=N_{ij}(\uppi)$.
\end{enumerate}
\end{theorem}
\begin{proof}
Without loss of generality, we assume
\begin{equation*}
    \scrR \cong \frac{\mathbb{C} \lal u, v, x, y \ral}{uv-f_0f_1 \dots f_{n}},
\end{equation*}
and $M=(u,f_{0}) \oplus (u,f_{0}f_{1}) \oplus \ldots \oplus (u,\prod_{i=0}^{n-1} f_{i})$ such that $\uppi$ is the associated crepant resolution with $\Lambda(\uppi) \cong \End_{\scrR}(M)$ in \ref{36}.

Let $\mathbf{r}$ be the tuple in \ref{lemma:permutate}. Then $\GV_{ij}(\uppi)=\GV_{11}(\uppi^{\mathbf{r}})$ by \ref{lemma:permutate}. Factor $\uppi^{\mathbf{r}}$ as $\scrX^{\mathbf{r}} \xrightarrow{\omega} \scrY \xrightarrow{} \Spec\scrR$ such that $A_1(\omega)=\Z\langle\Curve_1\rangle$. Since $\GV_{11}(\uppi^{\mathbf{r}})$ only depends on $\scrX^{\mathbf{r}}$ and the curve class $\Curve_1$ by \ref{GV}, we have $\GV_{11}(\uppi^{\mathbf{r}})=\GV_{11}(\omega)$, and hence $\GV_{ij}(\uppi)=\GV_{11}(\omega)$.

By \ref{36}, we have $\Lambda(\uppi^{\mathbf{r}})\cong \End_{\scrR}(M^{\mathbf{r}})$, where $M^{\mathbf{r}}=\scrR \oplus (u,f_{i-1}) \oplus (u,f_{i-1}f_{j}) \oplus \ldots \oplus (u,\prod_{i=0}^{n-1} f_{i})$. Using \cite[\S 5]{IW1}, $\scrX^{\mathbf{r}}$ is given pictorially by
\[
\begin{array}{ccc}
\begin{array}{c}
\scrX^{\mathbf{r}}
\end{array} &
\begin{array}{c}
\begin{tikzpicture}[xscale=0.6,yscale=0.6]
\draw[black] (-0.1,-0.04,0) to [bend left=25] (2.1,-0.04,0);
\draw[black] (1.9,-0.04,0) to [bend left=25] (4.1,-0.04,0);
\node at (5.5,0,0) {$\hdots$};
\draw[black] (6.9,-0.04,0) to [bend left=25] (9.1,-0.04,0);
\node at (1,0.6,0) {$\scriptstyle \Curve_{1}$};
\node at (3,0.6,0) {$\scriptstyle \Curve_{2}$};
\node at (8,0.6,0) {$\scriptstyle \Curve_{n}$};
\filldraw [black] (0,0,0) circle (1pt);
\filldraw [black] (2,0,0) circle (1pt);
\filldraw [black] (4,0,0) circle (1pt);
\filldraw [black] (7,0,0) circle (1pt);
\filldraw [black] (9,0,0) circle (1pt);
\node at (0,-0.4,0) {$\scriptstyle f_{i-1}$};
\node at (2,-0.4,0) {$\scriptstyle f_j$};
\node at (4,-0.4,0) {};
\node at (7,-0.4,0) {};
\node at (9,-0.4,0) {};
\end{tikzpicture} 
\end{array}
\end{array}
\]
Since $\uppi^{\mathbf{r}}  \colon \scrX^{\mathbf{r}} \xrightarrow{\omega} \scrY \xrightarrow{} \Spec\scrR$ where $A_1(\omega)=\Z\langle\Curve_1\rangle$, again by \cite[\S 5]{IW1} $\scrY$ is given pictorially by
\[
\begin{array}{ccc}
\begin{array}{c}
\scrY
\end{array} &
\begin{array}{c}
\begin{tikzpicture}[xscale=0.6,yscale=0.6]
\draw[black] (-0.1,-0.04,0) to [bend left=25] (2.1,-0.04,0);
\draw[black] (1.9,-0.04,0) to [bend left=25] (4.1,-0.04,0);
\node at (5.5,0,0) {$\hdots$};
\draw[black] (6.9,-0.04,0) to [bend left=25] (9.1,-0.04,0);
\node at (1,0.6,0) {$\scriptstyle \Curve_{2}$};
\node at (3,0.6,0) {$\scriptstyle \Curve_{3}$};
\node at (8,0.6,0) {$\scriptstyle \Curve_{n}$};
\filldraw [red] (0,0,0) circle (1pt);
\filldraw [black] (2,0,0) circle (1pt);
\filldraw [black] (4,0,0) circle (1pt);
\filldraw [black] (7,0,0) circle (1pt);
\filldraw [black] (9,0,0) circle (1pt);
\node at (0,-0.4,0) {$\scriptstyle f_{i-1}f_j$};
\node at (2,-0.4,0) {};
\node at (4,-0.4,0) {};
\node at (7,-0.4,0) {};
\node at (9,-0.4,0) {};
\end{tikzpicture} 
\end{array}
\end{array}
\]
where the singular point of $\scrY$ is locally $S \colonequals \C[ u,v,x,y ]/(uv-f_{i-1}f_j)$. Write $\scrS$ for the completion of $S$ at the singular point. Consider the flat morphism $\Spec\scrS \rightarrow \scrY$, the fibre product $\scrU \colonequals \scrX^{\mathbf{r}} \times_{\scrY} \Spec\scrS$, and the induced morphism $\omega|_{\scrU} \colon \scrU \rightarrow \Spec\scrS$. Since GV invariants are local by \ref{lemma: GV_local}, we have $\GV_{11}(\omega)= \GV_{11}(\omega|_{\scrU})$, and hence $\GV_{ij}(\uppi)= \GV_{11}(\omega|_{\scrU})$. 

\[
\begin{tikzpicture}
    \draw (1.7,0.2) ellipse (60pt and 25pt);
    \draw[black, bend left] (0,0) to (1,0);
     \draw[black, bend left] (0.8,0) to (1.8,0);
     \draw[black, bend left] (2.4,0) to (3.4,0);
     \node (a1) at (0.5,0.4)  {$\scriptstyle\Curve_1$};
     \node (a2) at (1.3,0.4)  {$\scriptstyle\Curve_2$};
     \node (a3) at (2.1,0.1)  {$\scriptstyle\dots$};
     \node (a4) at (2.9,0.4)  {$\scriptstyle\Curve_n$};

\draw[->] (1.7,-1) -- node[left] {$\scriptstyle \omega$} (1.7,-2);
      
     \draw (2,-2.8) ellipse (50pt and 20pt);
     \draw[black, bend left] (0.9,-3) to (1.8,-3);
     \draw[black, bend left] (2.4,-3) to (3.4,-3);
     \node (b1) at (1.3,-2.6)  {$\scriptstyle\Curve_2$};
     \node (b2) at (2.1,-2.9)  {$\scriptstyle\dots$};
     \node (b3) at (2.9,-2.6)  {$\scriptstyle\Curve_n$};
    \draw[fill][red] (1,-2.95) circle (1pt);

      \draw (6,0.3) ellipse (20pt and 12pt); 
     \draw[black, bend left] (5.6,0.1) to (6.4,0.1);
     \node (a2) at (6,0.4)  {$\scriptstyle\Curve_1$};

\draw[->] (6,-0.6) -- node[left] {$\scriptstyle \omega|_{\scrU}$} (6,-2.2);
      
     \draw (6,-2.8) ellipse (12pt and 6pt);
     \draw[fill][red] (6,-2.8) circle (1pt);

    \draw[->] (5.2,-2.8) -- node[above] {} (4,-2.8);
     \draw[->] (4.8,0.2) -- node[left] {} (4,0.2);

    \node (d1) at (-1,0.2)  {$\scrX^{\mathbf{r}}$};
    \node (d2) at (-0.4,-2.8)  {$\scrY$};
    \node (d3) at (7.2,0.2)  {$\scrU$};
     \node (d4) at (8.7,-2.8)  {$\Spec\scrS \colonequals  \Spec\frac{\C\lal u,v,x,y \ral}{(uv-f_{i-1}f_j)} $};

\end{tikzpicture}
\]

Consider the $\scrS$-module $N\colonequals \scrU \oplus (u,f_{i-1})$. In \ref{36}, $\omega |_{\scrU}$ is the crepant resolution of $\Spec \scrS$ associated to $N$.
Since $\Spec\scrS$ is a $cA_1$ singularity and admits a crepant resolution, then by \cite{R1} there exists a change of coordinates $\varphi$ (possibly different in the two cases below) such that:
\begin{enumerate}
\item $\omega|_{\scrU}$ is a divisor-to-curve contraction $\iff$ $\varphi(f_{i-1})=x=\varphi(f_j)$. 
\item $\omega|_{\scrU}$ is a flop $\iff$ $\varphi(f_{i-1})=x+y^n$ and $\varphi(f_j)=x-y^n$ for some $n \geq 1$.
\end{enumerate}

In case (1), we have $\Lambda_{\mathrm{con}}(\omega|_{\scrU}) \cong  \C\lal y \ral$ by \cite{DW2} and $\GV_{11}(\omega|_{\scrU})=-1$ by \cite{VG}, and so $\GV_{ij}(\uppi)=-1$.  Moreover,
\begin{equation*}
    N_{ij}(\uppi)= \operatorname{dim}_{\mathbb{C}} \frac{\mathbb{C}\lal x,y \ral}{( f_{i-1},f_{j})}= \operatorname{dim}_{\mathbb{C}} \frac{\mathbb{C}\lal x,y \ral}{\big( \varphi(f_{i-1}),\varphi(f_{j})\big)}= \operatorname{dim}_{\mathbb{C}}\C\lal y \ral= \infty.
\end{equation*}

In case (2), we have $\Lambda_{\mathrm{con}}(\omega|_{\scrU}) \cong  \C\lal y \ral/(y^n)$ by \cite{DW2}, and so $\GV_{11}(\omega|_{\scrU})=n$ by \ref{34}. Therefore $\GV_{ij}(\uppi)=n$.
It follows that,
\begin{equation*}
    N_{ij}(\uppi)=\operatorname{dim}_{\mathbb{C}} \frac{\mathbb{C}\lal x,y \ral}{( f_{i-1},f_{j})}  =\operatorname{dim}_{\mathbb{C}} \frac{\mathbb{C}\lal x,y \ral}{\big( \varphi(f_{i-1}),\varphi(f_{j})\big)}=\operatorname{dim}_{\mathbb{C}}\frac{\C\lal y \ral}{(y^n)}=n,
\end{equation*}
and so $N_{ij}(\uppi)=\GV_{ij}(\uppi)$.
\end{proof}

\begin{remark}\label{rmk:Nij}
Given a crepant resolution $\uppi$ of a $cA_n$ singularity, \ref{37} shows that the data of $N_{ij}$ are equivalent to those of $\GV_{ij}$. We will freely pass between them by replacing all occurrences of $-1$ in the GV invariants by $\infty$ in the corresponding generalised GV invariants.
For example,
\[
\begin{tikzpicture}[bend angle=30, looseness=1]
\node (a) at (-3.5,0) {$\GV_{11}$};
\node (b) at (-2.5,0) {$\GV_{22}$};
\node (c) at (-3,-0.6) {$\GV_{12}$};
\node (d) at (-1.6,-0.4) {$=$};
\node (e) at (-1,0) {$1$};
\node (f) at (-0.5,0) {$3$};
\node (g) at (-0.75,-0.6) {$-1$};
\node (h) at (0.5,-0.4) {$\iff$};
\node (e) at (1.5,0) {$N_{11}$};
\node (f) at (2.5,0) {$N_{22}$};
\node (g) at (2,-0.6) {$N_{12}$};
\node (h) at (3.3,-0.4) {$=$};
\node (i) at (4,0) {$1$};
\node (j) at (4.5,0) {$3$};
\node (k) at (4.25,-0.6) {$\infty$};
\end{tikzpicture}
\]
Below, the $N_{ij}$ are mildly easier to control, and they unify statements about the filtration structure in \ref{58} and \ref{stra}.
\end{remark}

\begin{cor}\label{thm: determine}
Let $\uppi_k \colon \scrX_k \rightarrow \Spec \scrR_k$ be two crepant resolutions of $cA_{n}$ singularities $\scrR_k$ for $k=1,2$. If $\Lambda_{\mathrm{con}}(\uppi_1) \cong \Lambda_{\mathrm{con}}(\uppi_2)$, then one of the following cases holds:
\begin{enumerate}
     \item $\GV_{ij}(\uppi_1)=\GV_{ij}(\uppi_2)$ for $1 \leq i \leq j \leq n$,
    \item $\GV_{ij}(\uppi_1)=\GV_{n+1-j,n+1-i}(\uppi_2)$ for $1 \leq i \leq j \leq n$.
\end{enumerate}
\end{cor}
\begin{proof}
This is immediate from \ref{371} and \ref{37}.
\end{proof}

\begin{remark}\label{rmk:ijtobeta}
This subsection has formulated several results using the indices $N_{ij}$ and $\GV_{ij}$. Using the facts below, these statements can be rephrased in terms of $N_{\upbeta}$ and $\GV_{\upbeta}$, as used in the introduction.

Let $\uppi$ be a crepant partial resolution of a $cA_n$ singularity with exceptional curves $\Curve_1,\dots,\Curve_m$, and consider the set of exceptional curve classes
\[
S \colonequals \{\Curve_i+\Curve_{i+1}+ \dots +\Curve_j \mid 1 \leq i \leq j \leq m\}.
\]
Recall that for a curve class $\upbeta=(\upbeta_1,\dots,\upbeta_m)$, its reflective curve class is defined by $\Bar{\upbeta}=(\upbeta_m,\dots,\upbeta_1)$.
\begin{enumerate}
\item By \cite{NW1,VG}, $\GV_{\upbeta}(\uppi)\neq 0$ if and only if $\upbeta \in S$.
\item By Definition~\ref{def:Nij}, $N_{\upbeta}(\uppi)\neq 0$ if and only if $\upbeta \in S$.
\item By the definition of the reflective curve class, $\upbeta \in S$ if and only if $\Bar{\upbeta} \in S$.
\item If $\upbeta=\Curve_i+\Curve_{i+1}+\dots+\Curve_j$, then, using the notation of \ref{intro:Toda}, we have $|\upbeta|=j-i+1$.
\item If $\upbeta=\Curve_i+\Curve_{i+1}+\dots+\Curve_j$, then its reflective curve class is
$\Bar{\upbeta}=\Curve_{m+1-j}+\Curve_{m+2-j}+\dots+\Curve_{m+1-i}$.

\end{enumerate}

Using these facts, the results of this subsection can be rephrased as those stated in the introduction.
\begin{itemize}
\item By (2) and (4), \ref{thm:Toda} implies \ref{intro:Toda}.
\item By (2), (3) and (5), \ref{371} implies \ref{intro:deter}.
\item By (1) and (2), \ref{37} implies \ref{intro: 37}.
\item By (1), (3) and (5), \ref{371} implies \ref{intro:determin}.
\end{itemize}
\end{remark}

\section{Matrices from Potentials}\label{sec:GV}
This section is largely technical: we construct explicit matrices attached to monomialised Type~$A$ potentials on the quiver $Q_n$ and study their determinants.
These determinant functions will later be used in \S\ref{sec: Filtrations} to define a filtration of the parameter space by vanishing loci, encoding the behaviour of the generalised GV invariants.
Readers primarily interested in the applications may skip to \S\ref{sec: Filtrations} on a first reading.

Throughout, we fix an integer $n\geq 1$ and work with monomialised Type~$A$ potentials \eqref{def:TypeA} on the quiver $Q_n$. We recall the quiver and the associated paths appearing in the definition of these potentials:
 \[
\begin{array}{c}
\begin{tikzpicture}[bend angle=15, looseness=1.2]
\node (q) at (-2.5,0)  {$Q_n$};
\node (a) at (-1.5,0) [vertex] {};
\node (b) at (0,0) [vertex] {};
\node (c) at (1.5,0) [vertex] {};
\node (c2) at (2,0) {$\hdots$};
\node (d) at (2.5,0) [vertex] {};
\node (e) at (4,0) [vertex] {};
\node (a1) at (-1.5,-0.2) {$\scriptstyle 1$};
\node (a2) at (0,-0.2) {$\scriptstyle 2$};
\node (a3) at (1.5,-0.2) {$\scriptstyle 3$};
\node (a4) at (2.5,-0.25) {$\scriptstyle n-1$};
\node (a5) at (4,-0.25) {$\scriptstyle n$};
\draw[->,bend left] (a) to node[above] {$\scriptstyle a_{2}$} (b);
\draw[<-,bend right] (a) to node[below] {$\scriptstyle b_{2}$} (b);
\draw[->,bend left] (b) to node[above] {$\scriptstyle a_{4}$} (c);
\draw[<-,bend right] (b) to node[below] {$\scriptstyle b_{4}$} (c);
\draw[->,bend left] (d) to node[above] {$\scriptstyle a_{2n-2}$} (e);
\draw[<-,bend right] (d) to node[below] {$\scriptstyle b_{2n-2}$} (e);
\draw[<-]  (a) edge [in=120,out=55,loop,looseness=10] node[above] {$\scriptstyle a_{1}$} (a);
\draw[<-]  (b) edge [in=120,out=55,loop,looseness=11] node[above] {$\scriptstyle a_{3}$} (b);
\draw[<-]  (c) edge [in=120,out=55,loop,looseness=11] node[above] {$\scriptstyle a_{5}$} (c);
\draw[<-]  (d) edge [in=120,out=55,loop,looseness=11] node[above] {$\scriptstyle a_{2n-3}$} (d);
\draw[<-]  (e) edge [in=120,out=55,loop,looseness=11] node[above] {$\scriptstyle a_{2n-1}$} (e);
\end{tikzpicture}
\end{array}
\]
As before, $b_{2i-1}$ denotes the trivial path $e_i$ at the vertex $i$ for $1 \leq i \leq n$, and we write $\x_i = a_ib_i$ and $\x_i' = b_ia_i$ for $1 \leq i \leq 2n-1$.

\begin{notation}\label{TypeA}
Since \S\ref{sec: Filtrations} and \S\ref{section:obs} study the parameter space of monomialised Type $A$ potentials on $Q_{n}$, we fix the following notation.
\begin{enumerate}
\item Define the set of monomialised Type $A$ potentials on $Q_{n}$
\begin{equation*}
    \MA\colonequals\{\sum_{i=1}^{2n-2}\x_i^{\prime}\x_{i+1} + \sum_{i=1}^{2n-1}\sum_{j \geq 2} k_{ij} \x_i^{j} \mid \text{ all } k_{ij} \in \C\}.
\end{equation*}
\item  Let $\M$ be the corresponding parameter space of coefficients,
\begin{equation*}
    \M \colonequals \{(k_{12},k_{13}, \dots,  k_{2n-1,2},k_{2n-1,3}, \dots)\mid \text{ all } k_{ij} \in \C\}.
\end{equation*}
We identify $\MA$ with $\M$ via the coefficients $(k_{ij})$.

\item Write \emph{$\uk$} for the tuple of \emph{variables} $\uk_{ij}$ for $1 \leq i \leq 2n-1$ and $j \geq 2$, inside the formal power series ring $\C \lal \upkappa_{12},\upkappa_{13},\hdots \upkappa_{2n-1,2},\upkappa_{2n-1,3}, \dots \ral\colonequals \C \lal \upkappa \ral$.

\item For each $i$ and $j$, define $\upvarepsilon_{ij}\colon \MA\to\C$ by
$\upvarepsilon_{ij}(f)\colonequals j\,k_{ij}$.
Via the identification $\M\cong \MA$, we also regard $\upvarepsilon_{ij}$ as a function on $\M$, so that $\upvarepsilon_{ij}(\upkappa)=j\,\upkappa_{ij}$. \label{TypeA 4}
\end{enumerate}
\end{notation}

Given matrices $A = (a_{ij})_{p\times q}$ and $B = (b_{ij})_{s \times t}$ such that $a_{pq}=b_{11}$, define their \emph{gluing} $A \oblong B \in \mathbf{M}_{(p+s-1) \times (q+t-1)}$ by 
\[
A \oblong B \colonequals
\begin{bmatrix}
a_{11}     & a_{12}  & \cdots& a_{1,n-1} & a_{1n} & 0  & \cdots & 0\\
\vdots     & \vdots  &  \ddots&  \vdots &  \vdots & \vdots & \ddots & \vdots \\
a_{p-1,1}  & a_{p-1,2}  & \cdots& a_{p-1,q-1} & a_{p-1,q} & 0 & \cdots & 0\\
a_{p1}     & a_{p2}  & \cdots & a_{p-1,q}& a_{pq} & b_{12} & \cdots & b_{1t}\\
    0      & 0       & \cdots &0         & b_{21} & b_{22} & \cdots & b_{2t}\\
\vdots      & \vdots  &  \ddots &\vdots  &  \vdots & \vdots & \ddots & \vdots \\
    0 & 0 & \cdots &0 & b_{s1} & b_{s2} & \cdots & b_{st}
\end{bmatrix}.
\]

\begin{definition}\label{Aij}
With the $\upvarepsilon_{ij}$ in \textnormal{\ref{TypeA}\eqref{TypeA 4}}, we next define a set of matrices $A_{ij}^d$ for 
\begin{enumerate}\label{Aijd}
    \item $1 \leq  i \leq j \leq 2n-1$, $j-i$ is odd, and $d=2$,
    \item $1 \leq  i \leq j \leq 2n-1$, $j-i$ is even, and $d \geq 2$.
\end{enumerate}
For any $1 \leq i \leq 2n-1$ and $d \geq 2$, define $A_{i,i}^d \colonequals 
\begin{bmatrix} \upvarepsilon_{i,d} \end{bmatrix}$. 
 
For any $1 \leq i \leq 2n-2$, define
$A_{i,i+1}^2 \colonequals
\begin{bmatrix}
    \upvarepsilon_{i,2}  & 1\\
    1 & \upvarepsilon_{i+1,2}
\end{bmatrix}$.

For any $1 \leq i \leq 2n-3$ and $d>2$, define $A_{i,i+2}^d \in \mathbf{M}_{(d+1) \times (d+1)}$ to be
\begin{equation}\label{Aii2}
     A_{i,i+2}^d \colonequals 
    \begin{bmatrix}
        \upvarepsilon_{i,d}  & 0 & 0 &\cdots & 0 & 1 & 0\\
        1 & 1 & 0 &\cdots & 0 & 0 & 0\\
        0 & 1 & 1 &\cdots & 0 & 0 & 0\\
        0 & 0 & 1 &\cdots & 0 & 0 & 0\\
        \vdots     & \vdots & \vdots &  \ddots&  \vdots &  \vdots & \vdots  \\
        0 & 0 & 0 &\cdots & 1 & 0 & 0\\
        0 & 0 & 0 &\cdots & 1 & 0 & 1\\
        0 & 0 & 0 &\cdots & 0 & 1 & \upvarepsilon_{i+2,d}
    \end{bmatrix}.
\end{equation}

The other $A_{ij}^d$ are defined inductively.
For any $i$, $j$ satisfying $j-i \geq 2$, define 
\begin{equation}\label{Aij2}
    A_{i,j}^2 \colonequals A_{i,i+1}^2 \oblong A_{i+1,i+2}^2 \oblong \dots \oblong A_{j-1,j}^2.
\end{equation}

For any $d > 2$, and $i$, $j$ satisfying $j-i \geq 4$ and even, define
\begin{equation}\label{Aijdd}
    A_{i,j}^d \colonequals A_{i,i+2}^d \oblong A_{i+2,i+4}^d \oblong \dots \oblong A_{j-2,j}^d.
\end{equation}

Given $f\in\MA$, we write $A_{ij}^d(f)$ for the matrix obtained from $A_{ij}^d$ by evaluating each symbol $\upvarepsilon_{*,d}$ at $f$, i.e.\ replacing $\upvarepsilon_{t,d}$ by $\upvarepsilon_{t,d}(f)$.
\end{definition}

\begin{remark}
Since $\upvarepsilon_{ij} \colon \MA \rightarrow \C$ in \ref{TypeA}\eqref{TypeA 4}, for each triple $(i,j,d)$ in \ref{Aijd}, the construction gives a map
\begin{align*}
    A_{ij}^d \colon \MA  & \rightarrow \mathbf{M}(\C),\\
                f & \mapsto A_{ij}^d(f)
\end{align*}
where $\mathbf{M}(\C)$ denotes the set of matrices with entries in $\C$. 
Via $\M\cong \MA$ we may also view $A_{ij}^d$ as a map on $\M$, and in particular $A_{ij}^d(\upkappa)$ has entries in $\C\lal \upkappa\ral$, so $\det A_{ij}^d(\upkappa)\in \C\lal \upkappa\ral$.
\end{remark}

\begin{example}\label{example:matrix}
$A_{i,i}^d(\upkappa) = \begin{bmatrix} d\upkappa_{id} \end{bmatrix}$, $A_{i,i+1}^2(\upkappa)  =
\begin{bmatrix}
   2\upkappa_{i,2} & 1\\
    1 & 2\upkappa_{i+1,2}
\end{bmatrix}$, and for $d>2$  
\begin{equation*}
     A_{i,i+2}^d(\upkappa)  = 
    \begin{bmatrix}
       d\upkappa_{id} & 0 & 0 &\cdots & 0 & 1 & 0\\
        1 & 1 & 0 &\cdots & 0 & 0 & 0\\
        0 & 1 & 1 &\cdots & 0 & 0 & 0\\
        0 & 0 & 1 &\cdots & 0 & 0 & 0\\
        \vdots     & \vdots & \vdots &  \ddots&  \vdots &  \vdots & \vdots  \\
        0 & 0 & 0 &\cdots & 1 & 0 & 0\\
        0 & 0 & 0 &\cdots & 1 & 0 & 1\\
        0 & 0 & 0 &\cdots & 0 & 1 & d\upkappa_{i+2,d}
    \end{bmatrix}.
\end{equation*}
\end{example}

Then we consider some subsets of the monomialised Type $A$ potentials $\MA$ on $Q_{n}$.
\begin{notation}\label{TypeAp}
Fix a tuple $\mathbf{p}=(p_1,p_2,\dots,p_{2n-1})$ where each $ 2 \leq p_i \in \N_{\infty}$, we adopt the following notation, which is parallel to that in \ref{TypeA}.
\begin{enumerate}
\item Define the following subset of monomialised Type $A$ potentials on $Q_{n}$
\begin{equation}\label{MAp}
    \MA_{\p}\colonequals\{\sum_{i=1}^{2n-2}\x_i^{\prime}\x_{i+1} + \sum_{i=1}^{2n-1}\sum_{j \geq 2} k_{ij} \x_i^{j} \mid k_{i,j_i}=0 \text{ for } 1 \leq i \leq 2n-1, 2 \leq j_i <p_i\}.
\end{equation}

\item Then set the parameter space $\M_{\p}$ associated to $\MA_{\p}$ to be
\begin{equation}\label{Mp}
    \M_{\p} \colonequals \{(k_{12},k_{13}, \dots,  k_{2n-1,2},k_{2n-1,3}, \dots)\mid k_{i,j_i}=0 \text{ for } 1 \leq i \leq 2n-1, 2 \leq j_i <p_i\}.
\end{equation}

\item Write \emph{$\uk_{\p}$} for the tuple of \emph{variables} $\uk_{ij_i}$, for $1 \leq i \leq 2n-1$ and $p_i \leq j_i \leq \infty$.

\item For any $i$, $j$ satisfying $1 \leq i \leq j \leq 2n-1$, define $d_{ij}(\p)$ to be 
\begin{equation}\label{dp}
d_{ij}(\mathbf{p})\colonequals \left\{
\begin{array}{cl}
        2 & \text{ if } j-i \text{ is odd}\\
        \min \left(p_i, p_{i+2}, \dots , p_j\right) & \text{ if } j-i \text{ is even}
\end{array}\right.
\end{equation}

\item Given another tuple $\p'=(p'_1,p'_2,\dots,p'_{2n-1})$, write $\mathbf{p'} \geq \mathbf{p}$ if $p'_i \geq p_i$ for each $i$. \label{TypeAp 5}
\end{enumerate}
\end{notation}

\begin{remark}\label{rmk: MAp}
We record several elementary remarks concerning the above notation. 
\begin{enumerate}
\item If $\p=(2,2,\dots,2)$, then $\upkappa_{\p}$, $\M_{\p}$ and $\MA_{\p}$ coincide with $\upkappa$, $\M$ and $\MA$, respectively.

\item The natural inclusion $\MA_{\p}\hookrightarrow \MA$ ensures that, for any $f\in \MA_{\p}$ and any triple $(i,j,d)$ as in \ref{Aijd}, the functions $\upvarepsilon_{ij}(f)$ and $A_{ij}^d(f)$ are well defined.

\item Similarly, via the inclusion $\M_{\p}\hookrightarrow \M$, we freely regard the functions $\upvarepsilon_{ij}$ and $A_{ij}^d$ as defined on $\M_{\p}$. In particular,
\[
A_{ij}^d(\upkappa_{\p})\in \mathbf{M}(\C\lal \upkappa_{\p}\ral),
\qquad
\upvarepsilon_{ij}(\upkappa_{\p}),\ \det A_{ij}^d(\upkappa_{\p}) \in \C\lal \upkappa_{\p}\ral .
\]

\item Let $f\in \MA_{\p}$ and write
\[
f=\sum_{i=1}^{2n-2} \x_i' \x_{i+1} + \sum_{i=1}^{2n-1}\sum_{j\ge 2} k_{ij} \x_i^{j}.
\]
If $d<p_i$, then $k_{id}=0$, hence $\upvarepsilon_{i,d}(f)=d k_{id}=0$. Equivalently, the function $\upvarepsilon_{i,d}$ restricts to the zero function on $\MA_{\p}$, and therefore $\upvarepsilon_{i,d}(\upkappa_{\p})=0$. \label{rmk: MAp 4}

\item If $\p' \geq \p$, then $\MA_{\p'}\subseteq \MA_{\p}$ and $\M_{\p'}\subseteq \M_{\p}$. \label{rmk: MAp 5}
\end{enumerate}
\end{remark}
Throughout, we will freely restrict functions defined on $\M$ or $\MA$ to their subsets $\M_{\p}$ and $\MA_{\p}$ without further comment.

The following identities are immediate from the inductive construction of the matrices $A_{ij}^d$ and will be used repeatedly in \S\ref{sec: Filtrations}.
\begin{lemma} \label{517}
Let $1\leq i\leq j\leq 2n-1$ with $j-i\ge 2$.
\begin{enumerate}
\item $\det A_{ij}^2 = \upvarepsilon_{j2} \det A_{i,j-1}^2 - \det A_{i,j-2}^2$, \label{517 1}
\item $\det A_{ij}^2 = \upvarepsilon_{i2} \det A_{i-1,j}^2 - \det A_{i-2,j}^2$. \label{517 2}
\end{enumerate}
If moreover $j-i$ is even, then for any $d>2$ we have 
\begin{enumerate}
\setcounter{enumi}{2}
\item $\det A_{ij}^d= -\det A_{i,j-2}^d+(-1)^{(j-i)(d-1)/2}\upvarepsilon_{j,d}$, \label{517 3}
\item $\det A_{ij}^d= (-1)^{d-1}\det A_{i+2,j}^d+(-1)^{(j-i)/2}\upvarepsilon_{i,d}$. \label{517 4}
\end{enumerate}
\end{lemma}
\begin{proof}
We prove each statement by expanding along the last or first row, using the inductive construction of the matrices.

\noindent
(1) By the inductive definition of $A_{ij}^2$ and $A_{i,j-1}^2$ \eqref{Aij2}, 
\begin{equation*}
    A_{ij}^2 = A_{i,j-1}^2 \oblong A_{j-1,j}^2 \text{ and } \ A_{i,j-1}^2 = A_{i,j-2}^2 \oblong A_{j-2,j-1}^2.
\end{equation*}

Let $v_m\colonequals[0,\dots,0,1]\in \C^{1\times m}$, and write $v_m^T\in \C^{m\times 1}$ for its transpose.
With respect to the block decomposition induced by $\oblong$, the matrix $A_{ij}^2$ takes the form
\begin{align*}
 A_{ij}^2 = \left[ 
\begin{array}{c | c} 
  \begin{array}{c c c} 
      & & \\ 
      & A_{i,j-1}^2 & \\ 
    &  & 
  \end{array} & v_{j-i}^T \\ 
  \hline 
  v_{j-i} & \upvarepsilon_{j2} 
 \end{array} 
\right],  
& & 
 A_{i,j-1}^2 = \left[ 
\begin{array}{c | c} 
  \begin{array}{c c c} 
      & & \\ 
      & A_{i,j-2}^2 & \\ 
    &  & 
  \end{array} & v_{j-i-1}^T \\ 
  \hline 
  v_{j-i-1} & \upvarepsilon_{j-1,2} 
 \end{array} 
\right].
\end{align*}
Let $B$ be the matrix obtained from $A_{ij}^2$ by deleting the last row and the second-to-last column. 
Expanding $\det A_{ij}^2$ along the last row yields $\det A_{ij}^2= \upvarepsilon_{j2}\det A_{i,j-1}^2 - \det B$.
From the above block form, $B$ has the form
\[
B = \left[ 
\begin{array}{c | c} 
  \begin{array}{c c c} 
      & & \\ 
      & A_{i,j-2}^2 & \\ 
    &  & 
  \end{array} & \text{\huge 0} \\ 
  \hline 
  v_{j-i-1} & 1 
 \end{array} 
\right],
\]
and expanding along the last column gives $\det B=\det A_{i,j-2}^2$. This proves~\eqref{517 1}.

\noindent
(2) This follows analogously by expanding along the first row of $A_{ij}^2$.

\noindent
(3) Suppose $j-i$ is even and $d>2$. By definition \eqref{Aijdd}, $ A_{ij}^d = A_{i,j-2}^d \oblong A_{j-2,j}^d.$ Using the explicit form of $A_{i,i+2}^d$ in \eqref{Aii2}, the matrix $A_{ij}^d$ admits the following block decomposition
\begin{equation*}
A_{ij}^d = \left[ 
\begin{array}{c | c} 
  \begin{array}{c c c c} 
      & & \\ 
      & A_{i,j-2}^d & \\ 
    &  & 
  \end{array} & \begin{array}{c  c c c c c  }
      0 & 0  &  \cdots   &0 & 0 &  0 \\ 
      \vdots      & \vdots &  \ddots &  \vdots &  \vdots &  \vdots \\
      0  & 0 &  \cdots  &0 & 0 &  0 \\
     0  & 0 & \cdots  &0& 1 & 0
  \end{array} \\ 
  \hline 
  \begin{array}{c c c c} 
      0 & 0 & \cdots & 1 \\
      0 & 0 & \cdots & 0 \\
      0 & 0 & \cdots & 0 \\ 
      \vdots    &  \vdots &  \ddots&  \vdots  \\
      0 & 0 & \cdots & 0 \\ 
     0 & 0 & \cdots & 0\\
      0 & 0 & \cdots & 0
  \end{array}  & \begin{array}{c c c c c c} 
    \ \ 1  & 0 & \cdots & 0&  0 & 0 \\ 
    \ \ 1 & 1 &     \cdots & 0&0 & 0\\ 
  \ \  0 & 1   & \cdots & 0&0 & 0\\
  \ \ \vdots     &  \vdots&  \ddots&  \vdots  &  \vdots &  \vdots\\
  \ \  0  & 0 & \cdots    & 1&0 & 0\\
  \ \   0 & 0 & \cdots  & 1 &0 &1\\
 \ \    0 & 0 & \cdots  & 0 &1 & \upvarepsilon_{jd}
  \end{array}
 \end{array} 
\right].    
\end{equation*}

Writing $C_{ij}^d$ (resp.\ $D$) for the matrix obtained from $A_{ij}^d$ by deleting the last row and last column (resp.\ the last row and second-to-last column). Expanding $A_{ij}^d$ along the last row yields $\det A_{ij}^d= \upvarepsilon_{jd} \det C_{ij}^d - \det D$. We claim that $\det D = \det A_{i,j-2}^d$ and  $\det C_{ij}^d=(-1)^{(j-i)(d-1)/2}$. So the statement follows.

To see this, by the form of $A_{ij}^d$, 
\begin{equation*}
D = \left[ 
\begin{array}{c | c} 
  \begin{array}{c c c c} 
      & & \\ 
      & A_{i,j-2}^d & \\ 
    &  & 
  \end{array} & \begin{array}{c  c c c c   }
      0 & 0  &  \cdots   &0  &  0 \\ 
      \vdots      & \vdots &  \ddots &  \vdots &   \vdots \\
      0  & 0 &  \cdots  &0  &  0 \\
     0  & 0 & \cdots  &1 & 0
  \end{array} \\ 
  \hline 
  \begin{array}{c c c c} 
      0 & 0 & \cdots & 1 \\
      0 & 0 & \cdots & 0 \\
      0 & 0 & \cdots & 0 \\ 
      \vdots    &  \vdots &  \ddots&  \vdots  \\
      0 & 0 & \cdots & 0 \\ 
     0 & 0 & \cdots & 0\\
  \end{array}  & \begin{array}{c c c c c } 
     1  & 0 & \cdots & 0&  0 \\ 
     1 & 1 &     \cdots & 0& 0\\ 
    0 & 1   & \cdots & 0& 0\\
   \vdots     &  \vdots&  \ddots&  \vdots  &    \vdots\\
    0  & 0 & \cdots    & 1& 0\\
    0 & 0 & \cdots  & 1 &1 
  \end{array}
 \end{array} 
\right].    
\end{equation*}

By repeated expansion along the rows/columns in the lower-right block that contain a single nonzero entry, it follows that $\det D = \det A_{i,j-2}^d$.

By the definition of $C_{ij}^d$, $C_{i,j-2}^d$ and the form of $A_{ij}^d$,
\begin{equation*}
C_{ij}^d = \left[ 
\begin{array}{c | c} 
  \begin{array}{c c c} 
      & & \\ 
      & C_{i,j-2}^d & \\ 
    &  & 
  \end{array} & \begin{array}{c cc c c c   }
     \quad 0 & \quad 0 & 0 &   \cdots & 0 &  0 \\ 
    \quad  \vdots     & \quad \vdots & \vdots &  \ddots&  \vdots &  \vdots  \\
     \quad 0 & \quad 0 & 0 &  \cdots & 0 &  0 \\
    \quad 1&\quad 0  & 0 &  \cdots & 0 & 0
  \end{array} \\ 
  \hline 
  \begin{array}{c c c c} 
      0 & 0 & \cdots & 1 \\
      0 & 0 & \cdots & 0 \\
      0 & 0 & \cdots & 0 \\ 
      \vdots    &  \vdots &  \ddots&  \vdots  \\
      0 & 0 & \cdots & 0 \\ 
     0 & 0 & \cdots & 0
  \end{array}  & \begin{array}{c c c c c c} 
     \upvarepsilon_{j-2,d} & 0 & 0 & \cdots & 0&  1 \\ 
     1 & 1 &   0&  \cdots & 0&0\\ 
    0 & 1 & 1  & \cdots & 0&0\\
   \vdots    &  \vdots &  \vdots&  \ddots&  \vdots  &  \vdots \\
    0  & 0 & 0&\cdots    & 1&0\\
     0 & 0 & 0&\cdots  & 1 &0
  \end{array}
 \end{array} 
\right],    
\end{equation*}
where the lower-right corner block is a $ d\times d$ matrix. 
By repeated expansion along the rows $d-1$ times in the lower-right block that contain a single nonzero entry, it follows that $\det C_{ij}^d = (-1)^{d-1} \det C$ where
\begin{equation*}
C = \left[ 
\begin{array}{c | c} 
  \begin{array}{c c c} 
      & & \\ 
      & C_{i,j-2}^d & \\ 
    &  & 
  \end{array} & \begin{array}{c   }
    0 \\ 
    \vdots  \\
     0 \\
   0
  \end{array} \\ 
  \hline 
  \begin{array}{c c c c} 
      0 & 0 & \cdots & 1 
  \end{array}  & \begin{array}{ c} 
     1 
  \end{array}
 \end{array} 
\right].    
\end{equation*}
Thus $\det C = \det C_{i,j-2}^d$, and so $\det C_{ij}^d = (-1)^{d-1} \det C_{i,j-2}^d$. Since $C_{i,i+2}^d$ is obtained by removing the last row and the last column of $A_{i,i+2}^d$ \eqref{Aii2}, $\det C_{i,i+2}^d = (-1)^{d-1}$. So
\begin{equation*}
    \det C_{ij}^d =   (-1)^{d-1} \det C_{i,j-2}^d=(-1)^{(j-i-2)(d-1)/2} \det C_{i,i+2}^d=(-1)^{(j-i)(d-1)/2}.
\end{equation*}

\noindent
(4) The proof is analogous to~(3), expanding instead along the first row of $A_{ij}^d$.
\end{proof}

\begin{notation}\label{Eij}
Let $1\leq i\leq j\leq 2n-1$ with $j-i$ even.
We adopt the following notation for the ideals in the polynomial ring
$\C[\upvarepsilon_{i,2},\upvarepsilon_{i+2,2},\dots,\upvarepsilon_{j,2}]$.
\begin{enumerate}
\item Let $m_{ij}\colonequals(\upvarepsilon_{i,2},\upvarepsilon_{i+2,2},\dots,\upvarepsilon_{j,2})$.
\item Let $E_{ij}$ be the ideal generated by all degree-two monomials in
$\upvarepsilon_{i,2},\upvarepsilon_{i+2,2},\dots,\upvarepsilon_{j,2}$
except the squares $\upvarepsilon_{i,2}^2,\upvarepsilon_{i+2,2}^2,\dots,\upvarepsilon_{j,2}^2$.
\end{enumerate}
\end{notation}

The following shows that the determinant of $A_{ij}^d$ admits an explicit leading term, with all remaining contributions lying in higher-order ideals. 
\begin{lemma}\label{lemma:matrix_2}
Let $1\leq i\leq j\leq 2n-1$.
\begin{enumerate}
\item If $j-i$ is odd, then $\det A_{ij}^2= (-1)^{(j-i+1)/2} +\upepsilon$, where $\upepsilon \in m_{i,j-1} \cap m_{i+1,j}$. \label{lemma:matrix_2 1}
\item If $j-i$ is even, then $\det A_{ij}^2=(-1)^{(j-i)/2}(\upvarepsilon_{i2}+\upvarepsilon_{i+2,2} + \dots +\upvarepsilon_{j2}) +\upepsilon$ where $\upepsilon \in E_{ij}$. \label{lemma:matrix_2 2}
\item If $j-i$ is even and $d >2$, then $\det A_{ij}^d= (-1)^{(j-i)/2}(\upvarepsilon_{id}+(-1)^d\upvarepsilon_{i+2,d} + \dots +(-1)^{(j-i)d/2}\upvarepsilon_{jd})$. \label{lemma:matrix_2 3}
\end{enumerate}
\end{lemma}
\begin{proof}
\noindent
(1) If $j-i=1$, then by definition $\det A_{ij}^2=-1+\upvarepsilon_{i,2}\upvarepsilon_{i+1,2}$. Since $\upvarepsilon_{i,2}\upvarepsilon_{i+1,2} \in (\upvarepsilon_{i,2}) \cap (\upvarepsilon_{i+1,2}) = m_{i,j-1} \cap m_{i+1,j}$, the statement follows. 

We next prove this statement by induction. 
Fix some $i$, $j$ satisfying $j-i \geq 3$ and odd. Assume that $\det A_{i,j-2}^2= (-1)^{(j-i-1)/2} +\upepsilon'$ where $\upepsilon' \in m_{i,j-3} \cap m_{i+1,j-2}$.
So we have
\begin{align*}
 \det A_{ij}^2 & = \upvarepsilon_{j2} \det A_{i,j-1}^2 - \det A_{i,j-2}^2 \tag{by \ref{517}\eqref{517 1}}\\
 &=  \upvarepsilon_{j2} \det A_{i,j-1}^2  -(-1)^{(j-i-1)/2} -\upepsilon' \tag{by assumption}\\
& = (-1)^{(j-i+1)/2}+\upvarepsilon_{j2} \det A_{i,j-1}^2-\upepsilon'.
\end{align*}

Set $\upepsilon \colonequals \upvarepsilon_{j2} \det A_{i,j-1}^2-\upepsilon'$. So it suffices to prove that $\upepsilon \in m_{i,j-1} \cap m_{i+1,j}$.

Since by definition \eqref{Aij2} $\det A_{i,j-1}^2 \in \C [\upvarepsilon_{i,2},\upvarepsilon_{i+1,2}, \dots, \upvarepsilon_{j-1,2}]$, $\upvarepsilon_{j2} \det A_{i,j-1}^2 \in m_{i+1,j}$. Together with $\upepsilon' \in  m_{i+1,j-2} \subseteq m_{i+1,j}$, it follows that $\upepsilon \in m_{i+1,j}$.
Similarly, we can prove $\upepsilon \in  m_{i,j-1}$ by  $\det A_{ij}^2 = \upvarepsilon_{i2} \det A_{i-1,j}^2 - \det A_{i-2,j}^2$ in \ref{517}\eqref{517 2}.
So $\upepsilon \in m_{i,j-1} \cap m_{i+1,j}$.

\noindent
(2),\,(3) These are similar, by \ref{517} and induction.
\end{proof}

\begin{prop}\label{518}
Let $f \in \MA$ and write
\begin{equation*}
    f= \sum_{i=1}^{2n-2}\x_i^{\prime}\x_{i+1} + \sum_{i=1}^{2n-1}\sum_{j=2}^{\infty} k_{ij} \x_i^{j}.
\end{equation*}
For any $1 \leq i \leq j \leq 2n-1$ such that $j-i$ is odd, the following holds. 
\begin{enumerate}
\item If $k_{t2}=0$ for $t=i, i+2,\dots , j-1$, then $\det {A}_{ij}^2(f)=(-1)^{(j-i+1)/2}$.
\item If $k_{t2}=0$ for $t=i+1, i+3,\dots , j$, then $\det {A}_{ij}^2(f)=(-1)^{(j-i+1)/2}$.
\end{enumerate}
In particular, given some $\mathbf{p}$ satisfying $d_{i,j-1}(\mathbf{p})>2$ or $d_{i+1,j}(\mathbf{p})>2$, then we have $\det A_{ij}^2(\uk_{\p})=(-1)^{(j-i+1)/2}$.
\end{prop}
\begin{proof}
\noindent
(1) For $t=i, i+2,\dots , j-1$, since $k_{t2}=0$, then $\upvarepsilon_{t2}(f)= 2k_{t2}=0$.
By \ref{lemma:matrix_2}\eqref{lemma:matrix_2 1}, $\det A_{ij}^2(f)= (-1)^{(j-i+1)/2} +\upepsilon(f)$ where $\upepsilon \in m_{i,j-1} \cap m_{i+1,j}$. In particular $\upepsilon$ belongs to the ideal generated by the functions $\upvarepsilon_{i2},\upvarepsilon_{i+2,2},\dots, \upvarepsilon_{j-1,2}$, all of which evaluate at $f$ to be zero. Thus $\upepsilon(f)=0$, and so $\det {A}_{ij}^2(f)=(-1)^{(j-i+1)/2}$. 

\noindent
(2) This is similar.

If $d_{i,j-1}(\mathbf{p})>2$, then  by \eqref{dp} $p_i,p_{i+2}, \dots , p_{j-1} > 2$. If further $f \in \MA_{\p}$, then $k_{t2}=0$ for $t=i, i+2,\dots , j-1$ by \eqref{MAp}, and so by (1) $\det {A}_{ij}^2(f)=(-1)^{(j-i+1)/2}$. Since $f$ is an arbitrary potential in $\MA_{\p}$, $\det A_{ij}^2(\upkappa_{\p})=(-1)^{(j-i+1)/2}$. Similarly, if $d_{i+1,j}(\mathbf{p})>2$, then by (2) $\det A_{ij}^2(\upkappa_{\p})=(-1)^{(j-i+1)/2}$.
\end{proof}

Recall the notation $\uk_{\p}$, $d_{ij}(\mathbf{p})$ in \ref{TypeAp}, and $\det A_{ij}^d(\uk_{\p})$ in \ref{rmk: MAp}.  The following is the main technical result of this subsection. It will be used in \S\ref{sec: Filtrations} below to construct a filtration structure on $\M_\p$ (for some fixed $\p$) with respect to the generalised GV invariant of some chosen curve class $\Curve_{i}+ \hdots +\Curve_{j}$. The zero locus of the polynomial $\det A_{ij}^{d_{ij}(\p)}(\uk_{\p})\in \C \lal \upkappa_{\mathbf{p}} \ral$ will turn out to be the first strata in the filtration, which motivates proving that this polynomial is nonzero in part (2) below.  Part (1) is more technical, but will be needed for inductive proof in \ref{064}.

\begin{prop} \label{41}
Given some $\mathbf{p}$, and any $i,j,d$ in \textnormal{\ref{Aij}}, then the following holds.
\begin{enumerate}
\item If $d < d_{ij}(\mathbf{p})$, then $\det A_{ij}^d(\uk_{\p})=0\in \C \lal \upkappa_{\mathbf{p}} \ral$. 
\item If $d= d_{ij}(\mathbf{p})$ and $d$ is finite, then $\det A_{ij}^d(\uk_{\p})\neq 0$ in $\C \lal \upkappa_{\mathbf{p}} \ral$.
\end{enumerate}
\end{prop}

\begin{proof}
For any $d\geq 2$, consider two complementary subsets of $S \colonequals \{i,i+2,\dots ,j \}$ 
\begin{align*}
    S_d &\colonequals \{t \in S \mid p_t \leq d \}, \quad  \overline{S_d} \colonequals \{t \in S \mid p_t > d \}.
\end{align*}

Then by \ref{rmk: MAp}\eqref{rmk: MAp 4},
\begin{align}
   t \in   \overline{S_d} \iff \upvarepsilon_{td}(f) = 0 \mbox{ for all } f \in \MA_\p \iff \upvarepsilon_{td}(\uk_\p) \mbox{ is the zero function over } \M_{\p}. \label{Sd}
\end{align}

If $j-i$ is even and $d < d_{ij}(\p)$, then by \eqref{dp} $d < \min(p_i,p_{i+2}, \dots , p_j)$, and so $S_d = \emptyset$, $\overline{S_d}=S$. 
If $j-i$ is even and $d = d_{ij}(\p)$, then by \eqref{dp} $d = \min(p_i,p_{i+2}, \dots , p_j)$, and so  $S_d \neq \emptyset$, $\overline{S_d} \neq S$. 

\noindent
(1) Since $d \geq 2$, the case $d_{ij}(\mathbf{p})=2$ cannot occur. Consequently $d_{ij}(\mathbf{p}) > 2$, and thus $j-i$ must be even by \eqref{dp}. Since $d< d_{ij}(\mathbf{p})$, $\overline{S_d} =S$, and so by \eqref{Sd} $\upvarepsilon_{td}(\uk_{\p}) $ is a zero function for each $t \in S =\{i,i+2,\dots, j\}$.

If furthermore $d > 2$, then
\begin{align*}
 \det A_{ij}^d(\uk_{\p})
 & = (-1)^{(j-i)/2}(\upvarepsilon_{id}(\uk_{\p})+(-1)^d\upvarepsilon_{i+2,d}(\uk_{\p}) + \dots +(-1)^{(j-i)d/2}\upvarepsilon_{jd}(\uk_{\p})) \tag{by \ref{lemma:matrix_2}\eqref{lemma:matrix_2 3}}\\
 & =0. \tag{since $\upvarepsilon_{td}(\uk_{\p}) =0$ for $t=i,i+2,\dots, j$}
\end{align*}

Otherwise, if $d=2$, then
\begin{align*}
\det A_{ij}^d(\uk_{\p}) 
& =\det A_{ij}^2(\uk_{\p})\\
&=(-1)^{(j-i)/2}(\upvarepsilon_{i2}(\uk_{\p})+\upvarepsilon_{i+2,2}(\uk_{\p}) + \dots +\upvarepsilon_{j2}(\uk_{\p}))+\upepsilon(\uk_{\p}) \tag{by \ref{lemma:matrix_2}\eqref{lemma:matrix_2 2}}\\
& =\upepsilon(\uk_{\p}), \tag{since $\upvarepsilon_{t2}(\uk_{\p}) =0$ for $t=i,i+2,\dots, j$}
\end{align*}
where $\upepsilon \in E_{ij}$ and $E_{ij}$ is the ideal generated by some degree two terms of $\upvarepsilon_{i2}$, $\upvarepsilon_{i+2,2}, \dots, \upvarepsilon_{j2}$. Since $\upvarepsilon_{t2}(\uk_{\p}) =0$ for $t=i,i+2,\dots, j$, $\upepsilon(\uk_{\p}) =0$, and so $\det A_{ij}^d(\uk_{\p})=0$.

\noindent
(2) We split the proof into cases.

\noindent
(i) $j-i$ is odd, $d=d_{ij}(\mathbf{p})$ and finite.

Since $j-i$ is odd, $d=d_{ij}(\mathbf{p}) = 2$ by \eqref{dp}. Thus by \ref{lemma:matrix_2}\eqref{lemma:matrix_2 1}, 
\begin{equation*}
    \det A_{ij}^d(\uk_{\p}) =\det A_{ij}^2 (\uk_{\p})= (-1)^{(j-i+1)/2} +\upepsilon(\uk_{\p}),
\end{equation*}
where $\upepsilon \in m_{i,j-1}$ and $m_{i,j-1}$ is the ideal generated by $\upvarepsilon_{i,2},\upvarepsilon_{i+2,2} \dots, \upvarepsilon_{j-1,2}$. Since by \ref{rmk: MAp}\eqref{rmk: MAp 4} $\upvarepsilon_{t2}(\uk_{\p})$ is either $2\uk_{t2}$ or zero for any $t$, $\upepsilon(\uk_{\p}) \in (\uk_{\p})$, and so $\det A_{ij}^d (\uk_{\p})$ is a non-zero polynomial.

\noindent
(ii) $j-i$ is even, $d= d_{ij}(\mathbf{p})  >2$ and finite.

Since $j-i$ is even and $d >2$, 
\begin{align*}
\det A_{ij}^d(\uk_{\p})&= (-1)^{(j-i)/2}(\upvarepsilon_{id}(\uk_{\p})+(-1)^d\upvarepsilon_{i+2,d}(\uk_{\p}) + \dots +(-1)^{(j-i)d/2}\upvarepsilon_{jd}(\uk_{\p}))  \tag{by \ref{lemma:matrix_2}\eqref{lemma:matrix_2 3}} \\
& = (-1)^{(j-i)/2} \sum_{t \in S_d} (-1)^{(t-i)d/2}d\uk_{td}. \tag{by \eqref{Sd}}
\end{align*}
Since $j-i$ is even and $d= d_{ij}(\p)$, $S_d \neq \emptyset$, and so $\det A_{ij}^d(\uk_{\p})$ is a non-zero polynomial.

\noindent
(iii) $j-i$ is even and $d= d_{ij}(\mathbf{p}) =2$. 

Since $j-i$ is even and $d=2$,
\begin{align*}
    \det A_{ij}^d(\uk_{\p})
    & =\det A_{ij}^2(\uk_{\p})\\
 &  = (-1)^{(j-i)/2}(\upvarepsilon_{i2}(\uk_{\p})+\upvarepsilon_{i+2,2}(\uk_{\p}) + \dots +\upvarepsilon_{j2}(\uk_{\p})) +\upepsilon(\uk_{\p})  \tag{by \ref{lemma:matrix_2}\eqref{lemma:matrix_2 2}}\\
 & = (-1)^{(j-i)/2} \big(\sum_{t \in S_d} 2\uk_{t2}\big)  +\upepsilon(\uk_{\p}), \tag{by \eqref{Sd}}
\end{align*}
where $\upepsilon \in E_{ij}$ and $E_{ij}$ is the ideal generated by some degree two terms of $\upvarepsilon_{i2}$, $\upvarepsilon_{i+2,2}, \dots, \upvarepsilon_{j2}$. Since by \ref{rmk: MAp}\eqref{rmk: MAp 4} $\upvarepsilon_{t2}(\uk_{\p})$ is either $2\uk_{t2}$ or zero for any $t$, $\upepsilon(\uk_{\p})$ is a degree two term in $\C \lal \uk_{\p} \ral$. Since $j-i$ is even and $d= d_{ij}(\p)$, $S_d \neq \emptyset$, and so $\sum_{t \in S_d} 2\uk_{t2}$ is a non-zero degree one term in $\C \lal \uk_{\p} \ral$. Combining these facts together, it follows that $\det A_{ij}^d(\upkappa_{\p})$ is a non-zero polynomial. 
\end{proof}

\section{Generalised GV Invariants of Potentials and Filtration Structures}\label{sec: Filtrations}
Section~\ref{sec: GGV} introduces generalised GV invariants associated to monomialised Type~A potentials on $Q_n$, paralleling those arising from crepant resolutions of $cA_n$ singularities (see Definition~\ref{def:Nij}).  
In \S\ref{sec: FS}, we study filtration structures on the parameter space of such potentials with respect to these invariants.

\subsection{Generalised GV invariants}\label{sec: GGV}
Motivated by the correspondence between monomialised Type~A potentials on $Q_n$ and crepant resolutions of $cA_n$ singularities established in \ref{12} and \cite[\S 5]{Z}, we define generalised GV invariants for monomialised Type~A potentials via their associated crepant resolutions.

We begin by recalling several results from \cite[\S 5]{Z}. Fix a monomialised Type $A$ potential $f$ on $Q_{n}$
\begin{equation*}
    f=\sum_{i=1}^{2n-2}\x_i^{\prime}\x_{i+1} + \sum_{i=1}^{2n-1}\sum_{j=2}^{\infty}k_{ij}\x_i^j,
\end{equation*}
where each $k_{ij} \in \C$. Consider the following system of equations, where each $g_i \in \C\lal x,y \ral$:
\begin{align}
   g_0+\sum_{j=2}^{\infty}jk_{1j}g_1^{j-1}+g_2&=0 \notag\\
     g_1+\sum_{j=2}^{\infty}jk_{2j}g_2^{j-1}+g_3&=0  \notag\\
 &\vdotswithin{=}  \label{501} \\
   g_{2n-2}+\sum_{j=2}^{\infty}jk_{2n-1,j}g_{2n-1}^{j-1}+g_{2n}&=0.  \notag
\end{align}

Fix an integer $s$ with $0 \leq s \leq 2n-1$, and set $g_s=y$ and $g_{s+1}=x$. 
Then there exist elements $g_0,\dots,g_{2n}$ satisfying \eqref{501}, with each $g_i\in(x,y)\subseteq\mathbb{C}\lal x,y \ral$. Moreover, for all $0\leq i\leq 2n-1$, one has $(g_i,g_{i+1})=(x,y)$.
The following definition mirrors the construction of generalised GV invariants for crepant resolutions (cf.~Definition~\ref{def:Nij}).
\begin{definition}\label{def:gvf}
With notation as above, for any $1 \leq i \leq j \leq n$, define the \emph{generalised GV invariant} $N_{ij}(f)$ by 
\begin{equation*}
    N_{ij}(f)\colonequals \operatorname{dim}_{\mathbb{C}} \frac{\mathbb{C}\lal x,y \ral}{(g_{2i-2},g_{2j})}. 
\end{equation*}
\end{definition}
We then consider the $cA_n$ singularity
\begin{equation*}
\scrR \colonequals \frac{\mathbb{C} \lal u, v, x, y \ral}{uv-g_{0}g_{2} \dots g_{2n}},
\end{equation*}
and the $\scrR$-module 
\[
M\colonequals \scrR \oplus (u,g_{0}) \oplus (u,g_{0}g_{2}) \oplus \ldots \oplus (u,\prod_{i=0}^{n-1} g_{2i}) \in  (\mathrm{MM}\,\scrR)\cap(\CM \, \scrR).
\]

In view of the above results \ref{372} and \ref{371}, we introduce the following notation.
\begin{notation}\label{no:ei}
Suppose that $\Lambda_1,\Lambda_2$ are complete quiver algebras of $Q_n$ subject to some relations. Write $e_i$ for the trivial path at vertex $i$ of $Q_n$, and $\varphi \colon\Lambda_1 \xrightarrow[]{\sim}\Lambda_2 $ if $\varphi $ is an algebra isomorphism satisfying $\varphi(e_i)=e_i$ for all vertices $i$.
\end{notation}

By \cite[5.7]{Z} $\underline\End_{\scrR}(M) \cong \Jac (f)$.
Since $( g_i,g_{i+1} ) = ( x,y )$ for $0\leq i \leq 2n-1$, each $g_i$ has a linear term. Hence $\scrR$ admits a crepant resolution by e.g. \cite[5.1]{IW1}. 
Moreover, since $M \in (\mathrm{MM}\,\scrR)\cap(\CM \, \scrR)$, 
it follows from \ref{36} that there exists a crepant resolution $\uppi: \scrX \rightarrow \Spec  \scrR$ with $\Lambda_{\con}(\uppi) \cong\underline\End_{\scrR}(M)$. 

By \ref{35}, $\underline\End_{\scrR}(M)$ and $\Lambda_{\con}(\uppi)$ can be presented as a complete quiver algebra of $Q_n$ with some relations. In this paper, we declare that the $i$th vertex of $\underline\End_{\scrR}(M) \cong \Lambda_{\con}(\uppi)$ is the vertex corresponding to the summand  $(u,\prod_{i=0}^{i-1} g_{2i})$.
Using \cite[\S 5]{IW1} $\scrX$ is given pictorially by
\[
\begin{array}{ccc}
\begin{array}{c}
\scrX
\end{array} &
\begin{array}{c}
\begin{tikzpicture}[xscale=0.6,yscale=0.6]
\draw[black] (-0.1,-0.04,0) to [bend left=25] (2.1,-0.04,0);
\draw[black] (1.9,-0.04,0) to [bend left=25] (4.1,-0.04,0);
\node at (5.5,0,0) {$\hdots$};
\draw[black] (6.9,-0.04,0) to [bend left=25] (9.1,-0.04,0);
\node at (1,0.6,0) {$\scriptstyle \Curve_{1}$};
\node at (3,0.6,0) {$\scriptstyle \Curve_{2}$};
\node at (8,0.6,0) {$\scriptstyle \Curve_{n}$};
\filldraw [red] (0,0,0) circle (1pt);
\filldraw [red] (2,0,0) circle (1pt);
\filldraw [red] (4,0,0) circle (1pt);
\filldraw [red] (7,0,0) circle (1pt);
\filldraw [red] (9,0,0) circle (1pt);
\node at (0,-0.4,0) {$\scriptstyle g_0$};
\node at (2,-0.4,0) {$\scriptstyle g_2$};
\node at (4,-0.4,0) {$\scriptstyle g_4$};
\node at (7,-0.4,0) {$\scriptstyle g_{2n-2}$};
\node at (9,-0.4,0) {$\scriptstyle g_{2n}$};
\end{tikzpicture}
\end{array}
\end{array}
\]
and under this convention, the curve $\Curve_i$ corresponds to the summand $(u,\prod_{i=0}^{i-1} g_{2i})$, and thus the vertex $i$ of $\Lambda_{\con}(\uppi)$.
Moreover, $\Jac(f)  \xrightarrow[]{\sim} \underline\End_{\scrR}(M) \xrightarrow[]{\sim} \Lambda_{\mathrm{con}}(\uppi)$. 

Thus the generalised GV invariant $N_{ij}(f)$ of a monomialised Type A potential $f$ coincides with $N_{ij}(\uppi)$ (see \ref{def:Nij}), where $\uppi$ is its associated crepant resolution. Namely,
\begin{equation}\label{Nst}
    N_{ij}(\uppi)= \operatorname{dim}_{\mathbb{C}} \frac{\mathbb{C}\lal x,y \ral}{(g_{2i-2},g_{2j})} = N_{ij}(f).
\end{equation}
Thus the data of $N_{ij}(f)$ is equivalent to the data of $\GV_{ij}(\uppi)$ in the sense of \ref{37} and \ref{rmk:Nij}. 
So in the rest of this section and \S\ref{section:obs}, we discuss generalised GV invariants of monomialised Type A potentials to reach conclusions about GV invariants of crepant resolutions of $cA_n$ singularities.

Recall that, in order to define $N_{ij}(f)$ in \ref{def:gvf}, we first fix some integer $s$ and set $g_{s}=y$, $g_{s+1}=x$, then solve \eqref{501} to obtain $g_0, g_1,  \dots, g_{2n}$. From this, we define $N_{ij}(f)=\operatorname{dim}_{\mathbb{C}} \mathbb{C}\lal x,y \ral/(g_{2i-2},g_{2j})$.
\begin{lemma}\label{lemma:Nijf}
The generalised GV invariant $N_{ij}(f)$ in \textnormal{\ref{def:gvf}} does not depend on $s$.
\end{lemma}
\begin{proof}
Fix $s$, set $g_s=y$ and $g_{s+1}=x$, and solve \eqref{501} to obtain $g_0,g_1,\dots,g_{2n}$. From this, the above constructs $\scrR$, $\uppi$ such that $\Lambda_{\mathrm{con}}(\uppi) \xrightarrow[]{\sim}  \Jac(f)$.    

We next start with another integer $t$ and set $g'_{t}=y$, $g'_{t+1}=x$, then solve to obtain $g'_0, g'_1,  \dots, g'_{2n}$. Similarly, the above constructs  $\scrR'$, $\uppi'$ such that $\Lambda_{\mathrm{con}}(\uppi') \xrightarrow[]{\sim}  \Jac(f)$. Thus $\Lambda_{\mathrm{con}}(\uppi) \xrightarrow[]{\sim}\Lambda_{\mathrm{con}}(\uppi')$, and so $N_{ij}(\uppi)=N_{ij}(\uppi') $ by \ref{371}. In particular 
\[
\operatorname{dim}_{\mathbb{C}} \mathbb{C}\lal x,y \ral/(g_{2i-2},g_{2j})=N_{ij}(\uppi)=N_{ij}(\uppi')=\operatorname{dim}_{\mathbb{C}} \mathbb{C}\lal x,y \ral/(g'_{2i-2},g'_{2j}),
\]
and so $N_{ij}(f)$ does not depend on $s$.
\end{proof}

\subsection{Filtration Structures}\label{sec: FS}
Fix some $\mathbf{p}$ and consider the bijection map $f \colon \M_\p \rightarrow \mathsf{MA}_{\mathbf{p}} $ under which
\begin{equation}\label{fkp}
    f(\uk_{\p})=\sum_{i=1}^{2n-2}\x_i^{\prime}\x_{i+1} + \sum_{i=1}^{2n-1}\sum_{j\geq 2} \upkappa_{ij} \x_i^{j},
\end{equation}
where  $\upkappa_{i,j_i}=0$ for $ 1 \leq i \leq 2n-1$ and $2 \leq j_i <p_i$.

By considering $\uk_{ij}$ as variables and solving the system of equations \eqref{501}, we can also realise the family of monomialised Type A potentials $f(\uk_{\p})$ over $\M_{\p}$ \eqref{Mp} by a family of crepant resolutions of $cA_n$ singularities over $\M_{\p}$. More precisely, fix some $s$ satisfying $0 \leq s \leq 2n-1$, and set $g_{s}=y$, $g_{s+1}=x$, then solve $g_{0}, g_{1}, \dots, g_{2n}$ by \eqref{501} where each $g_t  \in (  \uk_{\p},x,y ) \subseteq \C \lal \uk_{\p},x,y  \ral$. 

For any $k \in \M_{\p}$, write $g_t(k) \in \C \lal x,y \ral$ for $g_t$ evaluated at $k$, and consider the $cA_n$ singularity 
\begin{equation*}
\scrR_{k} \colonequals \frac{\mathbb{C} \lal u, v, x, y \ral}{uv-g_{0}(k)g_{2}(k) \dots g_{2n}(k)},
\end{equation*}
and the $\scrR_k$-module 
\[
M_k\colonequals \scrR_k \oplus (u,g_{0}(k)) \oplus (u,g_{0}(k)g_{2}(k)) \oplus \ldots \oplus (u,\prod_{i=0}^{n-1} g_{2i}(k)) \in  (\mathrm{MM}\,\scrR_k)\cap(\CM \, \scrR_k).
\]
Similar to \S\ref{sec: GGV}, $\Jac(f(k))  \xrightarrow[]{\sim} \underline\End_{\scrR_k}(M_k) \xrightarrow[]{\sim} \Lambda_{\mathrm{con}}(\uppi_k)$. Thus, if we vary $k$ over the parameter space $\M_{\p}$, the family of crepant resolutions $\uppi_k$ realises the family of potentials $f(\uk_\p)$ over $\M_{\p}$.

Recall that in the above construction, we first fix some integer $s$ satisfying $0 \leq s \leq 2n-1$, then construct $g_{0},g_{1},\dots, g_{2n}$ with $g_s=y$ and $g_{s+1}=x$ to realise $f(\uk_{\p})$.

\begin{notation}\label{notation:hst}
With the fixed $s$ as above, we adopt the following notation in \ref{064}.
\begin{enumerate}
\item Set $(g_{s0},g_{s1},\dots, g_{s,2n}) \colonequals (g_{0},g_{1},\dots, g_{2n}) $. 
\item For $0 \leq t \leq 2n$, set $h_{st} \colonequals g_{st}(\upkappa_\mathbf{p},x,0) \in \C\lal \upkappa_\mathbf{p}, x  \ral$.
\item Given any $h\in \C\lal \upkappa_\mathbf{p}, x  \ral$, write $[h]_i$ for the degree $i$ graded piece with respect to $x$.
\item Write $\scrO_d$ for an element in $\C\lal \upkappa_\mathbf{p}, x  \ral$ that satisfies $[\scrO_d]_i=0$ for each $i < d$.
\item For $1 \leq t \leq 2n-1$, write \emph{$\uk_{t,\p}$} for the tuple of variables $\uk_{ij}$ appearing in $\uk_{\p}$ with $1 \leq i \leq t$.
\end{enumerate}    
\end{notation}

For $0 \leq s \leq 2n-1$, since $g_{ss}=y$, for any $t$ we have $(g_{ss}, g_{st})= (y,g_{st})= (h_{st})$. Thus 
\begin{align}
    N_{ij}(f(\uk_{\p}))&= \operatorname{dim}_{\mathbb{C}} \frac{\mathbb{C}\lal x,y \ral}{(g_{2i-2,2i-2},g_{2i-2,2j})}\tag{by \ref{lemma:Nijf} with $s=2i-2$}\\
    &= \operatorname{dim}_{\mathbb{C}} \frac{\mathbb{C}\lal x,y \ral}{(y,g_{2i-2,2j})}\tag{since $g_{2i-2,2i-2}=y$}\\
    &=\operatorname{dim}_{\mathbb{C}} \frac{\mathbb{C}\lal x\ral}{(h_{2i-2,2j})}\label{Nstf} .
\end{align}

Thus the power series $h_{2i-2,2j}$ completely determines generalised GV invariant $N_{ij}(f(\uk_{\p}))$.
In particular, the lowest degree term (wrt.\ $x$) of $h_{2i-2,2j}$ determines the general value and general position of $N_{ij}(f(\uk_{\p}))$ over the parameter space $\M_{\p}$. 
The following establishes that the lowest degree term can be described by the matrix $A^d_{2i-1,2j-1}(\uk_{\p})$ where $d=d_{2i-1,2j-1}(\p)$ in \ref{TypeAp}.

\begin{prop}\label{064}
Given the family of monomialised Type $A$ potentials $f(\upkappa_{\mathbf{p}})$ \eqref{fkp} on $Q_n$ and with notation in \textnormal{\ref{notation:hst}}, for any $1 \leq s \leq t \leq 2n-1$, we have
\begin{equation*}
    h_{s-1,t+1}= \sum_{i=r}^{\infty} c_i x^{i}
\end{equation*}
for some $r\in\N_{\infty}$ with $r\ge 1$ and each $c_i \in \C \lal \upkappa_\mathbf{p}  \ral$. Moreover, the following holds.
\begin{enumerate}
\item If $d_{st}(\mathbf{p}) = \infty$, then $h_{s-1,t+1}=0$.
\item If $\mathsf{d}\colonequals d_{st}(\mathbf{p}) < \infty$, then $r= \mathsf{d}-1$, and the lowest degree term $(wrt.\ $x$)$ in $h_{s-1,t+1}$ has coefficient $c_r=(-1)^{t-s+1}\det A_{st}^{\mathsf{d}}(\upkappa_\mathbf{p})$. 
\end{enumerate}
\end{prop}
\begin{proof}
Since $h_{s-1,t+1} \in \C\lal \upkappa_\mathbf{p}, x  \ral$ , we first write $h_{s-1,t+1}$ as
\begin{equation}\label{hst0}
    h_{s-1,t+1}= \sum_{i=r_{st}}^{\infty} c_{st,i} x^{i} = \uplambda_{st}x^{r_{st}}+\scrO_{r_{st}+1},
\end{equation}
for some $r_{st} \geq 0$, each $c_{st,i} \in \C \lal \upkappa_\mathbf{p}  \ral$ and $\uplambda_{st} \colonequals c_{st,r_{st}}$.  
Now since the $h$'s are obtained from the $g$'s by evaluating at $y=0$, they must satisfy the same relations as the $g$'s. In particular, by \eqref{501}, 
\begin{align}\label{hst}
   h_{s-1,t-1}+\sum_{j=p_t}^{\infty}j\upkappa_{tj}h_{s-1,t}^{j-1}+h_{s-1,t+1}=0.  
\end{align}
In the equation above, the index $j$ starts at $p_t$ because $\upkappa_{tj}=0$ for $j < p_t$ in $f(\uk_{\p})$ \eqref{fkp}.
Rearranging \eqref{hst} in the case $t=s$, then using the fact that $g_{s-1,s-1}=y$, $g_{s-1,s}=x$ (thus $h_{s-1,s-1}=0$, $h_{s-1,s}=x$), we obtain 
\begin{equation}
    h_{s-1,s+1} =  -h_{s-1,s-1}-\sum_{j=p_s}^{\infty}j\upkappa_{sj}h_{s-1,s}^{j-1}     =  -\sum_{j=p_s}^{\infty}j\upkappa_{sj}x^{j-1}. \label{hst1}
\end{equation}
Next, rearranging \eqref{hst} in the case $t=s+1$ gives
\begin{align}\label{hst2}
h_{s-1,s+2} &= -
h_{s-1,s}-\sum_{j=p_{s+1}}^{\infty}j\upkappa_{s+1,j}h_{s-1,s+1}^{j-1}\notag\\
&= -
x-\sum_{j=p_{s+1}}^{\infty}j\upkappa_{s+1,j}h_{s-1,s+1}^{j-1}.
\end{align}
In the double index of $h_{s-1,*}$, we now induct on the second of the two indices to prove the result. 
We split the remainder of the proof into the following four lemmas (\ref{lemma:h1}, \ref{lemma:h4}, \ref{lemma:h2} and \ref{lemma:h3}).
\end{proof}

\begin{lemma}\label{lemma:h1}
0With notation in \textnormal{\ref{064}}, if $d_{st}(\mathbf{p}) = \infty$, then $h_{s-1,t+1}=0$.
\end{lemma}
\begin{proof}
If $d_{st}(\mathbf{p}) = \infty$, then by \eqref{dp} $t-s$ is even and $\upkappa_{sj}, \upkappa_{s+2,j}, \dots , \upkappa_{tj}=0$ for all $j$. 
In particular, $h_{s -1,s+1}=0$ via \eqref{hst1}. 
Substituting this into \eqref{hst2},  $h_{s-1,s+2}=-x$. Next, rearranging \eqref{hst} in the case $t=s+2$ gives
\[
h_{s-1,s+3} = -h_{s-1,s+1}-\sum_{j=p_{s+2}}^{\infty}j\upkappa_{s+2,j}h_{s-1,s+2}^{j-1}.
\]
Since $h_{s-1,s+1}=0$ and $\uk_{s+2,j}=0$ for all $j$, necessarily $h_{s-1,s+3}=0$. Iterating this argument shows that $h_{s-1,s+5}, h_{s-1,s+7}, \dots, h_{s-1,t+1}=0$.
\end{proof}

\begin{lemma}\label{lemma:h4}
With notation in \textnormal{\ref{notation:hst}} and \textnormal{\ref{064}}, for $s \leq t \leq 2n-1 $, $h_{s-1,t+1} \in \C\lal \upkappa_{t,\mathbf{p}}, x  \ral$, and in particular the lowest degree $(wrt. \ x)$ coefficient $\uplambda_{st}$ 
in $h_{s-1,t+1}$ \eqref{hst0} belongs to $\C\lal \upkappa_{t,\mathbf{p}} \ral$. 
\end{lemma}
\begin{proof}
We first check that $h_{s-1,s+1}$ and $h_{s-1,s+2}$ satisfy the statement. By \eqref{hst1}, it is straightforward that $h_{s-1,s+1} \in  \C\lal \upkappa_{s,\mathbf{p}}, x  \ral$. Then together with \eqref{hst2}, it follows that $h_{s-1,s+2} \in  \C\lal \upkappa_{s+1,\mathbf{p}}, x  \ral$.

We next prove the statement by induction on the second index: we assume that $h_{s-1,t-1} \in \C\lal \upkappa_{t-2,\mathbf{p}}, x  \ral$ and $h_{s-1,t} \in \C\lal \upkappa_{t-1,\mathbf{p}}, x  \ral$ for some $t \geq s+2$, and prove that $h_{s-1,t+1} \in \C\lal \upkappa_{t,\mathbf{p}}, x  \ral$. 
This follows immediately from \eqref{hst}, since the recursion expresses $h_{s-1,t+1}$ in terms of $h_{s-1,t-1}$ and $h_{s-1,t}$ with coefficients involving only $\upkappa_{t,\mathbf p}$.
\end{proof}

\begin{lemma}\label{lemma:h2}
With notation in \textnormal{\ref{064}}, if $\mathsf{d}\colonequals d_{st}(\mathbf{p}) < \infty$, then $r_{st}= \mathsf{d}-1$.
\end{lemma}
\begin{proof}
We first check that $r_{ss}$ and $r_{s,s+1}$ satisfy the statement.
By \eqref{dp}, $d_{ss}(\p)=p_s$ and $d_{s,s+1}(\p)=2$.
By \eqref{hst1}, 
\begin{equation*}
     h_{s-1,s+1} = -\sum_{j=p_s}^{\infty}j\upkappa_{sj}x^{j-1} 
\end{equation*}
This has lowest degree term $x^{p_s-1}$, and thus by definition $r_{ss}=p_s-1=d_{ss}(\mathbf{p})-1$. Similarly, since each $j\upkappa_{s+1,j}h_{s-1,s+1}^{j-1}$ in \eqref{hst2} contains $\upkappa_{s+1,j}$, these terms can not cancel the $-x$ in \eqref{hst2}.
Thus the lowest degree of $h_{s-1,s+2}$ is $1$, and so $r_{s,s+1}=1=d_{s,s+1}(\mathbf{p})-1$. 

We next prove the statement by induction on the second index: we assume that $r_{s,t-2}= d_{s,t-2}(\mathbf{p})-1$ and $r_{s,t-1}= d_{s,t-1}(\mathbf{p})-1$ for some $ t \geq s+2$, and prove that $r_{st}= d_{st}(\mathbf{p})-1$ by splitting into the following two cases.

\noindent
(1) $t-s$ is odd.

Since $t-s$ is odd, $d_{s,t-2}(\mathbf{p})=d_{st}(\mathbf{p}) =2$ by \eqref{dp}. By assumption $r_{s,t-1}=d_{s,t-1}(\mathbf{p})-1$ and $r_{s,t-2}=d_{s,t-2}(\mathbf{p})-1=1$. Thus by \eqref{hst0} (applied to $t-2$ and $t-1$),
\begin{equation*}
    h_{s-1,t-1}= \uplambda_{s,t-2}x+ \scrO_2, \quad h_{s-1,t}= \uplambda_{s,t-1}x^{d_{s,t-1}-1}+ \scrO_{d_{s,t-1}},
\end{equation*}
where $\uplambda_{s,t-2},\ \uplambda_{s,t-1}\neq 0$ by assumption.
Thus by \eqref{hst}, in order to give the lowest degree $r_{st}$ of $h_{s-1,t+1}$, we only need to consider the lowest degree term of $h_{s-1,t-1}$ (namely $\uplambda_{s,t-2}x$) and $\sum_{j=p_t}^{\infty}j\upkappa_{tj}h_{s-1,t}^{j-1}$.

Since by \ref{lemma:h4} $\uplambda_{s,t-2} \in \C\lal \upkappa_{t-2,\mathbf{p}} \ral$ and each $j\upkappa_{tj}h_{s-1,t}^{j-1}$ contains $\upkappa_{tj}$, $\uplambda_{s,t-2}x$ can not be canceled by $\sum_{j=p_t}^{\infty}j\upkappa_{tj}h_{s-1,t}^{j-1}$, and so the lowest degree $r_{st}$ of $h_{s-1,t+1}$ is $1$. Since $d_{st}(\p)=2$, $r_{st}=1= d_{st}(\mathbf{p})-1$.

\noindent
(2) $t-s$ is even.

Since $t-s$ is even, $d_{s,t-1}(\mathbf{p})=2$ by \eqref{dp}. By assumption $r_{s,t-1}=d_{s,t-1}(\mathbf{p})-1=1$ and $r_{s,t-2}=d_{s,t-2}(\mathbf{p})-1$. Thus again by \eqref{hst0} (applied to $t-2$ and $t-1$),
\begin{equation*}
 h_{s-1,t-1}=\uplambda_{s,t-2}x^{d_{s,t-2}-1}+\scrO_{d_{s,t-2}},  \quad h_{s-1,t}=\uplambda_{s,t-1}x+\scrO_2,
\end{equation*}
where $\uplambda_{s,t-2},\ \uplambda_{s,t-1}\neq 0$ by assumption. Thus by \eqref{hst}, in order to give the lowest degree $r_{st}$ of $h_{s-1,t+1}$, we only need to consider the lowest degree term of $h_{s-1,t-1}$ (namely $\uplambda_{s,t-2}x^{d_{s,t-2}-1}$) and $\sum_{j=p_t}^{\infty}j\upkappa_{tj}h_{s-1,t}^{j-1}$ (namely $p_t\uk_{t,p_t}(\uplambda_{s,t-1}x)^{p_t-1}$).

Since by \ref{lemma:h4} $\uplambda_{s,t-2} \in \C\lal \upkappa_{t-2,\mathbf{p}} \ral$, and $p_t\uk_{t,p_t}(\uplambda_{s,t-1}x)^{p_t-1}$ contains $\uk_{t,p_t}$, it follows that $\uplambda_{s,t-2}x^{d_{s,t-2}-1}$ and $p_t\uk_{t,p_t}(\uplambda_{s,t-1}x)^{p_t-1}$ can not cancel each other. Thus the lowest degree $r_{st}$ of $h_{s-1,t+1}$ is $\min(d_{s,t-2}(\p)-1, p_t-1)$. Since $d_{st}(\p)= \min(d_{s,t-2}(\p), p_t)$ by \eqref{dp}, $r_{st}=d_{st}(\p)-1$.
\end{proof}

\begin{lemma}\label{lemma:h3}
With notation in \textnormal{\ref{064}}, if $\mathsf{d} \colonequals d_{st}(\mathbf{p}) <\infty$, then the lowest degree $(wrt. \ x)$ coefficient in $h_{s-1,t+1}$ \eqref{hst0} is $\uplambda_{st}=(-1)^{t-s+1}\det A_{st}^{\mathsf{d}}(\upkappa_\mathbf{p})$.
\end{lemma}
\begin{proof}
To ease notation, we set $d_{ij}\colonequals d_{ij}(\mathbf p)$ and $A_{ij}^d\colonequals A_{ij}^d(\upkappa_{\mathbf p})$ for the remainder of the proof.

We first prove that the statement holds for $t=s$.
By \eqref{hst1}, the lowest degree coefficient in $h_{s-1,s+1}$ is $-p_s\upkappa_{s,p_s}$, thus
\begin{align*}
    \uplambda_{ss}& = -p_s\upkappa_{s,p_s}\\
         & = -d_{ss}\uk_{s,d_{ss}}  \tag{since $p_s=d_{ss}$ by \eqref{dp}}\\ 
        & = -\det A_{ss}^{d_{ss}}. \tag{since $\det A_{ss}^d =d\upkappa_{sd}$ for any $d$ by \ref{example:matrix}}
\end{align*}
We next prove that the statement holds for $t=s+1$. Indeed,
\begin{align*}
    h_{s-1,s+2} &=  -h_{s-1,s}-\sum_{j=p_{s+1}}^{\infty}j\upkappa_{s+1,j}h_{s-1,s+1}^{j-1} \tag{by \eqref{hst2}} \\
    & =-x-\sum_{j=p_{s+1}}^{\infty}j\upkappa_{s+1,j}(\uplambda_{ss}x^{r_{ss}} +\scrO_{r_{ss}+1})^{j-1} \tag{since $h_{s-1,s}=x$, and \eqref{hst0}}\\
    &=-x-\sum_{j=p_{s+1}}^{\infty}j\upkappa_{s+1,j}(-p_s\uk_{s,p_s}x^{r_{ss}} +\scrO_{r_{ss}+1})^{j-1} \tag{$\uplambda_{ss}=-p_s\uk_{s,p_s}$}\\
        &=-x-\sum_{j=p_{s+1}}^{\infty}j\upkappa_{s+1,j}(-p_s\uk_{s,p_s}x^{p_s-1} +\scrO_{p_s})^{j-1} \tag{$r_{ss}=d_{ss}-1=p_s-1$ by \ref{lemma:h2}}\\
    & =-x+ (-1)^{p_{s+1}}p_{s+1}\upkappa_{s+1,p_{s+1}}(p_s\upkappa_{s,p_s})^{p_{s+1}-1}x^{(p_s-1)(p_{s+1}-1)}  +\scrO_{(p_s-1)(p_{s+1}-1)}. 
\end{align*}
If $p_s=p_{s+1}=2$, then $(4\upkappa_{s,2}\upkappa_{s+1,2}-1)x$ is the lowest degree term in $h_{s-1,s+2}$, thus
\begin{align*}
    \uplambda_{s,s+1} & =4\upkappa_{s,2}\upkappa_{s+1,2}-1\\
      & = \det A_{s,s+1}^{2} \tag{since $\det A_{s,s+1}^2=4\uk_{s,2}\uk_{s+1,2}-1$ by \ref{example:matrix}}\\
      & = \det A_{s,s+1}^{d_{s,s+1}}. \tag{since $d_{s,s+1}=2$ by \eqref{dp}}
\end{align*}
Otherwise, if $p_s >2$ or $p_{s+1} > 2$, then $-x$ is the lowest degree term in $h_{s-1,s+2}$ and by \eqref{fkp} $\upkappa_{s,2}=0$ or $\upkappa_{s+1,2}=0$. Thus 
\begin{align*}
    \uplambda_{s,s+1}&=-1 \\
& = 4\upkappa_{s,2}\upkappa_{s+1,2}-1 
\tag{since $\upkappa_{s,2}=0$ or $\upkappa_{s+1,2}=0$}\\
   &   =\det A_{s,s+1}^{2}\tag{since $\det A_{s,s+1}^2 =4\uk_{s,2}\uk_{s+1,2}-1$ by \ref{example:matrix}}\\
   &=\det A_{s,s+1}^{d_{s,s+1}}. \tag{since $d_{s,s+1}=2$ by \eqref{dp}}
\end{align*}

We next prove the statement by induction on the second index. Fix some $t$ satisfying $t \geq s+2$. We assume that $\uplambda_{s,t-2}=(-1)^{t-s-1}\det A_{s,t-2}^{d_{s,t-2}}$ and $\uplambda_{s,t-1}=(-1)^{t-s}\det A_{s,t-1}^{d_{s,t-1}}$, and prove that $\uplambda_{st}=(-1)^{t-s+1}\det A_{st}^{d_{st}}$ by splitting into the following cases.

By \eqref{hst}, for any integer $d \geq 1$, we have
\begin{equation}\label{hd}
[h_{s-1,t-1}]_d+[\sum_{j=p_t}^{\infty}j\upkappa_{tj}h_{s-1,t}^{j-1}]_d+[h_{s-1,t+1}]_d=0,  
\end{equation}
where $[h]_d$ denotes the degree (wrt.$\ x$) $d$ graded piece of $h$ (see \ref{notation:hst}). 

\noindent
(1) $t-s$ is odd.

Since $t-s$ is odd, by \eqref{dp} $d_{s,t-2}=d_{s,t}=2$. Thus by \ref{lemma:h2}, $r_{s,t-2}=r_{st}=1$ and $r_{s,t-1}=d_{s,t-1}-1$. So by \eqref{hst0},
\begin{align*}
&h_{s-1,t-1} =\uplambda_{s,t-2}x+\scrO_2, \\
 & h_{s-1,t} =\uplambda_{s,t-1}x^{r_{s,t-1}} +\scrO_{r_{s,t-1}+1}=\uplambda_{s,t-1}x^{d_{s,t-1}-1}+\scrO_{d_{s,t-1}},\\
& h_{s-1,t+1} =\uplambda_{st}x+\scrO_{2}.    
\end{align*}
Thus the lowest degree of the terms in \eqref{hst} is one.
We then consider these lowest degree terms, thus set $d=1$ in \eqref{hd}, which gives
\begin{equation}\label{h1}
    \uplambda_{s,t-2}x+[p_t\upkappa_{t,p_t}(\uplambda_{s,t-1}x^{d_{s,t-1}-1})^{p_t-1}]_1+\uplambda_{st}x=0.
\end{equation}

Since $t-s$ is odd, the inductive assumption becomes $\uplambda_{s,t-2}=\det A_{s,t-2}^{2}$ and $\uplambda_{s,t-1}=-\det A_{s,t-1}^{d_{s,t-1}}$. We need to prove that $\uplambda_{st}=\det A_{st}^{2}$. We again split into subcases.

\noindent
(1.1) $t-s$ is odd and $p_t>2$.

Since $p_t>2$, $\upvarepsilon_{t2}(\uk_\p) =2\uk_{t2}=0$ by \ref{rmk: MAp}\eqref{rmk: MAp 4} and $[p_t\upkappa_{t,p_t}(\uplambda_{s,t-1}x^{d_{s,t-1}-1})^{p_t-1}]_1=0$. To ease notation, we write $\upvarepsilon_{t2}$ for $\upvarepsilon_{t2}(\uk_\p)$ in the following.
Thus 
\begin{align*}
   \uplambda_{st} &=- \uplambda_{s,t-2} \tag{by \eqref{h1} and $[p_t\upkappa_{t,p_t}(\uplambda_{s,t-1}x^{d_{s,t-1}-1})^{p_t-1}]_1=0$}\\
   &=-\det A_{s,t-2}^2 \tag{by assumption}\\
   &=\det A_{st}^2-\upvarepsilon_{t2} \det A_{s,t-1}^2\tag{by \ref{517}\eqref{517 1}} \\
   &= \det A_{st}^2. \tag{since $\upvarepsilon_{t2}=0$}
\end{align*}

\noindent
(1.2) $t-s$ is odd, $p_t=2$ and $d_{s,t-1} >2$.

Since $d_{s,t-1} >2$, $[p_t\upkappa_{t,p_t}(\uplambda_{s,t-1}x^{d_{s,t-1}-1})^{p_t-1}]_1=0$ and by \ref{41} $\det A_{s,t-1}^2=0$. Thus 
\begin{align*}
 \uplambda_{st} &=- \uplambda_{s,t-2} \tag{by \eqref{h1} and $[p_t\upkappa_{t,p_t}(\uplambda_{s,t-1}x^{d_{s,t-1}-1})^{p_t-1}]_1=0$}\\
& =-\det A_{s,t-2}^2 \tag{by assumption}\\
&=\det A_{st}^2-\upvarepsilon_{t2} \det A_{s,t-1}^2\tag{by \ref{517}\eqref{517 1}}\\
&= \det A_{st}^2. \tag{since $\det A_{s,t-1}^2=0$}
\end{align*}

\noindent
(1.3) $t-s$ is odd, $p_t=2$ and $d_{s,t-1} =2$.

Since $p_t=2$ and $d_{s,t-1} =2$, $[p_t\upkappa_{t,p_t}(\uplambda_{s,t-1}x^{d_{s,t-1}-1})^{p_t-1}]_1= 2\upkappa_{t2}\uplambda_{s,t-1}x$. Thus
\begin{align*}
 \uplambda_{st} &=-2\upkappa_{t2}\uplambda_{s,t-1}- \uplambda_{s,t-2}\tag{by \eqref{h1} and $[p_t\upkappa_{t,p_t}(\uplambda_{s,t-1}x^{d_{s,t-1}-1})^{p_t-1}]_1= 2\upkappa_{t2}\uplambda_{s,t-1}x$}\\
 &=\upvarepsilon_{t2}\det A_{s,t-1}^{d_{s,t-1}}-\det A_{s,t-2}^2\tag{by assumption and $\upvarepsilon_{t2}=2\upkappa_{t2}$}\\
 &=\upvarepsilon_{t2}\det A_{s,t-1}^{2}-\det A_{s,t-2}^2 \tag{since $d_{s,t-1} =2$ } \\
 &= \det A_{st}^2\tag{by \ref{517}\eqref{517 1}}.
\end{align*}

\noindent
(2) $t-s$ is even.

Since $t-s$ is even, then $d_{s,t-1}=2$ by \eqref{dp}. Thus by \ref{lemma:h2}, $r_{s,t-1}=1$, $r_{s,t-2}=d_{s,t-2}-1$ and $r_{st}=d_{st}-1$. So by \eqref{hst0},
\begin{align*}
&h_{s-1,t-1} =\uplambda_{s,t-2}x^{r_{s,t-2}} +\scrO_{r_{s,t-2}+1}=\uplambda_{s,t-2}x^{d_{s,t-2}-1}+\scrO_{d_{s,t-2}}, \\
 & h_{s-1,t} =\uplambda_{s,t-1}x+\scrO_2,\\
& h_{s-1,t+1} =\uplambda_{st}x^{r_{st}} +\scrO_{r_{st}+1}=\uplambda_{st}x^{d_{st}-1}+\scrO_{d_{st}}.    
\end{align*}
Since by \eqref{dp} $d_{s,t-2} \geq d_{st}$ and $p_t\geq d_{st}$, the lowest degree of $h_{s-1,t-1}$ and $(h_{s-1,t})^{p_t-1}$ is greater than or equal to that of $h_{s-1,t+1}$. Thus the lowest degree of the terms in \eqref{hst} is $d_{st}-1$.
We then consider these lowest degree terms, thus set $d=d_{st}-1$ in \eqref{hd}, which gives
\begin{equation}\label{h2}
    [\uplambda_{s,t-2}x^{d_{s,t-2}-1}]_{d_{st}-1}+[p_t\upkappa_{t,p_t}(\uplambda_{s,t-1}x)^{p_t-1}]_{d_{st}-1}+\uplambda_{st}x^{d_{st}-1}=0.
\end{equation}

Since $t-s$ is even, the inductive assumption now becomes $\uplambda_{s,t-2}=-\det A_{s,t-2}^{d_{s,t-2}}$ and $\uplambda_{s,t-1}=\det A_{s,t-1}^{2}$. We prove $\uplambda_{st}=-\det A_{st}^{d_{st}}$ by splitting into the following subcases.

\noindent
(2.1) $t-s$ is even and $p_t<d_{s,t-2}$.

Since $p_t<d_{s,t-2}$, by \eqref{dp} $p_t=d_{st}<d_{s,t-2}$, and so $[\uplambda_{s,t-2}x^{d_{s,t-2}-1}]_{d_{st}-1}=0$.  Thus by \eqref{h2}, it follows that
\begin{equation*}
    \uplambda_{st} = -p_t\upkappa_{t,p_t}\uplambda_{s,t-1}^{p_t-1}.
\end{equation*}
Now, since $d_{st}<d_{s,t-2}$,  by \ref{41} $\det A_{s,t-2}^{d_{st}}=0$.
If furthermore $p_t= d_{st}=2$, then 
\begin{align*}
\uplambda_{st} & =-2\upkappa_{t2}\uplambda_{s,t-1} \tag{since $p_t=2$}\\
&= -2\upkappa_{t2}\det A_{s,t-1}^2 \tag{by assumption} \\
&=  - \upvarepsilon_{t2}\det A_{s,t-1}^2 +\det A_{s,t-2}^2  \tag{since $\upvarepsilon_{t2}=2\upkappa_{t2}$, $\det A_{s,t-2}^{d_{st}}=0$ and $d_{st}=2$}\\
 & =-\det A_{st}^{2} \tag{by \ref{517}\eqref{517 1}} \\
 &=-\det A_{st}^{d_{st}}. \tag{since $d_{st}=2$}
\end{align*}

Otherwise, $p_t= d_{st}>2$, and then by \ref{518} $\det A_{s,t-1}^{2}=(-1)^{(t-s)/2}$, and so
\begin{align*}
  \uplambda_{st} & = -p_t\upkappa_{t,p_t}\uplambda_{s,t-1}^{p_t-1}\\
  & =  -p_t\upkappa_{t,p_t}(\det A_{s,t-1}^{2})^{p_t-1} \tag{by assumption}\\
  &=-d_{st}\upkappa_{t,d_{st}}(-1)^{(t-s)(d_{st}-1)/2} \tag{since $\det A_{s,t-1}^{2}=(-1)^{(t-s)/2}$ and $p_t=d_{st}$}\\
  &= -(-1)^{(t-s)(d_{st}-1)/2} \upvarepsilon_{t,d_{st}}  +\det A_{s,t-2}^{d_{st}}\tag{since $\upvarepsilon_{t,d_{st}}=d_{st}\upkappa_{t,d_{st}}$ and $\det A_{s,t-2}^{d_{st}}=0$}\\
 & =-\det A_{st}^{d_{st}}. \tag{by \ref{517}\eqref{517 3}}
\end{align*}

\noindent
(2.2) $t-s$ is even and $p_t>d_{s,t-2}$.

Since $p_t > d_{s,t-2}$, by \eqref{dp} $p_t> d_{s,t-2} =d_{st}$, and thus  $[p_t\upkappa_{t,p_t}(\uplambda_{s,t-1}x)^{p_t-1}]_{d_{st}-1}=0$. Hence by \eqref{h2}, it follows that
\begin{equation*}
    \uplambda_{st}= -\uplambda_{s,t-2}.
\end{equation*}
Since $p_t > d_{s,t}$, by \ref{rmk: MAp}\eqref{rmk: MAp 4} $\upvarepsilon_{t,d_{st}}(\uk_{\p})=d_{st}\uk_{t,d_{st}}=0$.
If furthermore $d_{s,t-2} =d_{st}=2$, then
\begin{align*}
\uplambda_{st} &= -\uplambda_{s,t-2}\\
&=\det A_{s,t-2}^{d_{s,t-2}} \tag{by assumption}\\
&= \det A_{s,t-2}^{2} \tag{since $d_{s,t-2}=2$}\\
&=- \upvarepsilon_{t2} \det A_{s,t-1}^2 + \det A_{s,t-2}^2\tag{since $\upvarepsilon_{t,d_{st}}=0$ and $d_{st}=2$}\\
&=-\det A_{st}^2 \tag{by \ref{517}\eqref{517 1}}\\
&=-\det A_{st}^{d_{st}}. \tag{since $d_{st}=2$}
\end{align*}

Otherwise, $d_{s,t-2} =d_{st}>2$, and then
\begin{align*}
\uplambda_{st} & = -\uplambda_{s,t-2}\\
&=\det A_{s,t-2}^{d_{s,t-2}} \tag{by assumption}\\
&=\det A_{s,t-2}^{d_{st}}\tag{since $d_{s,t-2} =d_{st}$}\\
&=- (-1)^{(t-s)(d_{st}-1)/2}\upvarepsilon_{t,d_{st}}+\det A_{s,t-2}^{d_{st}}\tag{since $\upvarepsilon_{t,d_{st}}=0$}\\
&= -\det A_{st}^{d_{st}}. \tag{by \ref{517}\eqref{517 3}}
\end{align*}

\noindent
(2.3) $t-s$ is even and $p_t=d_{s,t-2}$.

Since $p_t=d_{s,t-2}$, by \eqref{dp} $p_t=d_{s,t-2}=d_{st}$. Thus by \eqref{h2}
\begin{equation*}
    \uplambda_{st}=-\uplambda_{s,t-2}-p_t\upkappa_{t,p_t}(\uplambda_{s,t-1})^{p_t-1}.
\end{equation*}
If furthermore $p_t=d_{s,t-2}=d_{st}=2$, then
\begin{align*}
\uplambda_{st} &=-\uplambda_{s,t-2}-2\upkappa_{t2}\uplambda_{s,t-1} \tag{since $p_t=2$}\\
&=\det A_{s,t-2}^{2} -2\upkappa_{t2}\det A_{s,t-1}^{2}\tag{by assumption and $d_{s,t-2}=2$}\\
&= \det A_{s,t-2}^{2} -\upvarepsilon_{t2}\det A_{s,t-1}^{2} \tag{since $\upvarepsilon_{t2}=2\uk_{t2}$}\\
&=-\det A_{st}^{2}\tag{by \ref{517}\eqref{517 1}}\\
&= -\det A_{st}^{d_{st}}. \tag{since $d_{st}=2$}
\end{align*}
Otherwise, $p_t=d_{s,t-2}=d_{st}>2$. But then by \ref{518} $\det A_{s,t-1}^{2}=(-1)^{(t-s)/2}$, and so
\begin{align*}
    \uplambda_{st}&=-\uplambda_{s,t-2}-p_t\upkappa_{t,p_t}(\uplambda_{s,t-1})^{p_t-1}\\
    & = \det A_{s,t-2}^{d_{s,t-2}}-p_t\upkappa_{t,p_t}(\det A_{s,t-1}^{2})^{p_t-1} \tag{by assumption}\\
    &= \det A_{s,t-2}^{d_{st}}-d_{st}\upkappa_{t,d_{st}}(-1)^{(t-s)(d_{st}-1)/2} \tag{since $\det A_{s,t-1}^{2}=(-1)^{(t-s)/2}$ and $p_t=d_{s,t-2}=d_{st}$}\\
  &= \det A_{s,t-2}^{d_{st}}-(-1)^{(t-s)(d_{st}-1)/2}\upvarepsilon_{t,d_{st}}\tag{since $\upvarepsilon_{t,d_{st}}=d_{st}\upkappa_{t,d_{st}}$}\\
 & = -\det A_{st}^{d_{st}}. \tag{by \ref{517}\eqref{517 3}} 
\end{align*}
So by induction $\uplambda_{st}=(-1)^{t-s+1}\det A_{st}^{d_{st}}$ for any $1\leq s \leq t \leq 2n-1$.
\end{proof}

We fix $\p$ and the exceptional curve class $\Curve_{s} + \Curve_{s+1}+\dots +\Curve_{t}$, and from this data construct a filtration structure of $\M_{\mathbf{p}}$, which is the main result of this section.
Recall that $\M_{\mathbf{p}}$ is the parameter space of monomialised Type A potentials $f(\uk_{\p})$ \eqref{MAp}, namely
\begin{align*}
 &  f(\uk_{\p})= \sum_{i=1}^{2n-2}\x_i^{\prime}\x_{i+1} + \sum_{i=1}^{2n-1}\sum_{j=2}^{\infty} \upkappa_{ij} \x_i^{j}, \text{ where } \upkappa_{i,j_i}=0 \text{ for } 1 \leq i \leq 2n-1, 2 \leq j_i <p_i,\\
 &\M_{\p} = \{(k_{12},k_{13}, \dots,  k_{2n-1,2},k_{2n-1,3}, \dots)\mid k_{i,j_i}=0 \text{ for } 1 \leq i \leq 2n-1, 2 \leq j_i <p_i\}.
\end{align*}
Recall the notation $d_{ij}(\p)$ and $A_{ij}^d(\uk_{\p})$ in \ref{TypeAp}.

\begin{theorem}\label{58}
Fix $\p$, and integers $s,t$ with $1 \leq s \leq t \leq n$, and set $d\colonequals d_{2s-1,\,2t-1}(\mathbf p)$.

If $d$ is finite, then $\M_{\mathbf{p}}$ admits a filtration $\M_{\mathbf{p}}= M_1 \supsetneq M_2 \supsetneq M_3 \supsetneq \cdots$ such that:
\begin{enumerate}
\item For each $i \geq 1$, one has $N_{st}(f(k))=d+i-2$ for all $ k\in M_{i} \setminus M_{i+1}$.
\item Each $M_{i}$ is the common zero locus of finitely many power series in the coordinates $\upkappa_{\mathbf p}$. Moreover, 
\[
 M_2=\{k \in \M_\p \mid \det A^d_{2s-1,2t-1}(f(k))=0 \} .
\]
\item If $s=t$, then for each $i \geq 2$
\[
M_i=\{ k \in \M_\p \mid k_{2s-1,j} =0 \textnormal{ for }p_{2s-1} \leq j \leq p_{2s-1}+i-2\}.
\]
\end{enumerate}
Otherwise, if $d_{2s-1,2t-1}(\mathbf{p})$ is infinite, then $N_{st}(f(k)) = \infty$ for all $ k \in \M_{\mathbf{p}} $.
\end{theorem}

\begin{proof}
With notation as in \ref{notation:hst}, it follows from \eqref{Nstf} that
\begin{equation}\label{6i}
     N_{st}(f(\uk_{\p}))=\operatorname{dim}_{\mathbb{C}} \frac{\mathbb{C}\lal x\ral}{(h_{2s-2,2t})}.
\end{equation}
Recall that we set $d \colonequals d_{2s-1,2t-1}(\mathbf{p})$. By \ref{064},
\begin{equation}\label{6h}
 h_{2s-2,2t}=\left\{
\begin{array}{cl}
 0 &\text{if } d=\infty \\
 \sum_{i=r}^{\infty} c_i x^{i} &\text{if } d<\infty
\end{array}\right.
\end{equation}
where each $c_i \in \C \lal \upkappa_\mathbf{p} \ral$, $r= d-1$ and $c_r=-\det A_{2s-1,2t-1}^{d}(\upkappa_\mathbf{p})$.

If $d=\infty$, then $ h_{2s-2,2t}=0$, and hence $ N_{st}(f(\uk_{\p}))=\infty$ by \eqref{6i}.

\noindent
(1),\,(2)  Assume now that $d<\infty$. For each $i \geq r$, write $c_i(k)$ for the evaluation of $c_i$ at $k\in \M_{\mathbf p}$. Define $Z_1\colonequals \M_{\mathbf p}$ and, for $i\ge 2$, set
\[
Z_i\colonequals \Big\{k\in \M_{\mathbf p}\ \Big|\ c_r(k)=c_{r+1}(k)=\cdots=c_{r+i-2}(k)=0\Big\}.
\]
This yields a sequence $Z_1\supseteq Z_2\supseteq Z_3\supseteq \cdots$.
Let $M_1\supsetneq M_2\supsetneq \cdots \supsetneq M_m$ be the strictly decreasing sequence obtained by removing repetitions from $(Z_i)_{i\ge 1}$, where $1 \leq m \leq \infty$. Note that $m= \infty$ precisely when the resulting strictly decreasing sequence is infinite.

By construction, each $M_i$ is the common zero locus of finitely many power series in the coordinates $\upkappa_{\mathbf p}$.
Moreover, by \eqref{6h} and \eqref{6i}, the value of $N_{st}(f(k))$ is constant on each stratum $M_{i} \setminus  M_{i+1}$. We therefore set $d_i \colonequals N_{st}(M_{i} \setminus  M_{i+1})$, which obviously satisfies $d_1 < d_2 <  \cdots < d_{m}$.

By \ref{41}, the power series $\det A_{2s-1,2t-1}^{d}(\upkappa_\mathbf{p})$ is non-zero in $\C\lal \upkappa_\p \ral$. 
Consequently,
\[
Z_2=\{k\in \M_{\mathbf p}\mid c_r(k)=0\}
\subsetneq Z_1,
\] 
so that $M_2=Z_2\subsetneq Z_1=M_1$ and $m\geq 2$. Furthermore, $d_1=N_{st}(M_1\setminus M_2)=r=d-1$.

We next show that $d_i=N_{st}(M_i\setminus M_{i+1}) =d+i-2$ for all $i \geq 2$. 
Fix an integer $i$ with $i \geq 2$. By \eqref{dp}, there exists a tuple  $\mathbf{p'}$ with $\mathbf{p'} \geq \mathbf{p}$ (see \ref{TypeAp}\eqref{TypeAp 5}) such that 
\[
d_{2s-1,2t-1}(\mathbf{p'})=d_{2s-1,2t-1}(\mathbf{p})+i-1=d+i-1.
\]
Since $\mathbf{p'} \geq \mathbf{p}$, we have $\mathsf{MA}_{\mathbf{p'}} \subseteq \mathsf{MA}_{\mathbf{p}}$ and $\M_{\mathbf{p'}} \subseteq \M_{\mathbf{p}}$ by \ref{rmk: MAp}\eqref{rmk: MAp 5}.
Repeating the same argument above, this yields a filtration 
\[
\M_{\mathbf{p'}} = M'_1 \supsetneq M'_2 \supsetneq \cdots \supsetneq M'_{m'},
\quad m'\ge 2,
\]
such that
\[
N_{st}(M'_1\setminus M'_2)
=d_{2s-1,\,2t-1}(\mathbf{p'})-1
=d+i-2.
\]
Set $U_i \colonequals M'_1\setminus M'_2$, so that $U_i \subsetneq M_1'=\M_{\mathbf{p'}} \subseteq \M_{\mathbf{p}}=M_1$.

Since this construction applies to every $i \geq 2$, we obtain an infinite family of subsets $U_i\subsetneq M_1$ satisfying $N_{st}(U_i)=d+i-2$. On the other hand, by construction $d-1=d_1 < d_2 <  \cdots < d_m$ and $d_i = N_{st}(M_{i} \setminus M_{i+1})$, it follows that each $U_i$ is contained in
the stratum $M_i\setminus M_{i+1}$, and hence that $m=\infty$. Therefore, for every $ i \geq 2$,
\[
d_{i}=N_{st}(M_{i}\setminus M_{i+1}) = N_{st}(U_i) =d+i-2.
\]

\noindent
(3) Finally, we consider the case $s=t$. Recall that $h_{2s-2,2s-2}=0$ and $h_{2s-2,2s-1}=x$ (see \ref{notation:hst}). By \eqref{hst}, 
\begin{equation*}
    h_{2s-2,2s}= -\sum_{j=p_{2s-1}}^{\infty}j\upkappa_{2s-1,j}x^{j-1},
\end{equation*}
and hence \eqref{6i} gives
\begin{equation*}
     N_{ss}(f(\uk_{\p}))=\operatorname{dim}_{\mathbb{C}} \frac{\mathbb{C}\lal x\ral}{(h_{2s-2,2s})}=\operatorname{dim}_{\mathbb{C}} \frac{\mathbb{C}\lal x\ral}{(\sum_{j=p_{2s-1}}^{\infty}j\upkappa_{2s-1,j}x^{j-1})}.
\end{equation*}
Therefore the statement follows immediately.
\end{proof}

If we set $\mathbf{p}=(2,2,\dots,2)$ in \ref{58}, then $\M_{\p}$ coincides with $\M$ which is the parameter space of all monomialised Type A potentials $f(\uk)$ (see \ref{TypeAp}, \ref{rmk: MAp}), as follows.
\begin{align*}
&f(\uk)=\sum_{i=1}^{2n-2}\x_i^{\prime}\x_{i+1} + \sum_{i=1}^{2n-1}\sum_{j=2}^{\infty} \upkappa_{ij} \x_i^{j},\\
    & \M= \{(k_{12},k_{13}, \dots, k_{22},k_{23},\dots, k_{2n-1,2},k_{2n-1,3}, \dots)\mid \text{all }k_* \in\mathbb{C}\}.
\end{align*}
Thus, as a special case of \ref{58}, we next give a filtration structure of $\M$ with respect to a fixed curve class.

\begin{cor}\label{stra}
Fix some $s$, $t$ satisfying $1 \leq s \leq t \leq n$, then $\M$ admits a filtration $\M=M_1 \supsetneq M_2 \supsetneq M_3 \supsetneq \cdots $ such that:
\begin{enumerate}
\item For each $i \geq 1$, one has $N_{st}(f(k))=i$ for all $ k \in M_{i} \setminus M_{i+1}$.
\item Each $M_{i}$ is the common zero locus of finitely many power series in the coordinates $\upkappa$. Moreover,
\[
M_{2}= \{ k \in \M \mid \det A_{2s-1,2t-1}^{2}(f(k))=0\}.
\]
\item If $s=t$, then for each $i \geq 2$
\[
M_i=\{ k \in \M \mid k_{2s-1,j} =0 \textnormal{ for }2 \leq j \leq i \}.
\]
\end{enumerate}
\end{cor}
\begin{proof}
By setting $\mathbf{p}=(2,2,\dots,2)$ in \ref{58}, then $d_{2s-1,2t-1}(\mathbf{p})=2$, and so the statement follows immediately.
\end{proof}

\begin{remark}[Moduli-theoretic interpretation]\label{rmk:moduli}
The filtration constructed above reflects a familiar phenomenon in moduli theory: objects in general position typically exhibit the most regular behaviour, while more singular or degenerate objects appear along strata of higher codimension.
In the present setting, this structure is particularly transparent, since the parameter space is a vector space and no quotient by isomorphisms is taken. As a result, the filtration provides a simple model for such stratifications, free from the additional technicalities arising from schemes or stacks.
\end{remark}

\subsection{Examples}
In this subsection, we apply \ref{58} and \ref{stra} to illustrate the filtration structures on the parameter spaces of monomialised Type~A potentials on $Q_1$ and $Q_2$.

\begin{example}\label{example:filtration1}
Consider monomialised Type $A$ potentials $f(\uk)=\sum_{j=2}^{\infty} \uk_{1j}\x_1^j$ on $Q_1$, where
\[
\begin{tikzpicture}[bend angle=15, looseness=1]
\node (d) at (2,0) [vertex] {};
 \draw[->]  (d) edge [in=55,out=120,loop,looseness=12] node[above] {$\scriptstyle \x_1$}  (d);
 \node (z) at (1.3,0) {$Q_{1} =$};
\end{tikzpicture}
\]

The corresponding parameter space $\M$ is $\{(k_{12},k_{13}, \dots )\mid \text{all }k_* \in\mathbb{C}\}.$ By \ref{stra}(3), for any $i \geq 1$ and $k \in \M$ one has
\begin{equation*}
    N_{11}(f(k)) = i   \iff   k_{1,i+1} \neq 0 \text{ and } k_{1j}=0 \text{ for all } j \leq i.
\end{equation*}
This can also be interpreted geometrically. For any $k \in \M $, consider the $cA_1$ singularity 
\[
\scrR = \frac{\mathbb{C} \lal u, v, x, y \ral}{uv-y(y+\sum_{j=2}^{\infty} jk_{1j}x^{j-1})},
\]
together with the $\scrR$-module $M = \scrR \oplus (u,y)$.
The potential $f(k)$ is realised by the crepant resolution $\uppi \colon \scrX \rightarrow \Spec \scrR$ associated to $M$, and satisfies $\Jac(f(k)) \cong \Lambda_{\con}(\uppi) \cong \underline{\End}_{\scrR}(M)$ (see \S\ref{sec: GGV}).  
Therefore, by \eqref{Nst},
\begin{equation*}
    N_{11}(f(k)) = N_{11}(\uppi) = \operatorname{dim}_{\mathbb{C}} \frac{\mathbb{C}\lal x,y \ral}{(y,y+\sum_{j=2}^{\infty} jk_{1j}x^{j-1})}=\operatorname{dim}_{\mathbb{C}} \frac{\mathbb{C}\lal x \ral}{(\sum_{j=2}^{\infty} jk_{1j}x^{j-1})}.
\end{equation*}
This recovers the above characterisation directly.
\end{example}

\begin{example}\label{example:filtration2}
Consider monomialised Type $A$ potentials 
\[
f(\uk)=\sum_{j=2}^{\infty} \upkappa_{1j} \x_1^{j}+\x_1^{\prime}\x_{2} + \sum_{j=2}^{\infty} \upkappa_{2j} \x_2^{j}+\x_2^{\prime}\x_{3}+\sum_{j=2}^{\infty} \upkappa_{3j} \x_3^{j} 
\]
on $Q_2$, where
\[
 \begin{array}{cl}
\begin{array}{c}
\begin{tikzpicture}[bend angle=15, looseness=1.2]
\node (a) at (-1,0) [vertex] {};
\node (b) at (0,0) [vertex] {};
\node (z) at (-1.8,0) {$Q_{2} =$};

\node (a1) at (-1,-0.2) {$\scriptstyle 1$};
\node (a2) at (0,-0.2) {$\scriptstyle 2$};

\draw[->,bend left] (a) to node[above] {$\scriptstyle a_{2}$} (b);
\draw[<-,bend right] (a) to node[below] {$\scriptstyle b_{2}$} (b);

\draw[<-]  (a) edge [in=120,out=55,loop,looseness=11] node[above] {$\scriptstyle a_{1}$} (a);
\draw[<-]  (b) edge [in=120,out=55,loop,looseness=11] node[above] {$\scriptstyle a_{3}$} (b);

\end{tikzpicture}
\end{array}
&
\begin{array}{l}
\x_1=\x_1'=a_1\\
\x_3=\x_3'=a_3\\
\x_2=a_2b_2,  \x_2'=b_2a_2.
\end{array}
\end{array}
\]
The parameter space $\M $ is $\{(k_{12},k_{13}, \dots k_{22},k_{23}, \dots k_{32},k_{33}, \dots )\mid \text{all }k_* \in\mathbb{C}\}$. Recall from~\ref{Aij} that for any $k \in \M$
\begin{equation*}
A_{13}^2(f(k)) =
\begin{bmatrix}
    2k_{12}  & 1 & 0\\
    1 & 2k_{22} & 1\\
    0 & 1 & 2k_{32}
\end{bmatrix}.
\end{equation*}
Thus $\det A_{13}^2(f(k)) = 8k_{12}k_{22}k_{32}-2k_{12}-2k_{32}$. Fixing the curve class $\Curve_1+\Curve_2$, \ref{stra}(2) gives
\begin{align*}
  & N_{12}(f(k)) = 1   \iff   \det A_{13}^2(f(k)) \neq 0 \iff 4k_{12}k_{22}k_{32}-k_{12}-k_{32} \neq 0,\\
  & N_{12}(f(k)) > 1   \iff   \det A_{13}^2(f(k)) = 0 \iff 4k_{12}k_{22}k_{32}-k_{12}-k_{32}  = 0.
\end{align*}
Thus the generalised GV invariant $N_{12}$ takes the value $1$ at a general point of $\M$, while it is strictly greater than $1$ along the codimension-one locus defined by the equation $4\uk_{12}\uk_{22}\uk_{32}-\uk_{12}-\uk_{32} = 0$.
\end{example}

We now choose a different parameter $\p$ from the previous example and examine the resulting filtration.
This example shows explicitly that there exists a nonempty subspace of the parameter space on which the generalised GV invariant $N_{12}$ takes the value $2$, and illustrates the construction in the proof of \ref{58}.

\begin{example}\label{example:filtration3}
Set $\p = (3,2,3)$ and consider the subset $f(\uk_{\p})$ of monomialised Type $A$ potentials on $Q_2$, so 
\[
f(\uk_{\p})=\sum_{j=3}^{\infty} \upkappa_{1j} \x_1^{j}+\x_1^{\prime}\x_{2} + \sum_{j=2}^{\infty} \upkappa_{2j} \x_2^{j}+\x_2^{\prime}\x_{3}+\sum_{j=3}^{\infty} \upkappa_{3j} \x_3^{j} 
\]
(see \ref{TypeAp}). The parameter space $\M_{\p}$ is $\{(k_{13},k_{14}, \dots k_{22},k_{23}, \dots k_{33},k_{34}, \dots )\mid \text{all }k_* \in\mathbb{C}\}.$ Recall in \ref{TypeAp} and \ref{Aij} that $d_{13}(\p)=3$ and for any $k \in \M_\p$
\begin{equation*}
A_{13}^3(f(k)) =
\begin{bmatrix}
    3k_{13}  & 0 & 1 & 0\\
    1 &1 & 0 & 0\\
    0 & 1 & 0 &1\\
    0 & 0 & 1 & 3k_{33}
\end{bmatrix}.
\end{equation*}
Thus $\det A_{13}^3(f(k)) = 3k_{33}-3k_{13}$. For fixed curve class $\Curve_1+\Curve_2$ (so, $s=1, t=2$), $d_{2s-1,2t-1}(\mathbf{p})=d_{13}(\p)=3$, and thus by \ref{58} for any $k\in \M_\p$
\begin{align*}
  & N_{12}(f(k)) = d_{13}(\p)-1=2   \iff   \det A_{13}^3(f(k)) \neq 0 \iff k_{33}-k_{13} \neq 0,\\
  & N_{12}(f(k)) > d_{13}(\p)-1=2   \iff   \det A_{13}^3(f(k)) = 0 \iff k_{33}-k_{13} = 0.
\end{align*}
Thus the generalised GV invariant $N_{12}$ takes the value $2$ at a general point of $\M_\p$.

Since $\p = (3,2,3)$, Definition~\eqref{Mp} identifies $\M_\p = \{k \in \M \mid k_{12}=0=k_{32} \}$. Consequently,
\[
U_2 \colonequals  \{k \in \M_\p \mid  k_{33}-k_{13} \neq 0 \} =\{k \in \M \mid k_{12}=0=k_{32} \mbox{ and } k_{33}-k_{13} \neq 0 \},
\]
where $\M$ is the parameter space of all monomialised Type A potentials on $Q_2$ as in \ref{example:filtration2}.
By the above discussion, $N_{12}(U_2)=2$. In particular, $U_2$ is a nonempty subset of $\M$ on which the invariant $N_{12}$ takes the value $2$.
Now consider
\[
M_2 = \{k \in \M \mid 4k_{12}k_{22}k_{32}-k_{12}-k_{32} =0 \}.
\]
By \ref{example:filtration2}, this is precisely the first stratum of the filtration, satisfying $N_{12}(\M\setminus M_2)=1$ and $N_{12}(M_2)\ge 2$. Since $U_2\subseteq\M$ and $N_{12}(U_2)=2$, $U_2$ is necessarily contained in $M_2$. This inclusion can also be verified directly:
\begin{equation*}
    U_2=\{k \in \M \mid k_{12}=0=k_{32}, k_{33}-k_{13} \neq 0 \} \subseteq \{k \in \M \mid 4k_{12}k_{22}k_{32}-k_{12}-k_{32} =0 \} =M_2.
\end{equation*}
\end{example}

\section{Obstructions}\label{section:obs}

\subsection{Obstructions}
Building on the filtration structures developed in \S\ref{sec: Filtrations},
this subsection identifies obstructions and constructions of
generalised GV invariants arising from crepant resolutions of $cA_n$ singularities. The main results are summarised in \ref{thm:obs}.

Recall the definitions of generalised GV invariants for crepant resolutions of $cA_n$ singularities and for monomialised Type~A potentials from \ref{def:Nij} and \ref{def:gvf}, respectively.

\begin{definition}\label{def:GV_tuple}
Let $\uppi$ be a crepant resolution of a $cA_n$ singularity.
The \emph{generalised GV tuple} of $\uppi$ is defined to be
$N(\uppi) \colonequals (N_{st}(\uppi) \mid $ all $ 1\leq s  \leq t \leq n )$.
Similarly, for a monomialised Type~A potential $f$ on $Q_n$, the
\emph{generalised GV tuple} of $f$ is $N(f) \colonequals (N_{st}(f) \mid $ all $ 1\leq s  \leq t \leq n )$.
\end{definition}

\begin{lemma}\label{lemma:Nst}
Let $\uppi$ be a crepant resolution of a $cA_n$ singularity and
$f$ a monomialised Type~A potential on $Q_n$.
If $\Lambda_{\mathrm{con}}(\uppi)\xrightarrow[]{\sim}\Jac(f)$, then $N_{st}(\uppi)=N_{st}(f)$ for all $1 \leq s \leq t \leq n$, and hence $N(\uppi)=N(f)$.
\end{lemma}
\begin{proof}
By the construction of $N_{st}(f)$ in \ref{def:gvf}, there exists a
crepant resolution $\uppi'$ such that
$\Lambda_{\mathrm{con}}(\uppi')\xrightarrow[]{\sim}\Jac(f)$ and
$N_{st}(\uppi')=N_{st}(f)$ for all $s,t$ (see~\eqref{Nst}).
Since $\Lambda_{\mathrm{con}}(\uppi)\xrightarrow[]{\sim}
\Lambda_{\mathrm{con}}(\uppi')$, it follows from~\ref{371} that
$N_{st}(\uppi)=N_{st}(\uppi')$, and hence $N_{st}(\uppi)=N_{st}(f)$.
\end{proof}

For any fixed pair $1\leq s\leq t\leq n$ and any $N\in\N_{\infty}$,
\ref{stra} shows that there exists a crepant resolution $\uppi$
of a $cA_n$ singularity such that $N_{st}(\uppi)=N$.
However, this statement no longer holds when one considers
generalised GV invariants for multiple curve classes simultaneously.

\begin{notation}\label{notation:gv_tuple2}
Fix a positive integer $k$, let $\mathbf{q}=\{(\upbeta_1,q_1), (\upbeta_2,q_2), \dots, (\upbeta_k,q_k) \}$ where each $\upbeta_i \in  \bigoplus_{j=1}^n \Z \left\langle \Curve_j \right\rangle$ and $q_i \in \N_{\infty}$. Set $\mathbf{q}_{\min}\colonequals \min\{q_i\}$, and define
\[
\mathsf{CA}_{\mathbf{q}} \colonequals \{\textnormal{crepant resolutions } \uppi \textnormal{ of } cA_n \textnormal{ singularities} \mid
(N_{\upbeta_1}(\uppi),\dots,N_{\upbeta_k}(\uppi))=(q_1,\dots,q_k)\}.
\]
\end{notation}

\begin{notation}\label{notation:gv_tuple}
Fix some $s$, $t$ with $1 \leq s \leq t \leq n$, and a tuple $(q_s, \dots, q_t)\in \N_{\infty}^{t-s+1}$. 
\begin{enumerate}
\item As in \ref{notation:gv_tuple2}, consider $\mathbf{q}\colonequals \{(\Curve_s,q_s),(\Curve_{s+1},q_{s+1}), \dots, (\Curve_t,q_t)\}$, and its associated subset of crepant resolutions of $cA_n$ singularities $\mathsf{CA}_{\mathbf{q}}$.

\item Set $\mathbf{p}=(p_1, p_2, \dots, p_{2n-1})$, where $p_{2i-1}\colonequals q_i+1$ for $s \leq i \leq t$, else $p_i\colonequals 2$, and consider monomialised Type $A$ potentials $\mathsf{MA}_\mathbf{p}$ on $Q_n$ defined in \ref{TypeAp}.
\item We define a nonempty subset $\mathsf{MA}_{\mathbf{p}}^{\circ} \subseteq \mathsf{MA}_{\mathbf{p}}$ (defined in \eqref{MAp}) by
\begin{equation*}
 \mathsf{MA}_{\mathbf{p}}^{\circ} \colonequals \{f\in \mathsf{MA}_{\mathbf{p}} \mid  k_{2i-1,p_{2i-1}} \neq 0 \textnormal{ for all } i \text{ satisfying } s \leq i \leq t \text{ and } p_{2i-1} \text{ finite}\},
\end{equation*}  
and an open subspace $\M_{\mathbf{p}}^{\circ}$ of $ \M_{\mathbf{p}}$ (defined in \eqref{Mp}) by 
\begin{equation*}
    \M_{\mathbf{p}}^{\circ} \colonequals \{k\in \M_{\mathbf{p}} \mid k_{2i-1,p_{2i-1}} \neq 0 \textnormal{ for all } i  \text{ satisfying } s \leq i \leq t \text{ and } p_{2i-1} \text{ finite} \}.
\end{equation*}

\end{enumerate}
\end{notation}
We can, and will, consider $\mathsf{MA}_{\mathbf{p}}^{\circ}$ as a family of monomialised Type $A$ potentials over $ \M_{\mathbf{p}}^{\circ}$.

\begin{prop}\label{lemma:gv_obs}
With notation as in \textnormal{\ref{notation:gv_tuple}}, the set of isomorphism classes of contraction algebras associated to $\mathsf{CA}_{\mathbf{q}}$ coincides with the set of isomorphism classes of Jacobi algebras of $\mathsf{MA}_{\mathbf{p}}^{\circ}$.   
\end{prop}

\begin{proof}
Let $\uppi \in \mathsf{CA}_{\mathbf{q}}$. By \ref{514}, there exists a monomialised Type $A$ potential $f$ on $Q_n$ such that $\Jac(f) \xrightarrow[]{\sim}  \Lambda_{\mathrm{con}}(\uppi)$. We claim that $f \in \mathsf{MA}_{\mathbf{p}}^{\circ}$. 
To see this, we first fix $i$ with $s \leq i \leq t$. Since $\Jac(f) \xrightarrow[]{\sim}  \Lambda_{\mathrm{con}}(\uppi)$, \ref{lemma:Nst} gives $N_{ii}(f)=N_{ii}(\uppi)$. As $\uppi \in \mathsf{CA}_{\mathbf{q}}$, we have $N_{ii}(\uppi)=q_i$, and hence $N_{ii}(f)=q_i$. By \ref{stra}, this implies
\begin{enumerate}
\item If $q_i= \infty$, then $k_{2i-1,j}=0$ in $f$ for any $j$.
\item If $q_i \leq \infty$, then $k_{2i-1,q_i+1} \neq 0$ and $k_{2i-1,j}=0$ in $f$ for any $j \leq q_i$.
\end{enumerate}
In either case, since $p_{2i-1}=q_i+1$ by \ref{notation:gv_tuple}, it follows that $f \in  \mathsf{MA}_{\mathbf{p}}^{\circ}$.

Conversely, let $f \in \mathsf{MA}_{\mathbf{p}}^{\circ}$. By \ref{11} there exists a crepant resolution $\uppi$ of a $cA_n$ singularity such that $\Lambda_{\mathrm{con}}(\uppi)\xrightarrow[]{\sim}\Jac(f)$.
We claim that $\uppi \in \mathsf{CA}_{\mathbf{q}}$.
To see this, we first fix $i$ with $s \leq i \leq t$. Since $\Lambda_{\mathrm{con}}(\uppi)\xrightarrow[]{\sim} \Jac(f)$,  \ref{lemma:Nst} gives $N_{ii}(\uppi)=N_{ii}(f)$. Since $f \in \mathsf{MA}_{\mathbf{p}}^{\circ}$, \ref{stra} implies $N_{ii}(f)=p_{2i-1}-1=q_i$, and so $N_{ii}(\uppi)=q_i$. Thus $\uppi \in \mathsf{CA}_{\mathbf{q}}$.

Finally, by \ref{12}, the set of isomorphism classes of contraction algebras associated to crepant resolutions of $cA_n$ singularities coincides with the set of isomorphism classes of Jacobi algebras of monomialised Type $A$ potentials on $Q_n$, and the result follows.
\end{proof}

The following corollary identifies the generalised GV tuples realised by
$\mathsf{CA}_{\mathbf{q}}$ with those arising from $\mathsf{MA}_{\mathbf{p}}^{\circ}$,
which were described explicitly in \ref{58} and \ref{stra}.

\begin{cor}\label{cor:gv_equal}
The set of generalised $\GV$ tuples of $\mathsf{CA}_{\mathbf{q}}$ coincides with the set of generalised $\GV$ tuples of $\mathsf{MA}_{\mathbf{p}}^{\circ}$.     
\end{cor}
\begin{proof}
This follows immediately from \ref{lemma:gv_obs} and \ref{lemma:Nst}.
\end{proof}

Combining \ref{cor:gv_equal} with \ref{58}, we obtain obstructions and constructions for the tuples that can arise as generalised GV tuples of crepant resolutions of $cA_n$ singularities.

\begin{theorem}\label{thm:obs}
Fix integers $s,t$ with $1 \leq s \leq t \leq n$, and a tuple
$(q_s,\dots,q_t)\in \N_{\infty}^{\,t-s+1}$.
With notation as in \textnormal{\ref{notation:gv_tuple}}, the following holds.
\begin{enumerate}
\item For every $\uppi \in \mathsf{CA}_{\mathbf{q}}$, necessarily
$N_{st}(\uppi) \ge \mathbf{q}_{\min}$.
Moreover, there exists $\uppi \in \mathsf{CA}_{\mathbf{q}}$ such that
$N_{st}(\uppi)=\mathbf{q}_{\min}$.
\item Assume that $\mathbf{q}_{\min}$ is finite.
Then the equality $N_{st}(\uppi)=\mathbf{q}_{\min}$ holds for all
$\uppi\in \mathsf{CA}_{\mathbf{q}}$ if and only if
$\#\{\, i \mid q_i=\mathbf{q}_{\min}\,\}=1$.
\end{enumerate}
\end{theorem}

\begin{proof}
By \ref{cor:gv_equal}, it suffices to prove the statement for the generalised GV invariants of $\mathsf{MA}_{\mathbf{p}}^{\circ}$.
Recall from \ref{notation:gv_tuple} that $\mathsf{MA}_{\mathbf{p}}^{\circ}\subseteq \mathsf{MA}_{\mathbf{p}}$, $\M_{\mathbf{p}}^{\circ}\subseteq \M_{\mathbf{p}}$, and that $\mathbf{p}=(p_1,p_2,\dots,p_{2n-1})$ is defined by
$p_{2i-1}=q_i+1$ for $s\leq i\leq t$ and $p_i=2$ otherwise.
Consequently,
\[
d_{2s-1,2t-1}(\mathbf{p})
=\min(p_{2s-1},p_{2s+1},\dots,p_{2t-1})
=\min(q_s+1,q_{s+1}+1,\dots,q_t+1)
=\mathbf{q}_{\min}+1.
\]
For the remainder of the proof, we fix this notation and record several
facts that will be used repeatedly.

\begin{notation}\label{notation:obs}
Assume for the moment that $\mathbf{q}_{\min}$ is finite. We introduce the following notation and collect some immediate consequences.
\begin{enumerate}[label=(\alph*)]
\item Set
\[
\mathbf{I}\colonequals \{\, i \mid q_i \textnormal{ is finite for } s\leq i\leq t \,\}
     =\{\, i \mid p_{2i-1} \textnormal{ is finite for } s\leq i\leq t \,\}.
\]
Since $\mathbf{q}_{\min}$ is finite and $\mathbf{q}_{\min}=\min\{q_i\}$, we have $\mathbf{I}\neq \emptyset$. \label{notation:obs 1}

\item By \ref{notation:gv_tuple},
\[
\M_{\p}\setminus \M_{\p}^{\circ}
=\Bigl\{\, k\in \M_{\p}\ \Big|\ \prod_{i\in \mathbf{I}} k_{2i-1,p_{2i-1}}=0 \Bigr\}. \label{notation:obs 2}
\]

\item By \ref{58}, there exists a filtration
$\M_{\p}=M_1 \supsetneq M_2 \supsetneq M_3 \supsetneq \cdots$
such that $N_{st}(M_1\setminus M_2)=d_{2s-1,2t-1}(\p)-1=\mathbf{q}_{\min}$ and $N_{st}(M_2)>\mathbf{q}_{\min}$.
Moreover,
\[
M_2=\Bigl\{\, k\in \M_{\p}\ \Big|\ \det A^d_{2s-1,2t-1}(f(k))=0 \Bigr\},
\qquad \textnormal{where } d=d_{2s-1,2t-1}(\p). \label{notation:obs 3}
\]  
\end{enumerate} 
\end{notation}

\begin{notation}\label{notation:obs2}
Since $\M_{\p}$ is an infinite-dimensional vector space, it is convenient to isolate the finitely many coordinates that actually occur in the determinant condition in \ref{notation:obs}\ref{notation:obs 3}.
Accordingly, with notation as in \ref{notation:obs}, we introduce the
finite-dimensional linear subspaces
$\mathsf{N}_{\p}$, $\mathsf{N}_{\p}^{\circ}$, and $N_2$ of $\M_{\p}$ that will
be used throughout the proof.
\begin{enumerate}[label=(\alph*)]
\item Write $\kappa_{\p}$ for the finite tuple of variables
$\uk_{2s-1,p_{2s-1}}, \uk_{2s,p_{2s}}, \dots, \uk_{2t-1,p_{2t-1}}$.

\item Let $\mathsf{N}_{\p}\subseteq \M_{\p}$ be the linear subspace spanned by
the coordinate vectors corresponding to $\kappa_{\p}$, and let $V\subseteq \M_{\p}$
be the complementary subspace spanned by the remaining coordinate vectors
(i.e.\ those in $\uk_{\p}$ but not in $\kappa_{\p}$).
Then $\mathsf{N}_{\p}$ is finite-dimensional and
\[
\M_{\p}=\mathsf{N}_{\p}\oplus V.
\]
\label{notation:obs2 2}

\item In analogy with $\M_{\p}^{\circ}\subseteq \M_{\p}$ in \ref{notation:gv_tuple},
define the Zariski open subset $\mathsf{N}_{\p}^{\circ}\subseteq \mathsf{N}_{\p}$ by
\[
\mathsf{N}_{\p}^{\circ}\colonequals
\{k\in \mathsf{N}_{\p}\mid k_{2i-1,p_{2i-1}}\neq 0 \textnormal{ for all } i\in \mathbf{I}\}.
\]
Equivalently,
\[
\mathsf{N}_{\p}\setminus \mathsf{N}_{\p}^{\circ}
=\Bigl\{k\in \mathsf{N}_{\p}\ \Big|\ \prod_{i\in \mathbf{I}} k_{2i-1,p_{2i-1}}=0\Bigr\}.
\]
\label{notation:obs2 3}

\item In parallel with $M_2\subseteq \M_{\p}$ in \ref{notation:obs}\ref{notation:obs 3},
define the closed subspace $N_2\subseteq \mathsf{N}_{\p}$ by
\[
N_2 \colonequals
\Bigl\{k\in \mathsf{N}_{\p}\ \Big|\ \det A^d_{2s-1,2t-1}(f(k))=0\Bigr\},
\quad \textnormal{where } d=d_{2s-1,2t-1}(\mathbf{p}).
\]
\label{notation:obs2 4}

\item By definition \ref{Aij}, the matrix $A^d_{2s-1,2t-1}(\uk_{\p})$ involves
only the variables
\[
\uk_{2s-1,d},\, \uk_{2s,d},\, \dots,\, \uk_{2t-1,d}.
\]
In particular, when $d=d_{2s-1,2t-1}(\p)=\min(p_{2s-1},p_{2s+1},\dots,p_{2t-1})$,
all variables appearing in $A^d_{2s-1,2t-1}(\uk_{\p})$ lie in $\kappa_{\p}$, and hence
\[
A^d_{2s-1,2t-1}(\uk_{\p})=A^d_{2s-1,2t-1}(\kappa_{\p}).
\]
\label{notation:obs2 5}

\item Let $\varphi\colon \M_{\p}\twoheadrightarrow \mathsf{N}_{\p}$ be the
projection with kernel $V$.
Since $\M_{\p}^{\circ}$ and $\mathsf{N}_{\p}^{\circ}$ are cut out by the same
non-vanishing conditions on the coordinates indexed by $\mathbf{I}$, we have
$\varphi(\M_{\p}^{\circ})=\mathsf{N}_{\p}^{\circ}$, and therefore
\[
\M_{\p}^{\circ}=\mathsf{N}_{\p}^{\circ}\oplus V.
\]
Similarly, by \ref{notation:obs2}\ref{notation:obs2 5} the equation
$\det A^d_{2s-1,2t-1}(f(k))=0$ depends only on the coordinates in $\kappa_{\p}$,
so $\varphi(M_2)=N_2$, and hence
\[
M_2=N_2\oplus V.
\]
\label{notation:obs2 6}
\end{enumerate}
\end{notation}

\noindent
(1) If $\mathbf{q}_{\min}=\infty$, then $d_{2s-1,2t-1}(\mathbf{p})=\infty$, and hence $N_{st}(\M_{\p})=\infty$ by \ref{58}.
Since $\M_{\p}^{\circ}$ is nonempty and $\M_{\p}^{\circ}\subseteq \M_{\p}$, we can choose $f\in \mathsf{MA}_{\p}^{\circ}$, and then
$N_{st}(f)=\infty=\mathbf{q}_{\min}$.

Otherwise, assume that $\mathbf{q}_{\min}<\infty$. In this case, \ref{notation:obs} shows that $N_{st}(\M_{\p})\geq \mathbf{q}_{\min}$.
As $\M_{\p}^{\circ}\subseteq \M_{\p}$, it follows that $N_{st}(f)\geq \mathbf{q}_{\min}$ for every $f\in \mathsf{MA}_{\p}^{\circ}$.
This proves the lower bound in (1).

For the existence statement in (1), we claim that there exists $f\in \mathsf{MA}_{\p}^{\circ}$ such that $N_{st}(f)=\mathbf{q}_{\min}$.
By \ref{notation:obs}\ref{notation:obs 3}, we have $N_{st}(\M_{\p}\setminus M_2)=\mathbf{q}_{\min}$ and $N_{st}(M_2)>\mathbf{q}_{\min}$,
and so it suffices to show that
\[
(\M_{\p}\setminus M_2)\cap \M_{\p}^{\circ}\neq \emptyset.
\]
Using the decompositions in \ref{notation:obs2}\ref{notation:obs2 2} and \ref{notation:obs2}\ref{notation:obs2 6},
namely $\M_{\p}=\mathsf{N}_{\p}\oplus V$, $\M_{\p}^{\circ}=\mathsf{N}_{\p}^{\circ}\oplus V$ and $M_2=N_2\oplus V$,
this is equivalent to
\[
(\mathsf{N}_{\p}\setminus N_2)\cap \mathsf{N}_{\p}^{\circ}\neq \emptyset.
\]

By \ref{notation:obs2}\ref{notation:obs2 4}, $N_2$ is the zero locus of a nonzero element in $\C \lal \kappa_{\p} \ral$.
Hence $\mathsf{N}_{\p}\setminus N_2$ is a Zariski open in the finite-dimensional space $\mathsf{N}_{\p}$.
Likewise, $\mathsf{N}_{\p}^{\circ}$ is Zariski open in $\mathsf{N}_{\p}$ by \ref{notation:obs2}\ref{notation:obs2 3}.
Therefore their intersection is nonempty, as required.

\noindent
(2) For the remainder of the proof, assume that $\mathbf{q}_{\min}$ is finite.

($\Leftarrow$) Suppose that $\#\{\, i \mid q_i=\mathbf{q}_{\min}\,\}=1$.
We show that $N_{st}(f)=\mathbf{q}_{\min}$ for all $f\in \mathsf{MA}_{\p}^{\circ}$.
By \ref{notation:obs}, this is equivalent to proving that
\[
\M_{\p}^{\circ}\cap M_2=\emptyset,
\qquad \textnormal{equivalently,}\qquad
M_2\subseteq \M_{\p}\setminus \M_{\p}^{\circ}.
\]

Let $m$ be the unique index such that $q_m=\mathbf{q}_{\min}$, and set
$d\colonequals d_{2s-1,2t-1}(\p)$.
Since $p_{2i-1}=q_i+1$ for $s\leq i\leq t$ (see \ref{notation:gv_tuple}), the element $p_{2m-1}$ is the unique minimum of
$\{p_{2s-1},p_{2s+1},\dots,p_{2t-1}\}$.
Hence $d=p_{2m-1}$, and in particular $p_{2i-1}>d$ for all $i\neq m$ with $s\leq i\leq t$.
Therefore, by \ref{rmk: MAp}\eqref{rmk: MAp 4}, for $s\leq i\leq t$ we have:
\begin{itemize}
\item if $i=m$, then $p_{2i-1}=d$ and $\upvarepsilon_{2i-1,d}(\upkappa_{\p})=d\upkappa_{2i-1,d}$;
\item if $i\neq m$, then $p_{2i-1}>d$ and $\upvarepsilon_{2i-1,d}(\upkappa_{\p})=0$ on $\M_{\p}$.
\end{itemize}

If $d > 2$, then by \ref{lemma:matrix_2}\eqref{lemma:matrix_2 3} we obtain
\begin{align*}
 \det A_{2s-1,2t-1}^{d}(\uk_\p)& = (-1)^{t-s}\bigl(\upvarepsilon_{2s-1,d}(\uk_\p)+(-1)^d\upvarepsilon_{2s+1,d}(\uk_\p) + \dots 
 +(-1)^{(t-s)d}\upvarepsilon_{2t-1,d}(\uk_\p)\bigr)\\
 &= (-1)^{t-s}(-1)^{(m-s)d}\upvarepsilon_{2m-1,d}(\uk_\p)\\
& = (-1)^{t-s+(m-s)d}d\upkappa_{2m-1,d}.
\end{align*}
Hence, by \ref{notation:obs}\ref{notation:obs 3},
$M_2=\{\, k\in \M_{\p} \mid k_{2m-1,d}=0 \,\}$.
Since $q_m=\mathbf{q}_{\min}$ is finite, we have $m\in \mathbf{I}$ by \ref{notation:obs}\ref{notation:obs 1}.
Together with \ref{notation:obs}\ref{notation:obs 2} and $d=p_{2m-1}$, this yields
\[
\M_{\p}\setminus \M_{\p}^{\circ}
=\Bigl\{\, k\in \M_{\p}\ \Big|\ k_{2m-1,d}\prod_{i\in \mathbf{I}\setminus\{m\}} k_{2i-1,p_{2i-1}}=0 \Bigr\}.
\]
Hence $M_2\subseteq \M_{\p}\setminus \M_{\p}^{\circ}$.

Otherwise, $d=2$, and then by \ref{lemma:matrix_2}\eqref{lemma:matrix_2 2} we have 
\begin{align*}
\det A_{2s-1,2t-1}^{2}(\uk_\p)&=(-1)^{t-s}(\upvarepsilon_{2s-1,2}(\uk_\p)+\upvarepsilon_{2s+1,2}(\uk_\p) + \dots +\upvarepsilon_{2t-1,2}(\uk_\p))
 +\upepsilon(\uk_\p)\\
& =(-1)^{t-s}2\upkappa_{2m-1,2}+\upepsilon(\uk_\p),
\end{align*}
where $\upepsilon \in E_{2s-1,2t-1}$ and $E_{2s-1,2t-1}$ is the ideal generated by all the quadratic terms of $\upvarepsilon_{2s-1,2},\upvarepsilon_{2s+1,2}, \dots, \upvarepsilon_{2t-1,2}$ except $\upvarepsilon_{2s-1,2}^2,\upvarepsilon_{2s+1,2}^2, \dots, \upvarepsilon_{2t-1,2}^2$ (see \ref{Eij}).
Together with $\upvarepsilon_{2m-1,2}(\uk_\p)$ is the only nonzero element among $\upvarepsilon_{2s-1,2}(\uk_\p),\upvarepsilon_{2s+1,2}(\uk_\p), \dots, \upvarepsilon_{2t-1,2}(\uk_\p)$, it follows that $E_{2s-1,2t-1}(\uk_\p)=\{ 0\}$, and hence $\upepsilon(\uk_\p)=0$.  Thus $\det A_{2s-1,2t-1}^{2}(\uk_\p)=(-1)^{t-s}2\upkappa_{2m-1,2}$.
So by \ref{notation:obs}\ref{notation:obs 3}, $M_2=\{k \in \M_{\mathbf{p}} \mid k_{2m-1,d}=0 \} $. Similarly, $M_2 \subseteq \M_{\p}\setminus \M_{\p}^{\circ}$.

($\Rightarrow$) Conversely, suppose that $\#\{\, i \mid q_i=\mathbf{q}_{\min}\,\}>1$.
We show that there exists $f\in \mathsf{MA}_{\p}^{\circ}$ such that $N_{st}(f)>\mathbf{q}_{\min}$.
By \ref{notation:obs}\ref{notation:obs 3}, it suffices to prove that
\[
\M_{\p}^{\circ}\cap M_2\neq \emptyset,
\qquad \textnormal{equivalently,}\qquad
M_2\not\subseteq \M_{\p}\setminus \M_{\p}^{\circ}.
\]
Using $\M_{\p}=\mathsf{N}_{\p}\oplus V$, $\M_{\p}^{\circ}=\mathsf{N}_{\p}^{\circ}\oplus V$ and $M_2=N_2\oplus V$ from
\ref{notation:obs2}\ref{notation:obs2 2} and \ref{notation:obs2}\ref{notation:obs2 6},
it is enough to show that $N_2\not\subseteq \mathsf{N}_{\p}\setminus \mathsf{N}_{\p}^{\circ}$.
Set $d \colonequals d_{2s-1,2t-1}(\mathbf{p})$ and 
\[
I \colonequals \{i \mid q_i =\q_{\min} \text{ for }s\leq i \leq t\} =\{i \mid p_{2i-1} =d = \min(p_{2s-1}, \dots, p_{2t-1} ) \text{ for } s\leq i \leq t\}.
\]
By assumption we have $\#\{i  \mid q_i=\mathbf{q}_{\min}\}>1$, then the number of elements $|I| >1$.
By \ref{rmk: MAp}\eqref{rmk: MAp 4}, for $s \leq i \leq t$ we have:
\begin{itemize}
\item If $i\in I$, then $p_{2i-1}=d$, and so $ \upvarepsilon_{2i-1,d}(\upkappa_{\mathbf{p}})=d\uk_{2i-1,d}$.
\item If $i \notin I$, then $p_{2i-1}>d$, and so $\upvarepsilon_{2i-1,d}(\upkappa_{\mathbf{p}})=0$ on $\M_\p$.
\end{itemize}

If $d > 2$, then
\begin{align*}
\det A_{2s-1,2t-1}^{d}(\kappa_\p) &  \stackrel{\scriptstyle \ref{notation:obs2}\ref{notation:obs2 5}}{=} \det A_{2s-1,2t-1}^{d}(\uk_\p) \\
& \stackrel{\scriptstyle \ref{lemma:matrix_2}\eqref{lemma:matrix_2 3}}{=} (-1)^{t-s}\bigl(\upvarepsilon_{2s-1,d}(\uk_\p)+(-1)^d\upvarepsilon_{2s+1,d}(\uk_\p) + \dots 
 +(-1)^{(t-s)d}\upvarepsilon_{2t-1,d}(\uk_\p)\bigr)\\
&= (-1)^{t-s} \bigl(\sum_{i\in I} (-1)^{(i-s)d}\upvarepsilon_{2i-1,d}(\uk_\p) \bigr)  \\
&= (-1)^{t-s-sd}d\sum_{i\in I} (-1)^{id} \upkappa_{2i-1,d}.
\end{align*}

So by \ref{notation:obs2}\ref{notation:obs2 4}, $N_2=\{k\in \n_{\mathbf{p}} \mid \sum_{i\in I} (-1)^{id} k_{2i-1,d}=0 \}$. We next prove that $N_2 \not\subseteq \n_{\p}\setminus \n_{\p}^{\circ}$ by contradiction. 
Recall that $\n_{\p} \setminus \n_{\p}^{\circ} =  \{k \in \n_{\p} \mid \prod_{i \in \mathbf{I}}k_{2i-1,p_{2i-1}}=0\}$ in \ref{notation:obs2}\ref{notation:obs2 3}.
Thus if $N_2 \subseteq \n_{\p}\setminus \n_{\p}^{\circ}$, then 
\[
( \prod_{i \in \mathbf{I} }\upkappa_{2i-1,p_{2i-1}}) \subseteq (\sum_{i\in I} (-1)^{id} \upkappa_{2i-1,d})
\]
in $\C \lal \kappa_\p \ral$, and so there exists $\kappa' \in \C \lal \kappa_\p \ral$ such that 
\begin{equation}\label{6A}
    \prod_{i \in \mathbf{I} }\upkappa_{2i-1,p_{2i-1}} = \kappa'  (\sum_{i\in I} (-1)^{id} \upkappa_{2i-1,d}).
\end{equation}

Since $\C \lal \kappa_\p \ral$ is a power series ring in finitely many variables, it is a unique factorisation domain.
Together with \eqref{6A} and $|I|>1$, this would give two genuinely different factorisations of the same element in $\C \lal \kappa_\p \ral$, a contradiction.

Otherwise, $d=2$. In this case,
\begin{align*}
\det A_{2s-1,2t-1}^{2}(\kappa_\p) &  \stackrel{\scriptstyle \ref{notation:obs2}\ref{notation:obs2 5}}{=} \det A_{2s-1,2t-1}^{2}(\uk_\p) \\
& \stackrel{\scriptstyle \ref{lemma:matrix_2}\eqref{lemma:matrix_2 2}}{=}(-1)^{t-s}\bigl(\upvarepsilon_{2s-1,2}(\uk_\p)+\upvarepsilon_{2s+1,2}(\uk_\p) + \dots +\upvarepsilon_{2t-1,2}(\uk_\p)\bigr)
 +\upepsilon(\uk_\p)\\
& = (-1)^{t-s}2 \sum_{i\in I}\upkappa_{2i-1,2}+\upepsilon(\uk_\p),
\end{align*}
where $\upepsilon \in E_{2s-1,2t-1}$ and $E_{2s-1,2t-1}$ is the ideal generated by certain quadratic terms among $\upvarepsilon_{2s-1,2},\upvarepsilon_{2s+1,2}, \dots, \upvarepsilon_{2t-1,2}$. 
Therefore, by \ref{notation:obs2}\ref{notation:obs2 4},
\[
N_2=\bigl\{k\in \n_{\p}\ \big|\ (-1)^{t-s}2\sum_{i\in I}k_{2i-1,2}+\upepsilon(k)=0\bigr\}.
\]
Arguing as in the case $d>2$, one shows by contradiction that
$N_2\not\subseteq \n_{\p}\setminus \n_{\p}^{\circ}$.
\end{proof}

\begin{example}\label{eg obs}
Let $\uppi$ be a crepant resolution of a $cA_3$ singularity with exceptional curves
$\Curve_1$, $\Curve_2$ and $\Curve_3$. Suppose that
\[
(N_{11}(\uppi), N_{22}(\uppi), N_{33}(\uppi)) = (q_1, q_2, q_3),
\quad \text{with } q_1 < q_2 < q_3.
\]

Using the notation of \ref{thm:obs}, set $s=1$, $t=2$ and
$\mathbf{q}=\{(\Curve_1,q_1),(\Curve_2,q_2)\}$. Since
$N_{11}(\uppi)=q_1$ and $N_{22}(\uppi)=q_2$, we have
$\uppi \in \mathsf{CA}_{\mathbf{q}}$.
As $q_1<q_2$, it follows that $\mathbf{q}_{\min}=q_1<\infty$ and
$\#\{i \mid q_i=\mathbf{q}_{\min}\}=1$.
Hence, by \ref{thm:obs}(2), necessarily
$N_{12}(\uppi)=q_1$.

Similarly, by taking $s=2$, $t=3$ and
$\mathbf{q}=\{(\Curve_2,q_2),(\Curve_3,q_3)\}$, we obtain
$N_{23}(\uppi)=q_2$.
Finally, setting $s=1$, $t=3$ and
$\mathbf{q}=\{(\Curve_1,q_1),(\Curve_2,q_2),(\Curve_3,q_3)\}$,
we conclude that $N_{13}(\uppi)=q_1$.
\end{example}

\subsection{Obstructions from Iterated Flops}
Iterating flops yields further obstructions and constructions for the tuples that can arise as generalised GV invariants of crepant resolutions of $cA_n$ singularities.

\begin{notation}\label{notation:gv_tuple3}
Recall $\mathbf{r}$ and $\uppi^{\mathbf{r}}$ from \ref{notation:cAn}, and $|F_{\mathbf{r}}|$ from \ref{notation:NW}. There is a linear isomorphism
\[
|F_{\mathbf{r}}| \colon A_1(\uppi) \rightarrow  A_1(\uppi^{\mathbf{r}}),
\]
such that $\GV_{\upbeta}(\uppi)= \GV_{|F_{\mathbf{r}}|(\upbeta)}(\uppi^{\mathbf{r}})$ for any $\upbeta \in A_1(\uppi)$. By \ref{37}, $N_{\upbeta}(\uppi)= N_{|F_{\mathbf{r}}|(\upbeta)}(\uppi^{\mathbf{r}})$.
Varying $\mathbf{r}$ over all possible flops gives the following set,
\begin{equation*}
  \mathcal{F} \colonequals \bigcup_{i = 1}^{\infty} \{|F_{\mathbf{r}}| \ \mid  \ \mathbf{r}= (r_1,r_2,\dots , r_i) \text{ where each } 1\leq r_j \leq n\}.
\end{equation*}
Given any $F \in \mathcal{F}$ and $\mathbf{q}=\{(\upbeta_1,q_1), (\upbeta_2,q_2), \dots, (\upbeta_k,q_k) \}$ as in \ref{notation:gv_tuple2}, set 
\begin{equation*}
    F(\mathbf{q}) \colonequals \{(F(\upbeta_1),q_1), (F(\upbeta_2),q_2), \dots, (F(\upbeta_k),q_k) \}.
\end{equation*}    
\end{notation}

The flexibility of the transformations $F\in\mathcal{F}$, together with
\ref{thm:obs}, yields further obstructions and constructions for the
generalised GV tuples of crepant resolutions of $cA_n$ singularities, as follows.

\begin{cor}\label{cor:obs}
For any integers $s$ and $t$ with $1 \leq s \leq t \leq n$, any tuple $(q_s, \dots, q_t)\in \N_{\infty}^{t-s+1}$, and any $F\in\mathcal{F}$, with notation as in \textnormal{\ref{notation:gv_tuple}} and \textnormal{\ref{notation:gv_tuple3}}, the following statements hold.   
\begin{enumerate}
\item For any $\uppi \in \mathsf{CA}_{F(\mathbf{q})}$ necessarily $N_{F(st)}(\uppi) \geq \mathbf{q}_{\min}$. Moreover, there exists $\uppi \in \mathsf{CA}_{F(\mathbf{q})}$ such that $N_{F(st)}(\uppi)=\mathbf{q}_{\min}$.
\item Assume that $\mathbf{q}_{\min}$ is finite. Then the equality $N_{F(st)}(\uppi) = \mathbf{q}_{\min}$ holds for all $\uppi \in \mathsf{CA}_{F(\mathbf{q})}$ if and only if
$\#\{i  \mid q_i=\mathbf{q}_{\min}\}=1$.
\end{enumerate} 
\end{cor}

\begin{proof}
By the definition of $\mathcal{F}$ in \ref{notation:gv_tuple3}, there exists a finite sequence $\mathbf{r}=(r_1,r_2,\dots , r_j)$ such that $F=|F_{\mathbf{r}}|$. Let $\overline{\mathbf{r}}\colonequals (r_j,\dots,r_1)$ denote the reversed sequence.

From \ref{notation:gv_tuple3} we have
$N_{\upbeta}(\uppi)=N_{F(\upbeta)}(\uppi^{\mathbf{r}})$ for all curve classes $\upbeta$.
In particular, if $\uppi\in \mathsf{CA}_{\mathbf{q}}$, then $\uppi^{\mathbf{r}}\in \mathsf{CA}_{F(\mathbf{q})}$.
Likewise, since $N_{\upbeta}(\uppi^{\overline{\mathbf{r}}})=N_{F(\upbeta)}(\uppi)$, if
$\uppi\in \mathsf{CA}_{F(\mathbf{q})}$ then $\uppi^{\overline{\mathbf{r}}}\in \mathsf{CA}_{\mathbf{q}}$.

\smallskip
\noindent
(1) Let $\uppi \in \mathsf{CA}_{F(\mathbf{q})}$. Then $\uppi^{\overline{\mathbf{r}}}\in \mathsf{CA}_{\mathbf{q}}$, so by
\ref{thm:obs}(1) we have $N_{st}(\uppi^{\overline{\mathbf{r}}})\ge \mathbf{q}_{\min}$.
Since $N_{F(st)}(\uppi)=N_{st}(\uppi^{\overline{\mathbf{r}}})$, it follows that
$N_{F(st)}(\uppi)\ge \mathbf{q}_{\min}$.

For the existence statement, \ref{thm:obs}(1) provides
$\uppi_1\in \mathsf{CA}_{\mathbf{q}}$ such that $N_{st}(\uppi_1)=\mathbf{q}_{\min}$.
Then $\uppi_1^{\mathbf{r}}\in \mathsf{CA}_{F(\mathbf{q})}$ and
\[
N_{F(st)}(\uppi_1^{\mathbf{r}})=N_{st}(\uppi_1)=\mathbf{q}_{\min}.
\]

\smallskip
\noindent
(2) Let $\uppi \in \mathsf{CA}_{F(\mathbf{q})}$. Then $\uppi^{\overline{\mathbf{r}}}\in \mathsf{CA}_{\mathbf{q}}$ and
$N_{F(st)}(\uppi)=N_{st}(\uppi^{\overline{\mathbf{r}}})$.
If $\mathbf{q}_{\min}$ is finite and $\#\{\, i \mid q_i=\mathbf{q}_{\min}\,\}=1$, then
\ref{thm:obs}(2) gives $N_{st}(\uppi^{\overline{\mathbf{r}}})=\mathbf{q}_{\min}$, hence
$N_{F(st)}(\uppi)=\mathbf{q}_{\min}$.

Conversely, suppose that $N_{F(st)}(\uppi)=\mathbf{q}_{\min}$ holds for all
$\uppi \in \mathsf{CA}_{F(\mathbf{q})}$.
If $\uppi\in \mathsf{CA}_{\mathbf{q}}$, then $\uppi^{\mathbf{r}}\in \mathsf{CA}_{F(\mathbf{q})}$ and
$N_{st}(\uppi)=N_{F(st)}(\uppi^{\mathbf{r}})=\mathbf{q}_{\min}$.
Therefore $N_{st}(\uppi)=\mathbf{q}_{\min}$ for all $\uppi\in \mathsf{CA}_{\mathbf{q}}$, and
\ref{thm:obs}(2) implies $\#\{\, i \mid q_i=\mathbf{q}_{\min}\,\}=1$.
\end{proof}

\subsection{Examples}

We illustrate the preceding results with concrete examples.
Theorem~\ref{thm:obs} shows that the invariant $N_{st}$ is constrained by the values of the tuple $(N_{ss},\dots,N_{tt})$, while \ref{cor:obs} shows that, after applying iterated flops, the invariant $N_{F(st)}$ is constrained by the tuple $(N_{F(ss)},\dots,N_{F(tt)})$.

\begin{example}\label{example:obs1}
Consider the case $n=2$, with $s=1$ and $t=2$, and apply different elements
$F\in\mathcal{F}$ in \ref{cor:obs}.
The following table shows that the invariant $N_{\upbeta}$ is constrained
by the pair $(N_{\upbeta_1},N_{\upbeta_2})$, where
$(\upbeta_1,\upbeta_2,\upbeta)
\colonequals \bigl(F(11),F(22),F(12)\bigr)$.

\[
\begin{tabular}{ p{0.8cm}p{1cm}p{0.5cm} }
\hline
$F$ & $\upbeta_1,\upbeta_2$ & $\upbeta$ \\
\hline
$\operatorname{id}$ & $11,22$ & $12$ \\
$|F_{(1)}|$ & $11,12$ & $22$ \\
$|F_{(2)}|$ & $12,22$ & $11$ \\
\hline
\end{tabular}
\]

As a concrete consequence, for any crepant resolution $\uppi$ of a $cA_2$ singularity, we obtain the following constraints.
For brevity, we write $N_{\upbeta}$ for $N_{\upbeta}(\uppi)$.
\begin{enumerate}
\item
From the first row, $N_{12}\geq \min(N_{11},N_{22})$.
Moreover, if $N_{11}\neq N_{22}$, then necessarily
$N_{12}=\min(N_{11},N_{22})$.
\item
From the second row, $N_{22}\geq \min(N_{11},N_{12})$.
Moreover, if $N_{11}\neq N_{12}$, then necessarily
$N_{22}=\min(N_{11},N_{12})$.
\item
From the third row, $N_{11}\geq \min(N_{12},N_{22})$.
Moreover, if $N_{12}\neq N_{22}$, then necessarily
$N_{11}=\min(N_{12},N_{22})$.
\end{enumerate}
\end{example}

\begin{example}
Consider $n=3$, with $s=1$ and $t=3$, and apply different elements
$F\in\mathcal{F}$ in \ref{cor:obs}.
The following table shows that the invariant $N_{\upbeta}$ is constrained
by the triple $(N_{\upbeta_1},N_{\upbeta_2},N_{\upbeta_3})$, where
$(\upbeta_1,\upbeta_2,\upbeta_3,\upbeta)
\colonequals \bigl(F(11),F(22),F(33),F(13)\bigr)$.

\[
\begin{tabular}{ p{1cm}p{1.6cm}p{0.5cm} }
\hline
$F$ & $\upbeta_1,\upbeta_2,\upbeta_3$ & $\upbeta$ \\
\hline
$\operatorname{id}$ & $11,22,33$ & $13$ \\
$|F_{(1)}|$ & $11,12,33$ & $23$ \\
$|F_{(2)}|$ & $12,22,23$ & $13$ \\
$|F_{(3)}|$ & $11,23,33$ & $12$ \\
$|F_{(1,2)}|$ & $12,11,23$ & $33$ \\
$|F_{(2,1)}|$ & $22,12,13$ & $23$ \\
$|F_{(2,3)}|$ & $13,23,22$ & $12$ \\
$|F_{(3,2)}|$ & $12,33,23$ & $11$ \\
$|F_{(1,3)}|$ & $11,13,33$ & $22$ \\
\hline
\end{tabular}
\]
\end{example}

Combining \ref{thm:obs}, \ref{cor:obs},
and \ref{example:obs1}, we obtain a complete description of the
generalised GV tuples arising from crepant resolutions of $cA_2$ singularities.

\begin{cor}\label{prop:gv_cA2}
The generalised GV tuples of crepant resolutions of $cA_2$ singularities fall into one of the following two forms:
\[
\begin{tikzpicture}[bend angle=30, looseness=1]
\node (a) at (0,0) {$N_{11}$};
\node (b) at (1.2,0) {$N_{22}$};
\node (c) at (0.6,-0.6) {$N_{12}$};
\node (d) at (2,-0.4) {$=$};
\node (e) at (2.7,0) {$p$};
\node (f) at (3.5,0) {$q$};
\node (g) at (3.1,-0.6) {$\min(p,q)$};
\node (h) at (4.3,-0.4) {or};
\node (e) at (5,0) {$p$};
\node (f) at (5.8,0) {$p$};
\node (g) at (5.4,-0.6) {$r$};
\end{tikzpicture}
\]
where $p$, $q$, $r \in \N_{\infty}$ with $p \neq q$ and $ r \geq p$. Moreover, all possible such $p,q,r$ arise.
\end{cor}
\begin{proof}
Fix $p,\,q \in \N_{\infty}$. By \ref{thm:obs}(1), for any crepant resolution of a $cA_2$ singularity $\uppi$ with $N_{11}(\uppi)=p$ and $N_{22}(\uppi)=q$, we necessarily have $N_{12}(\uppi) \geq \min(p,q)$, and there exists such a $\uppi$ with $N_{12}(\uppi)=  \min(p,q)$.
If moreover $p \neq q$, then  \ref{thm:obs}(2) forces $N_{12}(\uppi)= \min(p,q)$, giving the first form.

It remains to consider the case $p=q$.
By example~\ref{example:obs1}, the invariant $N_{22}$ is constrained by the pair $(N_{11},N_{12})$. Given any $r\ge p$, \ref{cor:obs}(1) yields a  crepant resolution $\uppi$ of a $cA_2$ singularity such that
$N_{11}(\uppi)=p$, $N_{12}(\uppi)=r$, and $N_{22}(\uppi)=\min(p,r)=p$.
This gives the second form and shows that all such triples $(p,p,r)$ arise.
\end{proof}

\end{document}